\documentclass[12pt,letter]{amsart}

\usepackage{style/package, style/command, style/bibstyle}

\begin{document}

\title[General harmonic measures]{General harmonic measures for\\ distance-expanding dynamical systems}
\author{Zhiqiang~Li \and Ruicen~Qiu}

\thanks{Z.~Li and R.~Qiu were partially supported by NSFC Nos.~12101017, 12090010, 12090015, and BJNSF No.~1214021.}

\address{Zhiqiang~Li, School of Mathematical Sciences \& Beijing International Center for Mathematical Research, Peking University, Beijing 100871, CHINA}
\email{zli@math.pku.edu.cn}
\address{Ruicen~Qiu, School of Mathematical Sciences, Peking University, Beijing 100871, CHINA}
\email{ruicen121@stu.pku.edu.cn}



\begin{abstract}
	Partially motivated by the study of I.~Binder, N.~Makarov, and S.~Smirnov \cite{BMS03} on dimension spectra of polynomial Cantor sets, we initiate the investigation on some general harmonic measures, inspired by Sullivan's dictionary, for distance-expanding dynamical systems. Let $f\:X\to X$ be an open distance-expanding map on a compact metric space $(X,\rho)$. A Gromov hyperbolic tile graph $\Gamma$ associated to the dynamical system $(X,f)$ is constructed following the ideas from M.~Bonk, D.~Meyer \cite{BM17} and P.~Haïssinsky, K.~M.~Pilgrim \cite{HP09}. We consider a class of one-sided random walks associated with $(X,f)$ on $\Gamma$. They induce a Martin boundary of the tile graph, which may be different from the hyperbolic boundary. We show that the Martin boundary of such a random walk admits a surjection to $X$. We provide a class of examples to show that the surjection may not be a homeomorphism. Such random walks also induce measures on $X$ called harmonic measures. When $\rho$ is a visual metric, we establish an equality between the fractal dimension of the harmonic measure and the asymptotic quantities of the random walk.
\end{abstract}

\maketitle

\tableofcontents

\section{Introduction} \label{sct:Introduction}
	The use of measure-theoretic methods is standard in the study of dynamical systems nowadays. Understanding various properties of the invariant measures associated with the dynamical systems has been a central part of the ergodic theory.

	During the 20th century, complex dynamics and geometric group theory flourished with the collaborative efforts of numerous mathematicians. In the 1970s, D.~P.~Sullivan identified striking parallels between the theory of the action of Kleinian groups and the iterations of rational maps on the Riemann sphere. Notably, the analogy between expanding rational maps and hyperbolic Kleinian group actions, including fractal properties, is involved in Sullivan's dictionary. As topological generalizations of these concepts, expanding dynamics and Gromov hyperbolic groups also share similar properties. A hyperbolic graph associated with a dynamical system is constructed, whose Gromov boundary is identified with the phase space of the dynamical system, as an analog of the Cayley graph for a hyperbolic group. The construction of the hyperbolic graph for Thurston maps is due to M.~Bonk and D.~Meyer in \cite{BM17}, where they call it the tile graph. Independently, P.~Haïssinsky and K.~M.~Pilgrim also developed a similar graph in the context of coarse expanding conformal dynamics in \cite{HP09}. We investigate different boundaries of similar graphs associated with distance-expanding dynamical systems.

	In the 1960s, H.~Kasten studied the spectrum of random walks on countable groups and found that the property called \emph{amenability} can be determined by the random walk, which reflects features of the group itself. From then on, random walk methods gradually became powerful tools in characterizing various groups. One of the key observations is that the Markov operator is a discrete analog of the Laplacian operator. Hence, many concepts, including the harmonic functions, the Poisson boundary, and the Harnack principle, can be studied in the theory of random walks on countable groups. Mathematicians, including H.~Furstenberg, E.~B.~Dynkin, F.~Ledrappier, V.~A.~Kaimanovich, etc., have studied various properties of these potential-theoretic objects. For hyperbolic groups, the problem of finding the Poisson boundary explicitly is easier. There is another topological boundary associated with the random walk, which is called the Martin boundary. The Poisson boundary can be identified with the Martin boundary equipped with the representing measure of the constant function, that is, the harmonic measure on it; see \cite{Ka96} for details. A.~Ancona proved Ancona's inequality \cite[Theorem~5]{Anc87}, which can be used to show that, under some mild conditions, the Martin boundary of the random walk is isomorphic to the hyperbolic boundary of the group. Hence, for hyperbolic groups, the study of the Poisson boundaries is reduced to the study of the harmonic measures on the hyperbolic boundary. Several asymptotic constants, including the asymptotic entropy $h$ and the asymptotic drift $l$, are crucial in understanding the Poisson boundary of random walks on hyperbolic groups. One of the significant results is the so-called fundamental inequality $h\le lv$ proved by Y.~Guivarc'h \cite{Gu80} about the entropy, the drift of the random walk, and the logarithmic growth rate $v$ of the group.  In the 2000s, S.~Blach\`{e}re, P.~Haïssinsky, and P.~Mathieu in \cite{BHM08} and \cite{BHM11} introduced Green metrics for random walks on hyperbolic groups, which are quasi-isometric to the word metric. Using Green metrics, they showed that the harmonic measure is equivalent to the conformal measure of the group if and only if the equality in the fundamental inequality holds. They also provided a formula $\dims \nu=\frac{h}{al}$ to calculate the dimension of harmonic measures on the hyperbolic boundary endowed with an $a$-visual metric.

	In the theory of complex analysis and polynomial dynamics, there is also a notion of harmonic measure on the Julia set $J_f$ of a polynomial $f$ with degree $d\ge 2$. It is defined, for example, as the hitting distribution of a Brownian motion in the complement $\Omega_f=\C\setminus K_f$ of the filled Julia set. It was first introduced by H.~Brolin in \cite{Br65} where the equidistribution of preimages of the harmonic measure was proved. M.~Yu.~Lyubich proved the equidistribution of preimages of the measure of maximal entropy in \cite{Ly83}, which shows that the harmonic measure and the measure of maximal entropy coincide. This measure is also called the Brolin--Lyubich measure. The Hausdorff dimension of the Brolin--Lyubich measure is proved to be equal to $1$ by A.~Manning in \cite{Ma84}. For a polynomial with disconnected Julia set, the tile graph $\Gamma$ is a tree so that the harmonic measure on $\Omega_F$ can be reconstructed as the harmonic measure of a random walk on $\Gamma$; see \cite{Em18}. For a deep investigation of the universal dimension spectra of polynomial Canter sets, see the work of I.~Binder, N.~Makarov, and S.~Smirnov \cite{BMS03}.

	This article mainly concentrates on discrete random walks on the hyperbolic graphs associated with expanding dynamical systems, aiming to define a class of harmonic measures for dynamical systems as an analog to the harmonic measures for infinite groups. Moreover, some basic properties, including a dimension formula for the newly defined harmonic measure, are proved in this article. We will also show that the Martin boundary maps surjectively to the phase space and the surjection may possibly not be a homeomorphism.

	More precisely, let $(X,\rho)$ be a compact metric space, and $f\:X\to X$ be an open transitive distance-expanding map on $X$; see Subsection~\ref{subsect:expandingDyn}. Associated with a Markov partition $\alpha=\{A_0,\dots, A_N\}$, we can define a hyperbolic graph $\Gamma$ called the \textit{tile graph}. In general, the vertex set of the tile graph consists of words $u=u_1u_2\dots u_n$ with characters $u_1,\dots,u_n\in\{0,\dots, N\}$, such that for each $i\in\{0,\dots, n-1\}$, $A_{u_{i+1}}\subseteq fA_{u_i}$. Each vertex $u=u_1u_2\dots u_n$ corresponds to a tile
	\begin{equation*}
		A_u \= A_{u_1}\cap f^{-1} A_{u_2}\cap\cdots\cap f^{-(n-1)}A_{u_n}.
	\end{equation*}
	Two vertices $u$, $v$ are connected by an edge if and only if their levels differ at most by $1$ and $A_u\cap A_v\neq\emptyset$. The empty word $o\=\emptyset$ corresponds to the largest tile $X$. 

	For a detailed construction of the tile graph, see Subsection~\ref{subsct:visualMetric}.

		The assumption on the uniform expansion of $f$ in Subsection~\ref{subsect:expandingDyn} implies that $f$ is a local homeomorphism. This property is usually used in this paper in the form of Lemma~\ref{lem:locIsom}, which indicates that the natural shift map $\sigma$ on the tile graph restricts to an isomorphism between subgraphs away from the root. 

	Let $\cM(\Gamma)$ be the space of probability measures on $\Gamma$. Given a map $P\:\Gamma\to \cM(\Gamma)$, the random walk $\{Z_n\}$ associated to $P$ with starting point $u$ is defined as follows. Put $Z_0=u$ and let $Z_{n+1}$ follow the law of distribution $P(Z_n)$ for all $n\in\Z_{\ge0}$, inductively. Then $\{Z_n\}$ is a random walk on the tile graph $\Gamma$. Such a map $P$ is called a transition probability on $\Gamma$. Since $\Gamma$ is countable, $P$ induces a map $\hP\colon \Gamma\times \Gamma\to [0,1]$ given by $(x,y)\mapsto P(x)(\{y\})$.

	In this article, we focus on a particular family of random walks on the tile graph $\Gamma$, which is related to the dynamics $(X,f)$. 

	We say that $P$ satisfies the \textit{Assumptions in Section~\ref{sct:Introduction}} if $P$ satisfies the following assumptions.

	\medskip
	\textbf{Assumptions:}
	\begin{enumerate}[\hspace{2em}(A)]
		\smallskip
		\item There is a constant $R>0$ such that for each $x \in \Gamma$, we have $\supp P(x)\subseteq B(x,R)$, where $B(x,R)$ denotes the ball in $\Gamma$ centered at $x$ with radius $R$. 
		\smallskip
		\item For all $x,y\in\Gamma$, if $\hP(x,y)>0$, then $\abs{y}>\abs{x}$. 
		\smallskip
		\item For all $x,y\in\Gamma$, if $\abs{y}=\abs{x}+1$ and $d(x,y)=1$, then $\hP(x,y)>0$. 
		\smallskip
		\item $P$ commutes with the shift $\sigma$, i.e.,
		\begin{equation*}
			P(\sigma u) = (\sigma_* P(u))\: A\mapsto \sum_{v\in A}\sum_{\sigma w=v}\hP(u,w)
		\end{equation*}
		except for $u=o$. 
	\end{enumerate}
	
	For the definition of the shift map $\sigma$, see Subsection~\ref{subsct:visualMetric}.

	Generally speaking, Assumption~(A) is a ``locality'' assumption for the transition probability. It ensures that each step of the random walk does not move too far away from its current position. This assumption is crucial in determining the Martin boundary of the tile graph just as determining the Martin boundary of hyperbolic groups in \cite{Anc87}. Assumption~(B) makes sure that the random walk always increases the level of a vertex. We make this assumption because by this we can establish the subadditivity of the logarithm of the Green function. Hence, the Green drift $l_G$ in Lemma~\ref{lem:deflG} is well-defined. Assumption~(C) implies that the random walk resembles a diffusion process. Without this assumption, the topology of the phase space would not affect the random walk, so the identification between the Martin boundary and the phase space becomes impossible. Assumption~(D) relates the random walk with the dynamical system $f$ so that it is possible to attach some ergodic properties to it. It is essential in the proof of Theorem~\ref{thm:mixing}.

	In the theory of random walks on hyperbolic graphs, under some mild assumptions of irreducibility and locality, the Martin boundary coincides with the hyperbolic boundary, see \cite[Section~27]{Woe00} and \cite[Theorem~3.1]{Kai97}. However, this result cannot be applied directly to our setting. Although the tile graph is hyperbolic, Assumption~(B) makes the random walk reducible. Hence, the Harnack inequality about the harmonic functions fails, and thus, Ancona's inequality fails as well. To deal with such an obstacle, we introduce and establish variants of the Harnack inequality and Ancona's inequality in Lemmas~\ref{lem:Harnack} and~\ref{lem:multiplicative}, respectively.

	The following theorems are the main results of this article. The first theorem relates points on the Martin boundary with points on the Gromov boundary.
	
	\begin{theorem} \label{thm:main1}
		Suppose that $f\:X\to X$ is an open transitive distance-expanding map on a compact metric space $(X,\rho)$, $\alpha$ is a sufficiently fine Markov partition, and $(\Gamma, P)$ is random walk on the tile graph $\Gamma=\Gamma(f, \alpha)$ with $P$ satisfying the Assumptions in Section~\ref{sct:Introduction}. Then there is a natural surjection $\Phi$ from the Martin boundary $\Mb\Gamma$ of $(\Gamma, P)$ to $X$. 
	\end{theorem}
	By \emph{natural surjection}, we mean that the identity map on $\Gamma$ extends continuously to a surjection $\Phi$ from the Martin boundary to the Gromov boundary. Note that the Gromov boundary is naturally homeomorphic to $X$ due to Theorem~\ref{thm:hyp}.

	One may ask whether the natural surjection in Theorem~\ref{thm:main1} is a homeomorphism or not. Unfortunately, the answer is no in general. We provide a family of counterexamples of random walks on the tile graph of the doubling map on the circle in Section~\ref{sct:MartinBoundary}. 

	\begin{theorem} \label{thm:main1counter}
		Under the notations and hypotheses in Theorem~\ref{thm:main1}, there is a choice of $(X,f)$, $\alpha$, and $\Gamma$, such that for some choices of $P$, the surjection $\Phi$ in Theorem~\ref{thm:main1} is a homeomorphism, while for other choices of $P$, $\Phi$ is not a homeomorphism.
	\end{theorem}

	In the study of the Martin boundary and the harmonic measures for a random walk, one of the crucial concepts is the \defn{Green function}
	\begin{equation} \label{eq:green0}
		G(x,y) \= \bE_x\biggl(\sum_{n=0}^{+\infty} \bOne_y(Z_n) \biggr) = \sum_{n=0}^{+\infty} \bP_x(Z_n=y) = \sum_{n=0}^{+\infty} \hP^{(n)}(x,y),
	\end{equation}
	for $x,y\in\Gamma$, where $\bE_x(f)$ denotes the expectation of a random variable $f$ of the random walk $\{Z_n\}$ with $Z_0=x$, $\bP_x(A)$ denotes the probability of an event $A$ of the random walk $\{Z_n\}$ with $Z_0=x$, and for each integer $n>1$,
	\begin{equation*}
		\hP^{(n)}(x,y) \= \sum_{x_1, \dots ,x_{n-1}\in\Gamma} \hP(x,x_1)\hP(x_1,x_2)\cdots \hP(x_{n-1},y)
	\end{equation*}
	is the probability of $Z_n=y$ when $Z_0=x$. Let 
	\begin{equation} \label{eq:F}
		F(x,y)\=\bP_x(\exists n\in\Z_{\ge 0}, Z_n=y)
    \end{equation}
    be the probability that the random walk started at $x$ ever hits $y$. If Assumption~(B) holds, then for each $x\in\Gamma$, $G(x,x)=1$. Since for each $x,y\in\Gamma$, $G(x,x)F(x,y)=G(x,y)$, the functions $F$ and $G$ coincide. Hence, except for Section~\ref{sct:MartinBoundary}, we use both notations for the same function.

	If $\{Z_n\}$ $\bP_o$-a.s.~converges to a point $Z_{\infty}$ in some boundary $\partial\Gamma$, then we can define the harmonic measure on $\partial\Gamma$ by
	\begin{equation*}
		\nu^{\partial\Gamma}(A) \= \bP_o(Z_\infty\in A), \quad \text{for each Borel measurable subset } A \subseteq \partial\Gamma.
	\end{equation*}
	It is the escape distribution of the random walk $\{Z_n\}$ to such a boundary. In particular, as long as the random walk is transient, $\{Z_n\}$ $\bP_o$-a.s.~converges to a point $Z_{\infty}\in\Mb\Gamma$ in the Martin boundary. Hence, on the Martin boundary, the harmonic measure $\nu^{\Mb\Gamma}$ can be defined. As an immediate consequence of Theorem~\ref{thm:main1}, the harmonic measure can also be defined on the phase space $X$ by $\nu \= \Phi_*\nu^{\Mb\Gamma}$.

	To quantitatively study the fractal property of a measure, one of the effective ways is to calculate its fractal dimensions. For a metric space $(X,\rho)$, we recall that the \defn{Hausdorff dimension} of $X$ is defined by
	\begin{equation*}
		\dim_H(X) \= \inf \Bigl\{\delta>0 : \lim_{r\to 0}\inf_{\{U_i\}_{i\in\Z_{>0}}} \sum_{i\in\Z_{>0}}(\diam U_i)^{\delta} =0 \Bigr\},
	\end{equation*}
	where the second infimum is taken over all countable covers $\{U_i\}$ of $X$ such that for each $i\in\Z_{>0}$, $\diam U_i<r$. The \defn{packing dimension} of $X$ is defined by
	\begin{equation*}
		\dim_P(X) \= \inf \Bigl\{\delta>0 : \inf_{\{A_i\}_{i\in\Z_{>0}}}\sum_{i\in\Z_{>0}}\Bigl(\lim_{r\to 0}\sup_{\{(x_{i,j},r_{i,j})\}_{j\in\Z_{>0}}} \sum_{j\in\Z_{>0}}(r_{i,j})^{\delta}\Bigr) =0 \Bigr\},
	\end{equation*}
	where the second infimum is taken over all covers $\{A_i\}$ of $X$, and the supremum is taken over all countable pairs $\{(x_{i,j},r_{i,j})\}_{j\in\Z_{>0}}$ with $x_{i,j}\in A_i$, $r_{i,j} \in (0,r]$, and $\rho(x_{i,j}, x_{i,k}) \geq r_{i,j} + r_{i,k}$ for each pair of distinct $j,k\in\Z_{>0}$.

	We say that a measure $\mu$ on $X$ has \defn{Hausdorff dimension} $\delta$ if
	\begin{equation*}
		\inf\{\dim_H(A):A\subseteq X, \mu(A)>0\} = \inf\{\dim_H(A):A\subseteq X, \mu(A)=1\} = \delta.
	\end{equation*}

	We say that a measure $\mu$ on $X$ has \defn{packing dimension} $\delta$ if
	\begin{equation*}
		\inf\{\dim_P(A):A\subseteq X, \mu(A)>0\} = \inf\{\dim_P(A):A\subseteq X, \mu(A)=1\} = \delta.
	\end{equation*}

	For a detailed introduction to fractal dimensions, we refer the reader to see, for example, \cite[Chapter~8]{PU10}.

	To compute the dimension of the harmonic measure $\nu$, we need several asymptotic quantities associated with the random walk. Let $l_G$ be the almost sure limit of $-n^{-1} \log(G(Z_0,Z_n))$, $l$ be the almost sure limit of $\abs{Z_n}/n$. We show in Section~\ref{sct:harmonicMeasure} that both of them exist, and we have the following theorem about the fractal dimension of the harmonic measure.
	
	\begin{theorem}\label{thm:main2}
		Under the notations and hypotheses in Theorem~\ref{thm:main1}, if $X$ is equipped with an $a$-visual metric $\rho$ for a sufficiently small constant $a>0$, then the packing dimension of the harmonic measure $\nu$ on $X$ is equal to $\frac{l_G}{a l}$.
	\end{theorem}

	Note that it is easy to see that the \emph{asymptotic entropy} $h$ is no less than the \emph{Green drift} $l_G$. More precisely,
	\begin{equation*}
		h\=\lim_{n\to+\infty} \bE \bigl( -\log \hP^{(n)}(o,Z_n) \bigr) \ge \lim_{n\to+\infty} \bE(-\log F(o,Z_n)) \eqqcolon l_G.
	\end{equation*}
	Since the Hausdorff dimension of $\nu$ is not greater than the packing dimension of $\nu$, we have the following corollary of Theorem~\ref{thm:main2}.
	\begin{cor}
		Under the notations and the hypotheses in Theorem~\ref{thm:main1}, if $X$ is equipped with an $a$-visual metric $\rho$ for a sufficiently small constant $a>0$, then the Hausdorff dimension of the harmonic measure is not greater than $\frac{h}{a l}$.
	\end{cor}

	The formula of the Hausdorff dimension of the harmonic measure of hyperbolic groups was established by S.~Blach\`{e}re, P.~Haïssinsky, and P.~Mathieu in \cite{BHM11}. This dimension formula is closely related to the dimension formula that under some assumption, the Hausdorff dimension of an ergodic measure $\nu$ on $X$ is equal to the entropy divided by the Lyapunov exponent, i.e., $\dims \nu = h_\nu(f)/\chi_\nu(f)$.

	Finally, we establish the quasi-invariance of the harmonic measure.
	
	\begin{theorem}\label{thm:main3}
		Under the notations and hypotheses in Theorem~\ref{thm:main1}, the harmonic measure $\nu$ on $X$ is quasi-invariant under $f$, more precisely, $\nu$ and $f_*\nu$ are absolutely continuous to each other with both of the two Radon--Nikodym derivatives bounded.
	\end{theorem}

	\smallskip
	
	We will now give a brief description of the structure of this paper.

	In Section~\ref{sct:Preliminaries}, we recall the background of uniformly expanding dynamical systems on a compact metric space and Markov partitions associated with it. We also recall some background for random walks on discrete infinite graphs, the Martin boundary, and the harmonic measures associated with it. Finally, we build a graph $\Gamma$ called the tile graph from a fixed Markov partition $\alpha$ on $X$. The random walk will take place on this graph. This graph is equipped with a shift map $\sigma$ induced by $f$, and the tile graph shares similar properties with the Cayley graph in the group theory.

	In Section~\ref{sct:randomWalk}, we study basic properties of $\sigma$-invariant random walks with the transition probability $P$ satisfying Assumptions in Section~\ref{sct:Introduction}. Some lemmas about the shadow will also be proved in this section.

	In Section~\ref{sct:MartinBoundary}, we give a proof of Theorem~\ref{thm:main1}. Then we provide a class of examples to show that the surjection in Theorem~\ref{thm:main1} may not be a homeomorphism, establishing Theorem~\ref{thm:main1counter}.

	In Section~\ref{sct:harmonicMeasure}, we study the ergodic properties of the random process and justify the definition of the asymptotic quantities in the dimension formula of the harmonic measure in Theorem~\ref{thm:main2}. Then based on the estimation of the Martin kernels, we give a proof of Theorem~\ref{thm:main2}.
	
	Finally, in Section~\ref{sct:quasi-invariance}, we study some basic dynamical properties of the harmonic measure and establish Theorem~\ref{thm:main3}.

	In the appendix, we provide proofs of Theorem~\ref{thm:hyp} and Propositions~\ref{prop:visualMetric} and ~\ref{prop:visualLocal}. They are properties of the tile graph and the visual metrics that are used in this article.

	\subsection*{Acknowledgments} 
     The authors thank  Mario~Bonk, Manfred~Denker, Wenyuan~Yang, Yiwei~Zhang, and Tianyi~Zheng for interesting discussions.

\section{Preliminaries} \label{sct:Preliminaries}
	In this section, we state our settings of distance-expanding dynamical systems $(X, f)$ and review the notion of Markov partitions. Then we review the construction of the Martin boundary and the harmonic measure.

In this article, if two functions $f$ and $g$ are positive, then $f\lesssim g$ means that there is a universal constant $C>0$ such that $f\le Cg$. We write $f\asymp g$ if both $f\lesssim g$ and $f\gtrsim g$ hold.

	\subsection{Uniformly expanding maps and Markov partitions} \label{subsect:expandingDyn}

		In this subsection, we review the definition and some known properties of an open, transitive, and distance-expanding map. Then we recall the notion of Markov partitions and symbolic dynamical systems induced by a distance-expanding map.

		Let $(X, \rho)$ be a compact metric space. We denote by $C(X)$ the space of continuous functions on $X$, and by $\cM(X)=C(X)^*$ the space of Borel probability measures on $X$. For a point $x\in X$ and subsets $A, B\subseteq X$, we denote
		\begin{align*}
			\rho(x,A) &\= \inf\{\rho(x,y):y\in A\} \quad\text{and} \\
			\rho(A,B) &\= \inf\{\rho(x,y):x\in A, y\in B\}.
		\end{align*}
		A ball of radius $r\in (0,+\infty)$ centered at $x\in X$ is denoted by
		\begin{equation*}
			B(x,r) = \{y\in X : \rho(x,y) < r\}.
		\end{equation*}
		The $r$-neighborhood of a subset $A\subseteq X$ is denoted by
		\begin{equation}\label{def:neighborhood}
			B(A, r) = \{x\in X: \rho(x,A)<r\}.
		\end{equation}

		We assume that a map $f\colon X\to X$ satisfies the following assumptions.

		\smallskip
		
		\textbf{Assumptions:}
		
		\begin{enumerate}[\hspace{2em}(i)]
			\smallskip
			\item $f\: X \to X$ is continuous on a compact metric space $(X,\rho)$ and is \defn{topologically transitive}, i.e., for each pair of nonempty open sets $U$ and $V$, there exists a number $n\in\Z_{>0}$ such that $f^nU\cap V\neq\emptyset$. \smallskip
			\item $f$ is \defn{open}, i.e., the image of an open set is open.
			\smallskip
			\item $f$ is \defn{distance-expanding}, i.e., there exist constants $\xi>0$ and $\lambda>1$ such that for each pair of points $x,y\in X$ with $\rho(x,y)\le \xi$, we have $\rho(fx,fy)\ge \lambda \rho(x,y)$. By Assumption~(ii), $f$ is a local homeomorphism. So, for the sake of convenience, we assume, moreover, that $f|_{B(x,\xi)}$ is a homeomorphism to its image.
		\end{enumerate}

		We also denote by $\cM(X,f)$ the subspace of $f$-invariant measures on $X$ in $\cM(X)$.

		In the sequel, we say that a dynamical system $(X,f)$ satisfies the \textit{Assumptions in Subsection~\ref{subsect:expandingDyn}} if $f$ satisfies~(i), (ii), and~(iii).

		Then we briefly review the notion of Markov partitions and subshifts of finite type. With the help of a Markov partition, we can obtain a subshift of finite type for an open distance-expanding map.

		Let $S$ be a finite nonempty set, and $M\:S\times S\to \{0,1\}$ be a matrix with entries being either $0$ or $1$. We denote the \defn{set of admissible sequences defined by $M$} by
		\begin{equation*}
			\Sigma_M^+ \= \bigl\{ \{x_i\}_{i\in\Z_{\ge 0}} : x_i \in S, \, M(x_i,x_{i+1})=1, \, \text{for each } i\in\Z_{\ge 0}\bigr\}.
		\end{equation*}
		The topology on $\Sigma_M^+$ is induced by the product topology, which is compact by the Tychonoff theorem since $S$ is a finite set.

		The \defn{left-shift operator $\sigma_M\colon\Sigma_M^+\to\Sigma_M^+$} is given by
		\begin{equation*}
			\sigma_M \bigl(\{x_i\}_{i\in\Z_{\ge 0}}\bigr) = \{x_{i+1}\}_{i\in\Z_{\ge 0}} \quad \text{for each }\{x_i\}_{i\in\Z_{\ge 0}} \in \Sigma_M^+.
		\end{equation*}

		The pair $(\Sigma_M^+, \sigma_M)$ is called the \defn{one-sided subshift of finite type} defined by $M$. The set $S$ is called the \defn{set of states}, and the matrix $M\: S\times S \to \{0,1\}$ is called the \defn{transition matrix}.
		
		Fixing a one-sided subshift of finite type $(\Sigma_M^+, \sigma_M)$, we denote by 
		\begin{equation*}
			[y_0, y_1, \dots,y_n]\=\bigl\{\{x_i\}_{i\in\Z_{\ge 0}}\in\Sigma_M^+:x_i=y_i, 0\le i\le n\bigr\}
		\end{equation*}
		the cylinders of the $(n+1)$-tuple $(y_{0},\dots,y_{n})\in M^{n+1}$ satisfying $M_{y_{i-1}y_i}=1$,  for each integer $1\le i\le n$.

		Let $X$, $Y$ be topological spaces, and $f\:X\to X$, $g\:Y\to Y$ be two continuous maps. We say that $(X,f)$ is \defn{topologically semi-conjugate} to $(Y,g)$ if there exists a continuous surjection $h: X\to Y$ such that $h\circ f=g\circ h$. If, furthermore, $h$ is a homeomorphism, then we say that $(X,f)$ is \defn{topologically conjugate} to $(Y,g)$.

		For distance-expanding dynamical systems, Markov partitions associate the original dynamical systems with symbolic dynamics. For a more detailed discussion about results on Markov partitions related to our context, see for example, \cite[Section 4.5]{PU10}.

		A finite cover $\alpha=\{A_1,\dots, A_N\}$ is called a \defn{Markov partition} if the following conditions hold:

		\begin{enumerate}[\hspace{2em}(a)]
			\smallskip
			\item $A_i = \overline{\inter A_i}$, for all $i\in\{1,\dots,N\}$;  
			\smallskip
			\item $\inter A_i \cap \inter A_j = \emptyset$ for all $i,j\in\{1,\dots,N\}$ with $i\neq j$;  
			\smallskip
			\item $f(\inter A_i) \cap \inter A_j \neq \emptyset$ implies $\inter A_j\subseteq f(\inter A_i)$ for all $i,j\in\{1,\dots,N\}$.
		\end{enumerate}

		We denote the mesh of the Markov partition $\alpha$ by $\mesh{\alpha} \= \max\{\diam A:A\in\alpha\}$.
		For open distance-expanding maps, there exist Markov partitions of arbitrarily small meshes (see for example, \cite[Theorem 4.5.2]{PU10}). Hence, we can attach to each of these maps a Markov partition with $\mesh\alpha < \xi$, where $\xi>0$ is the constant in the Assumptions in Subsection~\ref{subsect:expandingDyn}. Such a Markov partition $\alpha=\{A_1,\dots,A_d\}$ gives rise to a coding of $f\: X\to X$. More precisely, let $M$ be a $N\times N$ matrix with entries $0$ or $1$ depending on whether $f(\inter A_i)\cap \inter A_j$ is empty or not. Then there is a one-sided subshift of finite type $(\Sigma_M^+, \sigma_M)$ together with a topological semi-conjugation $\pi\:\Sigma_M^+\to X$ given by
		\begin{equation*}
		u=\{u_i\}_{i\in\Zn}\mapsto x\in A_u=\bigcap_{i=0}^{+\infty} f^{-i}A_{u_i}.
		\end{equation*}

	\subsection{Markov chains on graphs, Martin boundaries, and harmonic measures}

		In this subsection, we review some concepts related to discrete random walks on a countable graph.

		Recall that a \defn{graph} is a set $V$ for which an edge set $E$ has been specified, where an edge set $E$ on $V$ is a collection consisting of subsets of $V$ of cardinality~2, called \defn{edges}. Properly speaking, a graph is an ordered pair $(V, E)$ consisting of a set $V$ and an edge set $E$ on $V$, but we usually omit specific mention of $E$ if no confusion arises. Points in $V$ are also called \defn{vertices}.
		
		A subgraph $\Gamma'\subseteq\Gamma$ is a subset of the set $\Gamma$ equipped with the edge set $E|_{\Gamma'}\=\{\{u,v\}\in E:u,v\in V'\}$. Two graphs $(\Gamma_1, E_1)$ and $(\Gamma_2,E_2)$ are \defn{isomorphic} if and only if there is a bijection $f:\Gamma_1\to \Gamma_2$ such that $E_2=\{\{f(u),f(v)\}:\{u,v\}\in E_1\}$.

		Let $\Gamma$ be a graph with basepoint $o$ and $\Omega\=\Gamma^{\Z_{\ge 0}}$ be the sample space. A \defn{transition probability} on $\Gamma$ is a function $P\: \Gamma\to\cM(\Gamma)$. Recall that we denote $\hP(x,y)=P(x)(\{y\})$.
		A \defn{Markov chain} on $\Gamma$ is defined as a series of random variables \begin{equation*}
			Z_n\colon\Omega\to \Gamma,~n \in \Z_{\ge 0},
		\end{equation*} 
		with
		\begin{equation*}
			\bP(Z_{n+1}=u\,|\,\sigma(Z_0,\dots, Z_n)) = \hP(Z_n, u),
		\end{equation*}
		where $\sigma(Z_0,\dots, Z_n)$ denotes the $\sigma$-field generated by $Z_0,\dots, Z_n$. The random walk usually starts from the basepoint $o\in\Gamma$, in which case we define $Z_0 \= o$.
		
		By Kolmogorov's extension theorem, $\Omega$ admits a probability measure generated by the transition probability $P$ and the initial distribution $p$. If $p$ is equal to $\delta_u$, the Dirac measure at $u\in\Gamma$, then we denote the probability measure on $\Omega$ by $\bP_u$. In particular, for the case that $u=o$ is the base point of $\Gamma$, we write $\bP=\bP_o$ if we do not emphasize the choice of $u$. The sample space $(\Omega, \bP, \sigma(Z_0,Z_1,\dots))$ is then a probability measure space.

		Suppose that each state $w\in \Gamma$ in the Markov chain is \defn{transient}, i.e.,
		\begin{equation*}
			\bP_w(\min\{n\in\Z_{>0}:Z_n=w\}<+\infty)<1.
		\end{equation*}
		Then the Markov process ``escapes to infinity'' almost surely.
		
		To formalize the intuition, we define the Martin boundaries and harmonic measures below.

		We now recall the notion of the Martin boundary $\Mb\Gamma$. In general, it is a compactification of $\Gamma$ at infinite, on which Borel measures represent all $P$-harmonic functions on the graph. The harmonic measure is defined as the escape distribution on some boundary of a Markov chain from one point in the graph. One can refer to \cite[Section~24]{Woe00} for more details.

		To be precise, recall that the Green function $G\colon\Gamma\times \Gamma\to \R$ of a Markov chain $(\Gamma,P)$ is defined in (\ref{eq:green0}) as the expectation of the total number $N_v\coloneqq\sum_{n=0}^{+\infty} \mathbf{1}_{v}(Z_n)$ of visits to $v$ from $u$.
		Hence, we have a recursive formula of $G$ as follows:
		\begin{equation} \label{eq:green2}
			G(u, v) = \sum_{w\in \Gamma}\hP(u, w)G(w, v) + \bOne_v(u) = \sum_{w\in \Gamma}G(u,w)\hP(w, v) + \bOne_u(v).
		\end{equation}

		The \defn{$P$-Laplacian} $\Delta_P f$ of a function $f\colon\Gamma\to \R$ is defined by
		\begin{equation*}
			\Delta_P f(u) \= - f(u) + \sum_{v\in\Gamma}\hP(u,v)f(v) 
		\end{equation*}
		for each $u\in\Gamma$. A function $f$ is \defn{$P$-harmonic} if $\Delta_P f=0$. By induction, for each harmonic function $f$ and each $n\in\Z_{>0}$, we have
		\begin{equation}\label{eq:hfun}
			\sum_{v\in\Gamma}\hP^{(n)} (u,v)f(v) = f(u).
		\end{equation}
		Hence, by (\ref{eq:green2}), the $P$-Laplacian of the Green function is $\Delta_P G(\cdot, v)=\bOne_{v}$. If there is no confusion about the choice of $P$, we may say that $f$ is harmonic if $f$ is $P$-harmonic.

		Let $F(u,v) \= \bP_u(v\in\{Z_0,Z_1,\dots\})$ be the probability of visiting $v\in\Gamma$ from $u\in\Gamma$. Then it is straightforward to show that for all $u,v\in\Gamma$, $F(u,v)\le 1$ and $G(u,v) = F(u,v)G(v,v)$. This implies that $F$ is a normalization of the Green function $G$ with $F(u,u)=1$. Since each vertex $u\in\Gamma$ is transient, by the definition of $G$, we can calculate that \begin{equation*}
			G(u,u)=\frac{1}{1-\bP_u(\min\{n\in\Z_{>0}:Z_n=w\}<+\infty)}<+\infty.
		\end{equation*}
		Hence, $G(u,v) \le G(v,v)$ is finite for all $u,v\in\Gamma$.

		Now we formulate the definition of the Martin boundary of a Markov chain.
		First, we construct a function $K(\cdot,v)\colon\Gamma\to\R$ for each $v\in\Gamma$ taking value $1$ at $o$, called \emph{the Martin kernel}, by
		\begin{equation}\label{def:K}
			K(u,v) \= \frac{G(u,v)}{G(o, v)} = \frac{F(u,v)}{F(o, v)}, \quad u, v\in \Gamma.
		\end{equation} 
		\begin{definition}\label{def:martinBoundary}
			Let $\mathrm{Map}(\Gamma, \R)$ be the family of $\R$-valued functions on $\Gamma$, equipped with the topology of pointwise convergence. The Martin kernel $K$ defines an embedding $\hK\colon\Gamma\to \mathrm{Map}(\Gamma, \R)$ given by $u\mapsto K(\cdot,u)$. The \emph{Martin boundary} is defined as $\Mb\Gamma\coloneqq \overline{\hK(\Gamma)}\setminus \hK(\Gamma)$.
		\end{definition}
		
		Thus, we can extend $\hK$ to $\Mb\Gamma$ and denote by $K(\cdot,\alpha)$ the function associated to $\alpha\in\Mb\Gamma$. It is harmonic because the $P$-Laplacian $\Delta_P K(\cdot,v)$ eventually becomes zero at each point as $\abs{v}\to+\infty$.
		
		For a more detailed construction, we may provide metrics on $\Mb\Gamma$ as follows. We denote $C_u\=1/G(o, u)$ for each $u\in\Gamma$. Then
		\begin{equation}\label{eq:harmonicLocBound}
			\abs{K(u,v)} \le \frac{G(u,v)}{G(o, u)G(u,v)} = C_u
		\end{equation}
		is bounded as a function of $v\in\Gamma$.
		
		Arbitrarily choose weights $D=\{D_u\}_{u\in\Gamma}$ with $D_u>0$ and $\sum_{u\in \Gamma} D_u=1$. Then we can construct the metric $\rho_D$ on $\mathrm{Map}(\Gamma,\R)$ as follows:
		\begin{equation*}
			\rho_D(f,g) \coloneqq \sum_{w\in\Gamma} D_w\,\frac{\abs{f(w) - g(w) }}{C_w}.
		\end{equation*}

		By (\ref{eq:harmonicLocBound}), it is straightforward to show that $\rho_D$ takes values in the interval $[0,2]$ when defined in the range of $\hK:\Gamma\to \mathrm{Map}(\Gamma, \R)$.
		
		The pullback $\hK^*\rho_D$ of $\rho_D$ by $\hK$ defines a metric completion $\overline{\Gamma}$ of $\Gamma$. It is a known result in general topology that $\rho_D$ is compatible with the topology of pointwise convergence. Hence, $\Mb\Gamma=\overline{\hK(\Gamma)}\setminus \hK(\Gamma)$ is a metric realization of the Martin boundary.

		An important characterization of the Martin boundary is the following Martin representation formula. See, for example, \cite[Section~24]{Woe00}. That is, for every positive harmonic function $h$ on $\Gamma$, there is a positive Borel measure $\nu^h$ on $\Mb\Gamma$ such that
		\begin{equation}\label{eq:MartinRep}
			h(u) = \int_{\Mb\Gamma} K(u,\xi) \mathrm{d}\nu^h(\xi).
		\end{equation}

		Since the random walk $Z_n$ is transient, by \cite[Theorem~4]{Dy69}, the random walk $Z_n$ starting from each point $u\in\Gamma$ $\bP_u$-a.s.~converges to a point $Z_{\infty}\in \Mb\Gamma$. The \defn{harmonic measure $\nu^{\Mb\Gamma}_u$ seen from $u\in\Gamma$} on the Martin boundary is then defined as the distribution of $Z_\infty\in\Mb\Gamma$, i.e.,
		\begin{equation} \label{eq:defHarm}
			\nu^{\Mb\Gamma}_u(A) \= \bP_u(Z_{\infty}\in A), \quad \text{for each Borel subset } A\in\cB(\Mb\Gamma).
		\end{equation}
		In particular, for the basepoint $u=o$ of the graph, we denote $\nu^{\Mb\Gamma} \= \nu_o^{\Mb\Gamma}$ and call it the harmonic measure if we do not emphasize the choice of $u$.

		By (\ref{eq:MartinRep}), the harmonic measure is the representing measure of the constant function $\bOne$. Hence, by the change of the basepoint of the Green kernel, for all $u\in\Gamma$ and $\xi\in\Mb\Gamma$, we have 
		\begin{equation*}
			K(u,\xi) \mathrm{d}\nu^{\Mb\Gamma}(\xi) = \mathrm{d}\nu_u^{\Mb\Gamma}(\xi),
		\end{equation*}
		That is,
		\begin{equation}\label{eq:harmonicBaseChange}
			\frac{\mathrm{d}\nu_u^{\Mb\Gamma}}{\mathrm{d}\nu^{\Mb\Gamma}}(\xi) = K(u,\xi).
		\end{equation}
		For a detailed discussion, see for example, \cite{Ka96}.

		In this article, Assumption~(B) in Section~\ref{sct:Introduction} implies $G(u,u)=1$ for each $u\in\Gamma$. Hence, the functions $F$ and $G$ are identical. In the following context, we use both notations $F$ and $G$ for the same function.

	\subsection{Tile graphs and visual metrics} \label{subsct:visualMetric}
		In this subsection, we review the notion of visual metrics for dynamical systems. It was introduced in the study of expanding Thurston maps by M.~Bonk and D.~Meyer in \cite{BM17}. Independently, P.~Haïssinsky and K.~M.~Pilgrim also introduced a similar notion to the theory of coarse expanding conformal dynamics in \cite{HP09}.

		Let $(X,f)$ be a dynamical system that satisfies the Assumptions in Subsection~\ref{subsect:expandingDyn}. Consider a finite Markov partition $\alpha=\{A_0,\dots, A_N\}$ of $X$. We can define a graph $\Gamma=\Gamma(f,\alpha)$, which is called the \defn{tile graph}. We give an explicit definition of $\Gamma$ below. 

		Let $\Gamma$ consists of a $0$-word $\emptyset$ and all $n$-words $u=u_1\dots u_n$, $n\in\Z_{>0}$, with $u_1,\dots, u_n\in\{0,1,\dots, N\}$, such that for each $i\in\{1,\dots, n-1\}$, $A_{u_{i+1}}\subseteq f(A_{u_{i}})$. Equivalently, the vertex set $\Gamma$ can be defined as
		\begin{equation*}
			\begin{aligned}
				\Gamma \= \{\emptyset\} \cup  \{u=u_1\dots u_n :  n &\in\Z_{>0}, \, u_1,\dots,u_n\in\{0,1,\dots,N\}, \, \\
				&\forall i\in\{1,\dots, n-1\}, \, A_{u_{i+1}}\subseteq f(A_{u_{i}})\} .
			\end{aligned}
		\end{equation*}
		Each word $u=u_1\dots u_n\in\Gamma$ is associated with a subset $A_u$ of $X$ given by
		\begin{equation*}
			A_u \= A_{u_1}\cap f^{-1} A_{u_2}\cap \cdots \cap f^{-n}A_{u_n}.
		\end{equation*}
		In addition, we put $A_{\emptyset} \= X$. The edge set of $\Gamma$ is defined by
		\begin{equation*}
			E \= \{\{u,v\}:u,v\in\Gamma,\abs{\abs{u}-\abs{v}}\le 1,A_u\cap A_v\neq\emptyset\}.
		\end{equation*}
		Denote $\abs{u}$ as the length of a word $u$. We also call it the \defn{level} of $u$. The tile graph admits a \emph{basepoint} $o\=\emptyset$. It is the unique vertex of level $0$.

		The tile graph $(\Gamma,E)$ admits a \defn{shift map} $\sigma$ on it, which is defined by 
		\begin{equation}   \label{def:sigma}
			A_{\sigma u} = f(A_u)
		\end{equation}
		for each $u\in\Gamma$ with $\abs{u}\ge 2$, and $\sigma u\=o$ for each $u\in\Gamma$ with $\abs{u}\le 1$. 

		Let $\tau\:\Gamma\to\Gamma$ be defined by $\tau\: u_1\dots u_n\mapsto u_1\dots u_{n-1}$ and $\emptyset \mapsto \emptyset$.
		It is the unique map such that for every $u\in\Gamma$ with $u\neq\emptyset$, $A_u\subseteq A_{\tau u}$ and $\abs{\tau u}=\abs{u}-1$. 

		As a metric space, $\Gamma$ is equipped with a combinatorial distance $d$. We now review some results about the hyperbolicity of the tile graph as a metric space $(\Gamma, d)$ and the definition of visual metrics. One can refer to \cite{HP09} and \cite{BM17} for details.

		Consider $(\Gamma, d)$ as a metric space. The \emph{Gromov product} of $u,v\in\Gamma$ with respect to $w\in \Gamma$ is defined to be
		\begin{equation}\label{def:Gprod}
			\Gprod{u,v}{w} \=\frac{1}{2}(d(u,w) + d(w,v) - d(u,v)).
		\end{equation}
		Let $\delta\ge 0$ be a constant. A metric space $\Gamma$ is said to be \emph{Gromov hyperbolic} or \emph{hyperbolic} if 
		\begin{equation}\label{eq:4ptHyp}
			\Gprod{x,y}{w} \ge \min\{\Gprod{x,z}{w}, \Gprod{z,y}{w}\} - \delta,
		\end{equation}
		for all $x,y,z,w\in \Gamma$.

		Let $(\Gamma, d)$ be a hyperbolic metric space. Fix $x\in \Gamma$. A sequence $\{x_n\}_{n\in\Z_{>0}}$ in $X$ converges at infinity if $\{\Gprod{x_i,x_j}{x}\}\to+\infty$. The \defn{Gromov boundary} $\partial\Gamma$ is the set of sequences $\{x_n\}_{n\in\Z_{>0}}$ converging at infinity modulo the equivalence relation defined by: $\{x_n\}_{n\in\Z_{>0}}\sim\{y_n\}_{n\in\Z_{>0}}$ if $\{\Gprod{x_n,y_n}{x}\}\to+\infty$. The Gromov product can be extended to the Gromov boundary in such a way that (\ref{eq:4ptHyp}) holds for all $x,y,z\in\Gamma\cup \partial \Gamma$ and $w\in\Gamma$.

		Fix $a>0$ and $x\in\Gamma$. A metric $\rho$ on $\partial \Gamma$ is said to be an \defn{$a$-visual metric} if
		\begin{equation}
			\rho(\eta,\zeta) \asymp e^{-a\Gprod{\eta,\zeta}{x}},
		\end{equation}
		for all $\eta,\zeta\in \partial \Gamma$. Here, the notation $f\asymp g$ means that there exists a constant $C=C(\asymp)>1$ independent of $\eta,\zeta$ such that $C^{-1}f\le g\le Cf$.

		\begin{rem}
			In fact, these definitions are independent of the choice of $x\in\Gamma$. For sufficiently small $a>0$, there always exists an $a$-visual metric. If for some $\delta>0$, there exists $w\in\Gamma$ such that the inequality (\ref{eq:4ptHyp}) holds, then $\Gamma$ is $2\delta$-hyperbolic for all $x,y,z,w\in\Gamma$. Hence, to verify the hyperbolicity, it suffices to verify (\ref{eq:4ptHyp}) for $w=o$. See for example, \cite[Section III.H]{BH99} for details of Gromov hyperbolic metric spaces and visual metrics. 
		\end{rem}

		The following theorem shows the hyperbolicity of the tile graph $\Gamma$. Such a result was proved by M.~Bonk and D.~Meyer for expanding Thurston maps in \cite[Theorems~10.1 and 10.2]{BM17}, and for coarse expanding dynamical systems, it was proved by P.~Haïssinsky and K.~M.~Pilgrim in \cite[Theorem~3.2.1 and Proposition~3.3.9]{HP09}. For the convenience of the reader, we give a proof in Proposition~\ref{prop:hyp}.
		\begin{theorem}\label{thm:hyp}
			Let $(X,f)$ be a dynamical system satisfying the Assumptions in Subsection~\ref{subsect:expandingDyn} with a Markov partition $\alpha$ such that $\mesh\alpha<\xi$. Then the tile graph $\Gamma$ is Gromov hyperbolic, and the Gromov boundary $\partial \Gamma$ of $\Gamma$ is naturally homeomorphic to $X$.
		\end{theorem}

		By \emph{naturally homeomorphic} we mean that the homeomorphism $\Psi$ from the Gromov boundary of $\Gamma$ to $X$ satisfies the following property: for every sequence of vertices $\{u_n\}_{n\in\Z_{>0}}$ in $\Gamma$ converging to $\xi\in\partial\Gamma$ in the Gromov boundary of $\Gamma$, the corresponding sequence of subsets $\{A_{u_n}\}_{n\in\Z_{>0}}$ converges in the sense of Gromov--Hausdorff convergence to a singleton $\{\Psi(\xi)\}\subseteq X$.

		In the following context, we usually identify the Gromov boundary $\partial \Gamma$ of the tile graph $\Gamma$ with the phase space $X$ by the homeomorphism $\Psi$ in Theorem~\ref{thm:hyp}. Hence, we call a metric $\rho$ on $X$ an $a$-visual metric if $\Psi^*\rho$ is an $a$-visual metric on $\partial \Gamma$. The following proposition follows the idea of \cite[Lemma~8.11]{BM17} and \cite[Proposition~3.3.2]{HP09}. For the convenience of the reader, we give a proof of it in Corollary~\ref{cor:visualMetric}.

		\begin{prop}\label{prop:visualMetric}
			Let $(X,f)$, $\alpha$, $\Gamma$ satisfies the assumptions in Theorem~\ref{thm:hyp}. There exists a constant $a_0>0$ such that the following statement holds. Let $\rho$ be an $a$-visual metric on $X$ for some constant $0<a<a_0$.
			Then there is a constant $C_0 > 1$ such that, for each $u \in \Gamma$, there is a point $x \in A_u$ such that
			\begin{equation*}
				B_\rho\bigl(x, C_0^{-1} e^{-a\abs{u}}\bigr) \subseteq A_u \subseteq B_\rho\bigl(x, C_0 e^{-a\abs{u}}\bigr),
			\end{equation*}
			and for all $u,v\in \Gamma$,
			\begin{equation*}
				C_0^{-1} e^{-a\Gprod{u,v}{o}} \le \diam_{\rho} (A_u\cup A_v) \le C_0 e^{-a\Gprod{u,v}{o}}. 
			\end{equation*}
		\end{prop}

		The following proposition follows the idea of \cite[Proposition~3.2.3]{HP09}. For the convenience of the reader, we provide a proof of it in Corollary~\ref{cor:visual}.
		\begin{prop} \label{prop:visualLocal}
			Let $(X,f)$, $\alpha$, $\Gamma$ satisfies the assumptions in Theorem~\ref{thm:hyp}. There exists a constant $a_0>0$ such that the following statement holds. Let $\rho$ be an $a$-visual metric on $X$ for some $0<a<a_0$. Then there is a constant $\xi>0$ such that for all $x,y\in X$ and $n\in\Z_{>0}$ with $\rho  ( f^m x,f^m y )<\xi$ for each integer $0\le m<n$, we have
			\begin{equation*}
				\rho(f^n x, f^n y) \asymp e^{an}\rho(x,y).
			\end{equation*}
			Moreover, for each $x\in X$, $f|_{B_{\rho}(x,\xi)}$ is a homeomorphism to its image.
		\end{prop}

		In this article, we always assume that $X$ is equipped with an $a$-visual metric $\rho$ for some sufficiently small $a>0$. The purpose of introducing the concept of visual metrics is that we want the tile graph to be Gromov hyperbolic,
		and $\rho$ is related to the tile graph $\Gamma$.

\section{Random walks on tile graphs}\label{sct:randomWalk}

	In this and all the following sections, we assume that the dynamical system $(X,f)$ satisfies the Assumptions in Subsection~\ref{subsect:expandingDyn}. Let $\Gamma$ be the tile graph associated with $f$ and a fixed Markov partition $\alpha$ with $\mesh\alpha<\xi$ so that Theorem~\ref{thm:hyp} can be applied. The tile graph $\Gamma$ is equipped with maps $\sigma$ and $\tau$ defined in the beginning of Subsection~\ref{subsct:visualMetric}. By Proposition~\ref{prop:visual}, we can further assume that the metric $\rho$ on $X$ is an $a$-visual metric for some $0<a<a_0$, where $a_0$ is a constant such that Propositions~\ref{prop:visualMetric} and~\ref{prop:visualLocal} hold. We focus on some basic properties of the random walks on the tile graph $\Gamma$ under the Assumptions in Section~\ref{sct:Introduction}.

	We denote the shadow and the neighborhood of $u$ associated with the random walk by
	\begin{align}
		\mho(u) &\= \{v\in\Gamma : F(u,v)>0\} \quad \text{and} \label{def:mho}\\
		N(u) &\= \{v\in\Gamma : \mho(u)\cap \mho(v)\neq\emptyset, \abs{\abs{u}-\abs{v}}\le R \},\label{def:N}
	\end{align}
	respectively. Here the constant $R>0$ is taken from Assumption~(A) in Section~\ref{sct:Introduction}. For each subset $S\subseteq \Gamma$, we put
	\begin{equation} \label{def:A}
		A_S\= \bigcup_{u\in S} A_u \subseteq X.
	\end{equation}
	Recall that in Subsection~\ref{subsect:expandingDyn}, we denote the ball in $X$ centered at $x\in X$ with radius $r>0$ as $B(x,r)$. We also denote the $r$-neighborhood of $A\subseteq X$ as $B(A,r)$.

	We first show the openness of the shadow of each vertex on the boundary.
	
	\begin{lemma}\label{lem:open}
		For each $v\in \Gamma$, the subset $A_{\mho(v)}$ of $X$ is open.
	\end{lemma}
	\begin{proof}
		By (\ref{def:A}) and (\ref{def:mho}), for each $\xi\in A_{\mho(v)}$, there is a vertex $w\in\Gamma$ such that $\xi\in A_w$ and $F(v,w)>0$. Choose
		\begin{equation*}
			\epsilon \= \rho\Bigl( A_w, \bigcup_{\substack{u \in\Gamma,\abs{u}=\abs{w}+1\\ A_u\cap A_w=\emptyset}} A_u \Bigr)>0.
		\end{equation*}
		By Assumption~(C) in Section~\ref{sct:Introduction}, we have
		\begin{equation*}
			\{u\in\Gamma: \abs{u}=\abs{w}+1, A_u\cap A_w\neq\emptyset\} \subseteq \mho(w).
		\end{equation*}
		Hence, by the definition of $\epsilon$,
		\begin{equation*}
			B(A_w, \epsilon)\subseteq \bigcup_{\substack{u\in\Gamma, \abs{u}=\abs{w}+1\\ A_u\cap A_w\neq\emptyset}} A_u\subseteq A_{\mho(w)}.
		\end{equation*}
		It follows that $B(\xi,\epsilon) \subseteq A_{\mho(w)}\subseteq A_{\mho(v)}$. Since the choice of $\xi\in A_{\mho(v)}$ is arbitrary, we finish our proof of the openness of $A_{\mho(v)}$.
	\end{proof}

	Then we prove a lemma about the boundedness of shadows.
	\begin{lemma}\label{lem:shadow1}
		There is a constant $C_1>0$ such that for each $u\in\Gamma$,
		\begin{equation} \label{ieq:shadow1}
			A_{\mho(u)} \subseteq B\bigl(A_u, C_1 e^{-a\abs{u}}\bigr).
		\end{equation}
	\end{lemma}
	\begin{proof}
		Note that for each $v\in \mho(u)$, there is a sequence of vertices $v_0,\dots,v_n\in\Gamma$ with $v_0=u$ and $v_n=v$ such that for each $i\in\{1,\dots, n\}$, $\hP(v_{i-1},v_i)>0$. By Assumptions (A) and (B) in Section~\ref{sct:Introduction}, $d(v_i,v_{i-1})<R$, and $\abs{v_{i}}>\abs{v_{i-1}}$. Hence, $\Gprod{v_i, v_{i-1}}{o} \ge \abs{v_i}-R$. By Proposition~\ref{prop:visualMetric}, for each $v\in\mho(u)$,
		\begin{equation*}
			\diam(A_v\cup A_u) 
			\le \sum_{i=1}^n \diam(A_{v_i}\cup A_{v_{i-1}}) 
			 \le \sum_{i=1}^n e^{-a(\abs{v_i}-R)} 
			 \le  \sum_{i=1}^{+\infty} e^{-a(\abs{u}+i-R)}
			= \frac{e^{aR} e^{-a\abs{u}} }{1-e^{-a}}.
		\end{equation*}
		Hence, by setting $C_1\= e^{aR} ( 1-e^{-a})^{-1}$, we finish the proof of the lemma.
	\end{proof}

	\begin{cor}\label{cor:shadow1}
		There is a constant $C_6>0$ such that for each $u\in\Gamma$,
		\begin{equation}
			\diam A_{\mho(u)} < C_6 e^{-a\abs{u}}.
		\end{equation}
	\end{cor}
	\begin{proof}
		Let $C_0>0$ and $C_1>0$ be the constants from Proposition~\ref{prop:visualMetric} and Lemma~\ref{lem:shadow1}, respectively. By Proposition~\ref{prop:visualMetric}, $\diam A_u< C_0e^{-a\abs{u}}$. Hence, by Lemma~\ref{lem:shadow1} and Proposition~\ref{prop:visualMetric}, $\diam_{A_{\mho(u)}}<2C_1e^{-a\abs{u}}+\diam A_u<(2C_0+2C_1) e^{-a\abs{u}}$. So $C_6\= 2C_0+2C_1$ is what we want.
	\end{proof}

	Roughly speaking, the following lemma shows that $\sigma$ is a local similarity of the graph $\Gamma$ and the random walk is preserved locally by $\sigma$. This is the key property we need from the assumption of $f$ being a local homeomorphism.

	\begin{lemma}\label{lem:locIsom}
		There is a number $N_0\in\Z_{>0}$ such that for each $u\in\Gamma$ with $\abs{u}\ge N_0$,
		\begin{enumerate}[\hspace{2em}(i)]
			\smallskip
			\item $\sigma|_{\mho(u)} \colon \mho(u) \to \mho(\sigma u)$ is an isomorphism between subgraphs of $\Gamma$,
			\smallskip
			\item $\sigma|_{N(u)} \colon N(u) \to N(\sigma u)$ is an isomorphism between subgraphs of $\Gamma$. 
		\end{enumerate}
	\end{lemma}
	\begin{proof}
		Let $C_6>0$ and $\xi>0$ be the constants from Corollary~\ref{cor:shadow1} and Proposition~\ref{prop:visualLocal}. We choose $N_0\in\Z_{>0}$ such that
		\begin{equation*}
			10(R+1) C_6 e^{-a(N_0-R)}<\xi.
		\end{equation*}
		Then by Corollary~\ref{cor:shadow1}, for each $u\in\Gamma$ with $\abs{u}\ge N_0$, 
		\begin{equation}\label{eq:shadowLocality}
			\diam A_{\mho(u)} < C_6 e^{-aN_0}<\xi/4.
		\end{equation}
		
		By Proposition~\ref{prop:visualLocal}, $f|_{A_{\mho(u)}}$ is a homeomorphism to its image. It follows from the definition of the vertices and the edges of the tile graph that $\sigma|_{\mho(u)}$ is an isomorphism between subgraphs of $\Gamma$ since whether two subsets intersect is preserved by a homeomorphism. It follows from Assumption~(D) in Section~\ref{sct:Introduction} that the image of $\sigma|_{\mho(u)}$ is exactly $\mho(\sigma u)$.

		For the proof of statement (ii), we fix $u\in\Gamma$ such that $\abs{u}\ge N_0$. For each $v\in \Gamma$ with $d(u,v) < R$, there is a path $u=u_0,u_1,\dots, u_k=v$ in $\Gamma$ connecting $u$ and $v$ for some $k<R$. By Corollary~\ref{cor:shadow1}, since each $u_i$ satisfies $\abs{u_i}\ge N_0-R$, we get 
		\begin{equation*}
			\diam  \bigl( A_{\mho(u)}\cup A_{\mho(v)} \bigr)
			\le \sum_{i=0}^{k-1} \diam \bigl( A_{\mho(u_i)}\cup A_{\mho(u_{i+1})} \bigr) 
			\le \sum_{i=0}^{k-1} 2C_6e^{-a (N_0-R)} 
			<\xi/3.
		\end{equation*}
		It follows that if we put $U_R\=\bigcup_{v\in\Gamma, d(u,v)<R} A_{\mho(v)}$, then
		$\diam U_R < \xi$.
		By Proposition~\ref{prop:visualLocal}, $f|_{U_R}$ is a homeomorphism to its image. It is easy to show that for all $v,w\in \Gamma$, $\mho(w)\cap \mho(v)\neq\emptyset$ if and only if $A_\mho(w)\cap A_\mho(v)\neq\emptyset$. Hence, for each $v\in \Gamma$ with $d(u,v)<R$, $\mho(u)\cap \mho(v)\neq\emptyset$ if and only $\mho(\sigma u)\cap \mho(\sigma v)\neq\emptyset$. Therefore, it follows immediately that $\sigma|_{N(u)}\: N(u)\to N(\sigma u)$ is an isomorphism between subgraphs of $\Gamma$.
	\end{proof}

	Recall that under the Assumptions in Section~\ref{sct:Introduction}, the Green function $G$ is equal to the function $F$. The following lemma shows that the Green function is nearly multiplicative. It is our key estimate on the Green function.

\begin{lemma}\label{lem:multiplicative}
	Assume that $u,v,s\in\Gamma$, $\abs{v}\le \abs{u}$, and $w\in\mho(u)$. Then 
	\begin{equation*}
		F(v,s)F(s,w) 
		\le F(v,w) 
		\le \sum_{t\in N(u)} F(v,t)F(t,w)
		\le N_1 \sup_{t\in N(u)} \{F(v,t)F(t,w)\},
	\end{equation*}
	where $N_1 \= \sup_{u\in\Gamma} \#N(u)$ is finite.
\end{lemma}

\begin{proof}
	By Lemma~\ref{lem:locIsom}, the size of the set $N(u)$ is equal to $\#N(\sigma^m u)$ if $m \=\abs{u} - N_0 - R$ is positive. Thus, the following is finite:
	\begin{equation*} 
		N_1 = \sup_{u\in\Gamma} \#N(u) = \sup_{u\in\Gamma, \abs{u}\le N_0+R} \#N(u) <+\infty.
	\end{equation*}

	For each random trajectory $Z_0,\dots, Z_n$ from $Z_0=v\in\Gamma$ to $Z_n=w\in\Gamma$, by Assumption~(A) in Section~\ref{sct:Introduction}, $\abs{Z_i} < \abs{Z_{i+1}} \le \abs{Z_i}+R$. Hence, there exists $i\in\{0,\dots, n\}$ such that $\abs{\abs{Z_i}-\abs{u}}\le R$. Since $w\in\mho(u)$ and $w=Z_n\in\mho(Z_i)$, we have $Z_i\in N(u)$. It follows that
	\begin{align*}
		F(v,w)
		&= \bP_v(\exists n\in\Z_{>0} \text{ with } Z_n=w)\\
		&\le \sum_{t\in N(u)} \bP_v(\exists n,i\in\Z_{\ge0} \text{ with } i <n, Z_n=w, Z_i=t)\\
		&= \sum_{t\in N(u)} F(v,t)F(t,w) \\
		&\le N_1 \sup_{t\in N(u)} \{F(v,t)F(t,w)\}.
	\end{align*}
	
	For the first inequality,
	\begin{align*}
		F(v,w) 
		&= \bP_v(\exists n\in\Z_{>0} \text{ with } Z_n=w) \\
		&\ge \bP_v(\exists n,i\in\Z_{\ge0} \text{ with } i <n, Z_n=w, Z_i=s) 
		= F(v,s)F(s,w).\qedhere
	\end{align*}
\end{proof}

\section{Topology of the Martin boundary} \label{sct:MartinBoundary}
  In this section, we assume that the dynamical system $(X,f)$ satisfies the Assumptions in Subsection~\ref{subsect:expandingDyn}. Let $\Gamma$ be the tile graph associated with $f$ and a fixed Markov partition $\alpha$ with each $A\in\alpha$ connected so that Theorem~\ref{thm:hyp} can be applied. We focus on some basic properties of the random walks on the tile graph $\Gamma$ under the Assumptions in Section~\ref{sct:Introduction}. 

 In Subsection~\ref{subsct:ProofThmSurjection}, we show that under the Assumptions in Section~\ref{sct:Introduction}, the Martin boundary of $(\Gamma, P)$ admits a surjection to the Gromov boundary of $\Gamma$. Combining this with Theorem~\ref{thm:hyp}, we establish Theorem~\ref{thm:main1}. We then provide a family of examples in Subsection~\ref{subsct:non-injective} to illustrate that this surjection may not be injective.

\subsection{Proof of Theorem~\ref{thm:main1}}  \label{subsct:ProofThmSurjection}
	Theorem~\ref{thm:main1} follows immediately from Theorem~\ref{thm:hyp} and Theorem~\ref{thm:surj} below. 

	\begin{theorem}\label{thm:surj}
		Let $(X,f)$, $\alpha$, and $\Gamma$ satisfies the assumptions in Theorem~\ref{thm:hyp}. Let $P$ be a transition probability satisfying the Assumptions in Section~\ref{sct:Introduction}. Let $\partial_M \Gamma$ and $\partial \Gamma$ be the Martin boundary of $(\Gamma, P)$ and the Gromov boundary of $\Gamma$, respectively. Then the identity map on $\Gamma$ extends continuously to a surjection $\Phi\:\partial_M\Gamma\rightarrow \partial \Gamma$.
	\end{theorem}
	\begin{proof}
		In the proof of this theorem, we always identify the Gromov boundary $\partial \Gamma$ of $\Gamma$ and the phase space $X$ in the sense of Theorem~\ref{thm:hyp}.

		Fix an arbitrary point $\xi\in\partial_M \Gamma$. By the definition of the Martin boundary, $\xi$ is associated with a harmonic function $K(\cdot,\xi)$ on $\Gamma$. Assume that a sequence $\{x_n\}_{n\in\Z_{>0}}$ in $\Gamma$ converges to $\xi\in\partial_M \Gamma$, or equivalently, $K(\cdot,x_n)$ converges pointwise to $K(\cdot, \xi)$. We aim to define $\Phi(\xi)$ as the limit point of $x_n$ on $\partial\Gamma$.
		
		\smallskip 
		\emph{Claim.} The sequence $x_n$ converges to a point $\eta\in\partial\Gamma$ as $n$ tends to $+\infty$.
		\smallskip 

		Note that for each $x\in\Gamma\cup \partial_M\Gamma$, $K(o,x)=1$. Since $K(\cdot,\xi)$ is harmonic, by (\ref{eq:hfun}) and Assumption~(B) in Section~\ref{sct:Introduction}, for each $M\in\Z_{>0}$,
		\begin{equation*}
			K(o,\xi) = \sum_{x\in\Gamma,\abs{x}\ge M} \hP^{(M)}(o,x) K(x,\xi).
		\end{equation*}
		Hence, there is a vertex $u\in\Gamma$ with $\abs{u}\ge M$ such that
		\begin{equation}\label{eq:Kneq0}
			K(u,\xi)\neq 0.
		\end{equation}

		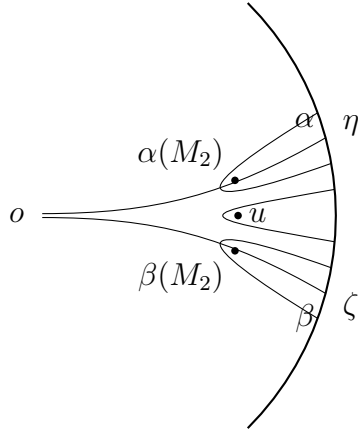
\begin{figure}[!ht]
			\begin{tikzpicture}
				\draw [thick,domain=-45:45] plot ({4*cos(\x)}, {4*sin(\x)});
				\draw[rotate=15] (0.1,0) to [bend right=15] (4,0) node[above left] {$\alpha$};
				\draw[rotate=-15] (0.1,0) to [bend left=15] (4,0) node[below left] {$\beta$};
				\draw (4.2,1.2) node {$\eta$};
				\draw (4.2,-1.2) node {$\zeta$};
				\draw (canvas polar cs:angle=20,radius=4cm) .. controls (2,.3) and (2,0) .. (canvas polar cs:angle=10,radius=4cm);
				\draw (canvas polar cs:angle=5,radius=4cm) .. controls (2,0) and (2,0) .. (canvas polar cs:angle=-5,radius=4cm);
				\draw (canvas polar cs:angle=-10,radius=4cm) .. controls (2,0) and (2,-.3) .. (canvas polar cs:angle=-20,radius=4cm);
				\fill (canvas polar cs:angle=10,radius=2.7cm) circle(.05) node[above left] {$\alpha(M_2)$};
				\fill (canvas polar cs:angle=-10,radius=2.7cm) circle(.05) node[below left] {$\beta(M_2)$};
				\fill (canvas polar cs:angle=0,radius=2.7cm) circle(.05) node[right] {$u$};
				\node (0,0) [left] {$o$};
				
			\end{tikzpicture}
			\caption{Shadows split near the Gromov boundary by hyperbolicity.}
			\label{fig:hyperbolicShadow}
		\end{figure}

		We prove the claim by contradiction and assume that there are subsequences $\{y_n\}$ and $\{z_n\}$ of $\{x_n\}$ converging to distinct points $\eta$ and $\zeta$ in the Gromov boundary $\partial\Gamma$ (which, by Theorem~\ref{thm:hyp}, is compact), respectively. For each $n\in\Z_{>0}$, choose $\alpha(n), \beta(n)\in\Gamma$ with $\abs{\alpha(n)}=\abs{\beta(n)}=n$ such that $\eta\in A_{\alpha(n)}$ and $\zeta\in A_{\beta(n)}$. Then $\alpha$ and $\beta$ are geodesic rays starting from $o$. See Figure~\ref{fig:hyperbolicShadow} for an intuition.

		Choose a sufficiently large number $M_2\in\Z_{>0}$ such that 
		\begin{equation}\label{def:M2}
			3C_6e^{-a(M_2-R)}<\rho(\zeta,\eta),
		\end{equation}
		where constants $C_6>0$ and $R>0$ are from Corollary~\ref{cor:shadow1} and the definition (\ref{def:N}) of $N(\cdot)$, respectively. If there is a vertex $u\in\Gamma$ with $\abs{u}\ge M_2-R$ such that $\mho(u)$ intersects with both $\mho(\alpha(M_2))$ and $\mho(\beta(M_2))$, then since $\eta\in A_{\mho(\alpha(M_2))}$ and $\zeta\in A_{\mho(\beta(M_2))}$, by Corollary~\ref{cor:shadow1},
		\begin{align*}
			\rho(\zeta, \eta) &\le \rho(\zeta, A_{\mho(u)}) + \rho(\eta, A_{\mho(u)}) + \diam A_{\mho(u)}\\
			&< \diam A_{\mho(\alpha(M_2))}+A_{\mho(\beta(M_2))}+ \diam A_{\mho(u)}
			<3C_6 e^{-a(M_2-R)},
		\end{align*}
		which contradicts with (\ref{def:M2}). Hence, by the definition of $N(\cdot)$,
		\begin{equation} \label{eq:emptyIntersect}
			N(\alpha(M_2))\cap N(\beta(M_2))=\emptyset.
		\end{equation}
		
		Since $\{y_n\}$ and $\{z_n\}$ converges to $\eta$ and $\zeta$, respectively, there is a number $M_1>0$ such that for each integer $n>M_1$, 
		\begin{equation*}
			A_{y_n} \subseteq \bigcup_{\substack{u\in\Gamma,\abs{u}=M_2+1,\\\eta\in A_u}}A_u
			\quad \text{and}\quad
			A_{z_n} \subseteq \bigcup_{\substack{u\in\Gamma,\abs{u}=M_2+1,\\\zeta\in A_u}}A_u.
		\end{equation*}
		Thus, by the construction of $\alpha$ and $\beta$, 
		\begin{equation*}
			y_n \in \mho(\alpha(M_2))
			\quad \text{and}\quad
			z_n \in \mho(\beta(M_2)).
		\end{equation*}
		Hence, by the definition of $N(\cdot)$, for each $u\in \Gamma$ with $\abs{u}\ge M_2$, if $y_{n_1},z_{n_2}\in\mho(u)$ for some integers $n_1,n_2 > M_1$, then $\tau^{\abs{u}-M_2}u\in N(\alpha(M_2))\cap N(\beta(M_2))$, which contradicts with (\ref{eq:emptyIntersect}). Here $\tau$ is defined in Subsection~\ref{subsct:visualMetric}. It follows from the definitions (\ref{def:mho}) and (\ref{def:K}) of $\mho(\cdot)$ and $K(\cdot,\cdot)$ that for all integers $n_1, n_2>M_1$ and each $u\in\Gamma$ with $\abs{u}>M_2$, either $K(u,y_{n_1})=0$ or $K(u,z_{n_2})=0$. Hence, $K(u,\xi)=0$ because $\xi$ is the limit point of $\{y_n\}$ and $\{z_n\}$ in the Martin boundary. This contradicts the discussion (\ref{eq:Kneq0}) above. This finishes the proof of the claim.

		For each $\xi\in\partial_M\Gamma$, we choose an arbitrary sequence $\{x_n\}$ in $\Gamma$ converging to $\xi$. Recall that since $\Gamma$ is a proper geodesic metric space as a $1$-complex, the Gromov boundary produces a compactification of $\Gamma$, i.e., $\Gamma\cup \partial\Gamma$ is compact. See for example, \cite[Part III, Proposition~3.7]{BH99}. By the claim above, there is a unique limit point $\eta\in\partial\Gamma$ of $\{x_n\}$. Now put $\Phi(\xi)\=\eta$, then $\Phi$ is what we want. The well-definedness is exactly what we have proved in the contradictory process of the proof of the claim.

		\smallskip

		We verify that the map $\Phi$ is continuous by a diagonal argument. Suppose for the purpose of contradiction that a sequence $\{x_n\}_{n=1}^{+\infty}$ in $\partial_M\Gamma$ converges to $x\in\partial_M\Gamma$, but there is an open set $U\subseteq \Gamma\cup\partial\Gamma$ with $U\ni\Phi(x)$ such that $\Phi(x_n)\not\in U$ for all $n\in\Z_{>0}$. Recall that $\Gamma\cup \partial\Gamma$ is compact and Hausdorff. Hence, we can find a closed subset $V\subseteq U$ and an open subset $W$ of $\Gamma \cup \partial\Gamma$ with $\Phi(x)\in W\subseteq V$.
		Since $\Gamma$ is dense in $\Gamma\cup \partial_M\Gamma$, for each $n\in\Z_{>0}$, we can choose vertices $\{y_{n,m}\}_{m=1}^{+\infty}$ in $\Gamma$ that converge to $x_n$ in $\Gamma\cup \partial_M \Gamma$ (thus to $\Phi(x_n) \notin V$ in $\Gamma\cup \partial \Gamma$) such that $y_{n,m}\not\in V$ for all $m\in\Z_{>0}$. We can choose open subsets $Y_1\supseteq \cdots \supseteq Y_n\supseteq\cdots$ of $\Gamma\cup \partial_M \Gamma$ such that
		\begin{equation}\label{eq:ConvergeTox}
			\bigcap_{n=1}^{+\infty}Y_n=\{x\}
		\end{equation}
		because $\Gamma\cup \partial_M\Gamma$ is metrizable. 
		For each $n\in\Z_{>0}$, since $\{x_i\}_{i=1}^{+\infty}$ converges to $x$, there exists $i_n\in\Z_{>0}$ such that $x_{i_n}\in Y_n$. Since $\{y_{i_n,m}\}_{m=1}^{+\infty}$ converges to $x_{i_n}$, there exists $j_n\in\Z_{>0}$ such that $y_{i_n,j_n}\in Y_n$. Hence, by (\ref{eq:ConvergeTox}), $\{y_{i_n,j_n}\}_{n=1}^{+\infty}$ converges to $x$ in the Martin boundary.
		By the definition of $\Phi$, $\{y_{i_n,j_n}\}$ converges to $\Phi(x)$ in the Gromov boundary. This contradicts with the assumptions that $y_{n,m}\not\in V$. Hence, the assumption that the sequence $\{x_n\}_{n=1}^{+\infty}$ in $\partial_M\Gamma$ converges to $x\in\partial_M\Gamma$ implies that $\{\Phi(x_n)\}_{n=1}^{+\infty}$ converges to $\Phi(x)$. Therefore, $\Phi$ is continuous.

		To see that $\Phi$ is surjective, we recall that $\Gamma\cup \partial_M\Gamma$ is compact. For each point $\xi\in\partial \Gamma$, we may choose a sequence $\{x_n\}\to \xi$ and find a subsequence of it which converges to $\eta\in\partial_M\Gamma$. Then by definition, $\Phi(\eta)=\xi$.
	\end{proof}

\subsection{Non-injective examples}  \label{subsct:non-injective}

		In this subsection, we give a proof of Theorem~\ref{thm:main1counter} to show that the surjection in Theorem~\ref{thm:surj} may not be a homeomorphism. We provide a family of examples to illustrate it. In these examples, the dynamical system is the doubling map on the unit circle and the Markov partition is associated with the dyadic expansion of real numbers. So these examples are simple in the combinatorial structure. The complexity comes from the transition probabilities.

		Put $X=S^1=\R/\Z$ and let $f\:x\mapsto 2x$ be the doubling map on $X$. We set $\alpha \= \{A_0, A_1\}$ with $A_0\=[0,1/2]$, $A_1\=[1/2, 1]$ as the Markov partition for $(X,f)$. Then the vertices of $\Gamma = \Gamma(f, \alpha)$ are all of the finite binary sequences. Each vertex corresponds to an interval of the form $I_{i,n} \= [i/2^n, (i+1)/2^n]$ with $n\in\Z_{\ge 0}$, $i\in\{0,1,\dots, 2^n-1\}$. For the sake of convenience, we use the notation $I_{i,n}$ for $i\in\Z$ and we should note that $I_{i,n}=I_{i+2^n,n}\subseteq X$. We denote the vertex $u\in\Gamma$ with $A_u=I_{i,n}$ by $u_{i,n}$. To better understand what the graph looks like, see Figure~\ref{fig:doublingGraph}. 

		\begin{figure}[!ht]
			\hspace*{-5em}\scalebox{.8}{
				\begin{tikzpicture}
					\tikzset{
						box/.style ={
							circle, 
							minimum width =10pt, 
							minimum height =10pt, 
							inner sep=2pt, 
							draw=black,
						},
						level 1/.style = {sibling distance = 200pt},
						level 2/.style = {sibling distance = 100pt},
						level 3/.style = {sibling distance = 50pt},
						level 4/.style = {sibling distance = 25pt, font=\tiny},
						edge from parent path = {(\tikzparentnode\tikzparentanchor) -- (\tikzchildnode\tikzchildanchor)}
					}
					
			
					\node[box] {$I_{0,0}$}
						child {node[box] (0){$I_{0,1}$}
							child {node[box] (00){$I_{0,2}$}
								child {node[box] (000){$I_{0,3}$}
									child {node[box] (0000){$I_{0,4}$}}
									child {node[box] (0001){$I_{1,4}$}}
								}
								child {node[box] (001){$I_{1,3}$}
									child {node[box] (0010){$I_{2,4}$}}
									child {node[box] (0011){$I_{3,4}$}}
								}
							}
							child {node[box] (01){$I_{1,2}$}
								child {node[box] (010){$I_{2,3}$}
									child {node[box] (0100){$I_{4,4}$}}
									child {node[box] (0101){$I_{5,4}$}}
								}
								child {node[box] (011){$I_{3,3}$}
									child {node[box] (0110){$I_{6,4}$}}
									child {node[box] (0111){$I_{7,4}$}}
								}
							}
						}
						child {node[box] (1){$I_{1,1}$}
							child {node[box] (10){$I_{2,2}$}
								child {node[box] (100){$I_{4,3}$}
									child {node[box] (1000){$I_{8,4}$}}
									child {node[box] (1001){$I_{9,4}$}}
								}
								child {node[box] (101){$I_{5,3}$}
									child {node[box] (1010){$I_{10,4}$}}
									child {node[box] (1011){$I_{11,4}$}}
								}
							}
							child {node[box] (11){$I_{3,2}$}
								child {node[box] (110){$I_{6,3}$}
									child {node[box] (1100){$I_{12,4}$}}
									child {node[box] (1101){$I_{13,4}$}}
								}
								child {node[box] (111){$I_{7,3}$}
									child {node[box] (1110){$I_{14,4}$}}
									child {node[box] (1111){$I_{15,4}$}}
								}
							}
						};

						\draw  (0) -- (10);
						\draw  (1) -- (01);
			
						\draw  (01) -- (100);
						\draw  (10) -- (011);
						\draw  (10) -- (110);
						\draw  (11) -- (101);
						\draw  (00) -- (010);
						\draw  (01) -- (001);
						
						\draw  (101) -- (1100);
						\draw  (110) -- (1011);
						\draw  (110) -- (1110);
						\draw  (111) -- (1101);
						\draw  (100) -- (1010);
						\draw  (101) -- (1001);
						\draw  (001) -- (0100);
						\draw  (010) -- (0011);
						\draw  (010) -- (0110);
						\draw  (011) -- (0101);
						\draw  (000) -- (0010);
						\draw  (001) -- (0001);
						\draw  (011) -- (1000);
						\draw  (100) -- (0111);
			
						\draw  (0) .. controls +(-3,-1) and +(3,1) ..  (11);
						\draw  (1) .. controls +(3,-1) and +(-3,1) ..  (00);
						\draw  (00) .. controls +(-3,-1) and +(3,1) ..  (111);
						\draw  (11) .. controls +(3,-1) and +(-3,1) ..  (000);
						\draw  (000) .. controls +(-3,-1) and +(3,1) ..  (1111);
						\draw  (111) .. controls +(3,-1) and +(-3,1) ..  (0000);

						\draw  (0) -- (1);
						\draw  (00) -- (01) -- (10) -- (11);
						\draw  (000) -- (001) -- (010) -- (011) -- (100) -- (101) -- (110) -- (111);
						\draw  (0000) -- (0001) -- (0010) -- (0011) -- (0100) -- (0101) -- (0110) -- (0111) -- (1000) -- (1001) -- (1010) -- (1011) -- (1100) -- (1101) -- (1110) -- (1111);
			
						\draw[dashed] (00) .. controls +(-5,-.5) and +(5,-.5) .. (11);
						\draw[dashed] (000) .. controls +(-5,-.5) and +(5,-.5) .. (111);
						\draw[dashed] (0000) .. controls +(-5,-.5) and +(5,-.5) .. (1111);
					
				\end{tikzpicture}
			}
			\caption{The tile graph of the doubling map on the circle.}
			\label{fig:doublingGraph}
		\end{figure}
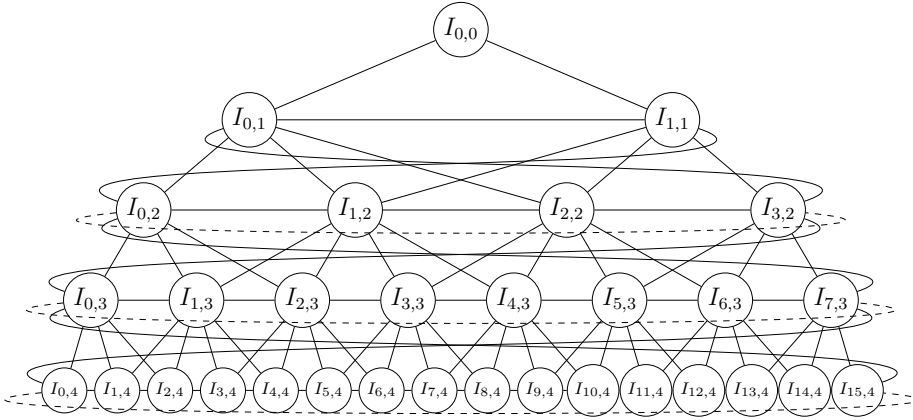
		
		For each $x\in (0,1)$, put $y\= (1-x)/3$.
		Fix $x\in(0,1)$. We define a family of transition probabilities $p_x$ on $\Gamma$. Define
		\begin{equation*}
		    p_x(o,1) \= \frac{1+2x}{3},\quad p_x(o,0) \= \frac{2-2x}{3},
		\end{equation*}
		and for all integers $n,m\in \Z_{>0}$, $i\in\{0,1,\dots, 2^n-1\}$, and $j\in\{0,1,\dots, 2^{m}-1\}$, put
		\begin{equation*}
			p_x(u_{i,n}, u_{j,m}) \= \begin{cases}
				y & \text{if }m=n+1, j\in J_{i,n} \text{ such that } j \not\equiv 2\bmod{4},\\
				x & \text{if }m=n+1, j\in J_{i,n}  \text{ such that } j \equiv 2\bmod{4},\\
				0 &\text{otherwise}, \end{cases}
		\end{equation*}
		where
		\begin{equation*}
			J_{i,n} \= \bigl\{ j \in \bigl[0, 2^{n+1}-1 \bigr] : k \in \{2i-1, 2i, 2i+1, 2i+2\}, j \equiv k \bmod{ 2^{n+1}} \bigr\}.
		\end{equation*}

		It is easy to verify that $p_x$ satisfies the Assumptions in Section~\ref{sct:Introduction}. For the sake of convenience, we write $p=p_x$ if there is no other choice of $x$. See Figure~\ref{fig:doublingWalk} for an intuition of the distribution of $p_x$.

			\begin{figure}[!ht]
				\scalebox{.8}{
				\begin{tikzpicture}
					\tikzset{
						box/.style ={
							circle, 
							minimum width =10pt, 
							minimum height =10pt, 
							inner sep=2pt, 
							draw=black,
						},
						level 1/.style = {sibling distance = 200pt},
						level 2/.style = {sibling distance = 100pt},
						level 3/.style = {sibling distance = 50pt},
						level 4/.style = {sibling distance = 25pt, font=\tiny},
						edge from parent path = {[->](\tikzparentnode\tikzparentanchor) -- (\tikzchildnode\tikzchildanchor)}
					}
					

					\node[box] {$\emptyset$}
						child {node[box] (0){0}
							child {node[box] (00){00}
								child {node[box] (000){000}
									child {node[box] (0000){0000}}
									child {node[box] (0001){0001}}
								}
								child {node[box] (001){001}
									child {node[box,fill=yellow] (0010){0010}}
									child {node[box] (0011){0011}}
								}
							}
							child {node[box] (01){01}
								child {node[box,fill=yellow] (010){010}
									child {node[box] (0100){0100}}
									child {node[box] (0101){0101}}
									coordinate (B1)
								}
								child {node[box] (011){011}
									child {node[box,fill=yellow] (0110){0110} coordinate (B2)}
									child {node[box] (0111){0111}}
								}
							}
							edge from parent node[above left]{$\frac{2-2x}{3}$}
						}
						child {node[box] (1){1}
							child {node[box,fill=yellow] (10){10}
								child {node[box] (100){100}
									child {node[box] (1000){1000} coordinate (B3)}
									child {node[box] (1001){1001}}
									coordinate (B4)
								}
								child {node[box] (101){101}
									child {node[box,fill=yellow] (1010){1010}}
									child {node[box] (1011){1011}}
								}
							}
							child {node[box] (11){11}
								child {node[box,fill=yellow] (110){110}
									child {node[box] (1100){1100}}
									child {node[box] (1101){1101}}
								}
								child {node[box] (111){111}
									child {node[box,fill=yellow] (1110){1110}}
									child {node[box] (1111){1111}}
								}
							}
							edge from parent node[above right]{$\frac{1+2x}{3}$}
						};

						\draw [->] (0) -- (10);
						\draw [->]  (1) -- (01);
						\draw [->]  (01) -- (100);
						\draw [->]  (10) -- (011);
						\draw [->]  (10) -- (110);
						\draw [->]  (11) -- (101);
						\draw [->]  (00) -- (010);
						\draw [->]  (01) -- (001);
						
						\draw [->]  (101) -- (1100);
						\draw [->]  (110) -- (1011);
						\draw [->]  (110) -- (1110);
						\draw [->]  (111) -- (1101);
						\draw [->]  (100) -- (1010);
						\draw [->]  (101) -- (1001);
						\draw [->]  (001) -- (0100);
						\draw [->]  (010) -- (0011);
						\draw [->]  (010) -- (0110);
						\draw [->]  (011) -- (0101);
						\draw [->]  (000) -- (0010);
						\draw [->]  (001) -- (0001);
						\draw [->]  (011) -- (1000);
						\draw [->]  (100) -- (0111);

						\draw [->]  (0) .. controls +(-3,-1) and +(3,1) ..  (11);
						\draw [->]  (1) .. controls +(3,-1) and +(-3,1) ..  (00);
						\draw [->]  (00) .. controls +(-3,-1) and +(3,1) ..  (111);
						\draw [->]  (11) .. controls +(3,-1) and +(-3,1) ..  (000);
						\draw [->]  (000) .. controls +(-3,-1) and +(3,1) ..  (1111);
						\draw [->]  (111) .. controls +(3,-1) and +(-3,1) ..  (0000);

						\draw[red] ([xshift=-30pt, yshift=20pt]B1)--([xshift=-10pt, yshift=-15pt]B2)--([xshift=10pt, yshift=-15]B3)--([xshift=20pt, yshift=20pt]B4)--cycle;
					
				\end{tikzpicture}
			}
				\caption{The transition probabilities to the shaded vertices are all equal to $x$, while to the remaining vertices (except for $0$ and $1$), they are all equal to $(1-x)/3$. Some of the vertices $x_n$, $y_n$, and $z_n$ are enclosed by the trapezoid.}
				\label{fig:doublingWalk}
			\end{figure}
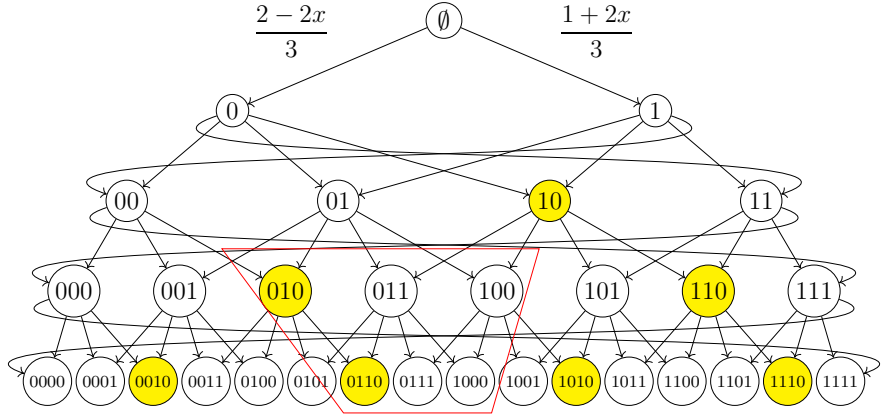

		Theorem~\ref{thm:main1counter} follows from the following proposition. Informally speaking, the proposition says that if $x\in(2/5, 1)$, then the transition probability $p_x$ becomes ``unbalanced'', and this property leads to the existence of different growth rates in harmonic functions supported near a geodesic ray starting from $o$. On the other hand, although there is a counterexample of $\Phi$ being a homeomorphism, we still have ideas to prove that sometimes with a ``balanced'' transition probability, it is a homeomorphism.

		\begin{prop}
			Assume that the dynamical system $(X,f)$, the Markov partition $\alpha$, and the transition probability $p_x$ are defined above for some $x\in(0,1)$. The surjection $\Phi$ given in Theorem~\ref{thm:surj} is a homeomorphism if $x\in (0,2/5)$, while if $x\in(2/5,1)$, then $\Phi$ is not a homeomorphism.
		\end{prop}
		\begin{proof}

		For each $x\in (2/5, 1)$, we show that for $t=1/2$, $\# \Phi^{-1}(t)\ge 2$. Consider two sequences of vertices $\{x_n\}_{n\in\Z_{>0}}$ and $\{y_n\}_{n\in\Z_{>0}}$ with
		\begin{equation*}
			x_n \= u_{2^{n-1}-2, n} \quad \text{and} \quad y_n \= u_{2^{n-1}-1, n} .
		\end{equation*}
		Since $A_{y_n} = [1/2-1/2^n, 1/2]$ contains the point $t$, $\{y_n\}$ is a geodesic ray from $o$ to the boundary point $t$. Hence, $y_n$ converges to $t$ in the topology of the Gromov boundary. Note that $d(x_n,y_n)=1$, so $x_n$ has a bounded distance from $y_n$ and $x_n$ also converges to $t$ in the topology of the Gromov boundary.

		Then we show that $\{K(\cdot, x_n)\}, \{K(\cdot,y_n)\}$ converge to different harmonic functions on $\Gamma$, thus, $\{x_n\}$ and $\{y_n\}$ converge to different points in the Martin boundary $\Mb\Gamma$. We moreover put $z_n\= u_{2^{n-1}, n}$. Then by the definition of $p$, as we can see in Figure~\ref{fig:harmonicConstruct} that for each $n \in \Z_{>0}$ and each $v\in\Gamma$,
		\begin{align*}
			\hP(v,x_{n+1})>0 \iff v \in \{x_n, y_n\},\\
			\hP(v,y_{n+1})>0 \iff v \in \{y_n, z_n\},\\
			\hP(v,z_{n+1})>0 \iff v \in \{y_n, z_n\}.
		\end{align*}
		Hence, by induction, $K(\cdot, x_n)$ and $K(\cdot, y_n)$ are both supported on $\{x_n\}\cup\{y_n\}\cup\{z_n\} \cup \{o, 0, 1\}$.

		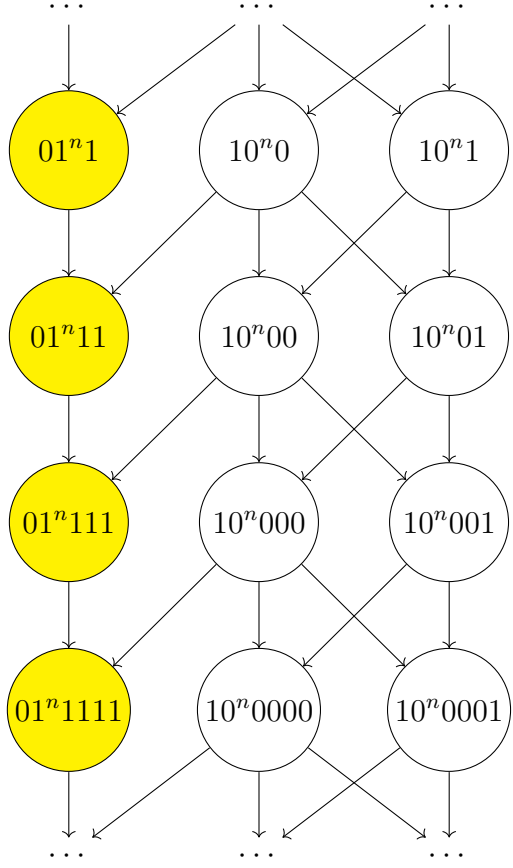
\begin{figure}[!ht]
			\begin{tikzpicture}
				\tikzset{
					box/.style ={
						circle, 
						minimum width =45pt, 
						minimum height =45pt, 
						inner sep=2pt, 
						draw=black,
					}
				}
				\matrix (m) [matrix of nodes, row sep=25pt, column sep=25pt]{
					$\cdots$ &$\cdots$& $\cdots$\\
					|[box, fill=yellow]| $01^n1$ & |[box]| $10^n0$ & |[box]| $10^n1$ \\
					|[box, fill=yellow]| $01^n11$ & |[box]| $10^n00$ & |[box]| $10^n01$ \\
					|[box, fill=yellow]| $01^n111$ & |[box]| $10^n000$ & |[box]| $10^n001$ \\
					|[box, fill=yellow]| $01^n1111$ & |[box]| $10^n0000$ & |[box]| $10^n0001$ \\
					$\cdots$ &$\cdots$& $\cdots$\\
				};
				\foreach \i [remember=\i as \j ((initially 1))] in {2,...,6}{
					
					\draw[->] (m-\j-1) -- (m-\i-1);
					\draw[->] (m-\j-2) -- (m-\i-1);
					\draw[->] (m-\j-2) -- (m-\i-2);
					\draw[->] (m-\j-3) -- (m-\i-2);
					\draw[->] (m-\j-2) -- (m-\i-3);
					\draw[->] (m-\j-3) -- (m-\i-3);
				}
			\end{tikzpicture}
			\caption{A part of the subgraph consisting of $x_n$, $y_n$, and $z_n$.}
			\label{fig:harmonicConstruct}
		\end{figure}

		Note that $K(u, x_n)$ (resp.\ $K(u, y_n)$) is harmonic at each $u\in\Gamma$ with $\abs{u}<n$. Thus, by the definition of $p_x$,
		\begin{align*}
			K(x_m, x_n) &= \hP(x_m, x_{m+1})K(x_{m+1}, x_n) = x K(x_{m+1}, x_n),\\
			K(y_m, x_n) &= \hP(y_m, x_{m+1})K(x_{m+1}, x_n) + \hP(y_m, y_{m+1})K(y_{m+1}, x_n)\\
			&\qquad + \hP(y_m, z_{m+1})K(z_{m+1}, x_n)\notag\\
			&= x K(x_{m+1}, x_n) + y (K(y_{m+1}, x_n) + K(z_{m+1}, x_n)),\notag\\
			K(z_m, x_n) &= \hP(z_m, y_{m+1})K(y_{m+1}, x_n) + \hP(z_m, z_{m+1})K(z_{m+1}, x_n)\\
			&= y (K(y_{m+1}, x_n) + K(z_{m+1}, x_n))\notag
		\end{align*}
		for each integer $m \in [0,n-1]$ with initial value
		\begin{subequations}
			\begin{align*}
				K(x_m, x_m) &=  1 / F(o,x_m),  &K(y_m, x_m)&=K(z_m,x_m)=0,\\
				K(y_m, y_m) &=  1 / F(o,y_m),  &K(x_m, y_m)&=K(z_m,y_m)=0.
			\end{align*}
		\end{subequations}
		
		Hence, if we put
		\begin{equation*}
			M \= \begin{bmatrix}
				x&0&0\\
				x&y&y\\
				0&y&y
			\end{bmatrix},
		\end{equation*}
		then
		\begin{equation} \label{eq:EXiter}
			\begin{bmatrix}
				K(x_{m-k}, x_m)\\K(y_{m-k}, x_m)\\K(z_{m-k}, x_m)
			\end{bmatrix}
			= \frac{M^k}{F(o,x_m)} 
			\begin{bmatrix}
				1\\0\\0
			\end{bmatrix}, \quad \begin{bmatrix}
				K(x_{m-k}, y_m)\\K(y_{m-k}, y_m)\\K(z_{m-k}, y_m)
			\end{bmatrix}
			= \frac{M^k}{F(o,x_m)} 
			\begin{bmatrix}
				0\\1\\0
			\end{bmatrix}.
		\end{equation}
		The characteristic polynomial of $M$ is
		\begin{equation*}
			\chi_M(t) = (t-x)(t- 2 y)t .
		\end{equation*}
		Recall that $y=(1-x)/3$. The characteristic vectors of characteristic values $0$, $x$, and $2y$ are $(0,-1,1)$, $(5x-2 ,4x-1, 1-x)$ and $(0, 1, 1)$, respectively. We should also note that for each $x > 2/5$, the inequality $x > 2y$ always holds. For each $n\in\Z_{>0}$, we can see the asymptotic behavior of $M^n(1,0,0)$ from the decomposition
		\begin{equation*}
			M^n\begin{bmatrix}
				1\\0\\0
			\end{bmatrix}=\frac{x^n}{5x-2}
			\begin{bmatrix}
				5x-2\\4x-1\\1-x
			\end{bmatrix}-\frac{3x(2(1-x)/3)^n}{2(5x-2)}
			\begin{bmatrix}
				0\\1\\1
			\end{bmatrix}.
		\end{equation*}
		Hence, by (\ref{eq:EXiter}), we have
		\begin{subequations}\label{eq:EXasymp1}
			\begin{align}
				\lim_{m\to+\infty} K(x_i, x_m) : K(y_i, x_m) : K(z_i, x_m) &= 5x-2:4x-1:1-x, \\
				\lim_{m\to+\infty} K(x_i, x_m):K(x_{i+1}, x_m) &= 1:x.
			\end{align}
		\end{subequations}
		That implies that $x_m$ converges to a point $\xi_x$ in the Martin boundary, and moreover $K(\cdot, \xi_x)$ can be calculated by (\ref{eq:EXasymp1}) explicitly since we have showed that $K(\cdot, x_m)$ is supported on $\{x_n,y_n,z_n:n\in\Z_{>0}\}\cup\{0,1,o\}$. To be precise, there is a constant $C_x>0$ such that for each integer $n\ge 2$,
		\begin{equation}\label{eq:res1}
			K(x_n,\xi_x) = C_x x^{-n}(5x-2),\quad K(y_n,\xi_x) = C_x x^{-n}(4x-1), \quad K(z_n,\xi_x) = C_x x^{-n}(1-x).
		\end{equation}
		
		However, for the asymptotic behavior of $K(\cdot, y_m)$, we have for each $n\in\Z_{>0}$,
		\begin{equation*}
			M^n\begin{bmatrix}
				0\\1\\0
			\end{bmatrix}=
			\frac{(2(1-x)/3)^n}{2}
			\begin{bmatrix}
				0\\1\\1
			\end{bmatrix}.
		\end{equation*}
		Hence, by (\ref{eq:EXiter}), we have
		\begin{subequations}\label{eq:EXasymp2}
			\begin{align}
				\lim_{m\to+\infty} K(x_i, y_m) : K(y_i, y_m) : K(z_i, y_m) &= 0:1:1,\\
				\lim_{m\to+\infty} K(x_i, x_m):K(x_{i+1}, x_m) &= 3:2(1-x).
			\end{align}
		\end{subequations}
		
		This implies that $y_m$ converges to a point $\xi_y$ in the Martin boundary, and moreover $K(\cdot, \xi_y)$ can be calculated by (\ref{eq:EXasymp2}) explicitly since we have showed that $K(\cdot, y_m)$ is supported on $\{x_n,y_n,z_n:n\in\Z_{>0}\}\cup\{0,1,o\}$. To be precise, there is a constant $C_y>0$ such that for each $n\ge 2$,
		\begin{equation}\label{eq:res2}
			K(x_n,\xi_y) = 0,\quad K(y_n,\xi_y) = C_y\biggl(\frac{3}{2(1-x)}\biggr)^n, \quad K(z_n,\xi_y) = C_y\biggl(\frac{3}{2(1-x)}\biggr)^n.
		\end{equation}

		Now by (\ref{eq:res1}) and (\ref{eq:res2}), $\xi_x$ and $\xi_y$ are different points in the Martin boundary. However, from the construction of $\{x_n\}$ and $\{y_n\}$, $\Phi(\xi_x) = \Phi(\xi_y) = 1/2$. Therefore, in this situation of $x\in (2/5,1)$, $\Phi$ is not a homeomorphism.

		\smallskip

		For $x\in (0, 2/5)$, we prove that $\Phi$ is a homeomorphism. For each $\xi\in X$, we assume that $\eta, \zeta\in\Mb\Gamma$ are two preimages of $\xi$, i.e., $\Phi(\eta)=\Phi(\zeta)=\xi$. We aim to prove $\eta=\zeta$. Let $K(\cdot,\eta)$ and $K(\cdot, \zeta)$ be the harmonic functions associated with $\eta$ and $\zeta$, respectively. It suffices to show that $K(\cdot, \eta)=K(\cdot, \zeta)$.
		
		\smallskip
		
		\emph{Case 1.} If $\xi = 2^{-k} m$ for some $m,k\in\Z_{\ge 0}$ with $k>3$ and $4|m$, then we denote, for each $n\in \Z_{>0}$,
		\begin{equation*}
			x_n \= u_{2^nm-2, k+n}, \quad y_n \= u_{2^nm-1, k+n}, \quad z_n \= u_{2^nm, k+n}, \quad w_n \= u_{2^nm+1, k+n}.
		\end{equation*}
		Note that for each $u=u_{i,k+n}$ with $u\not\in\{x_n, y_n, z_n, w_n\}$, $A_{\mho(u)} = \bigl( (i-1)/2^{n+k}, (i+2)/2^{n+k} \bigr)$. For such a vertex $u$, $\xi\not\in \overline{A_{\mho(u)}}$, thus $K(u,\eta)=0$. Therefore, if $\abs{u}>k$, then $K(u,\eta)>0$ implies $u\in\{x_n, y_n, z_n, w_n\}$ for some $n\in\Z_{>0}$.

		We do similar calculations as in the case of $x\in(2/5, 1)$. Put $y\=(1-x)/3$.
		Since $K(\cdot, \eta)$ is harmonic, we have 
		\begin{equation*}
			\begin{bmatrix}
				K(x_{n}, \eta)\\K(y_{n}, \eta)\\K(z_{n}, \eta)\\K(w_{n}, \eta)
			\end{bmatrix}
			= \begin{bmatrix}
				x&&&\\ x&y&y&\\ &y&y&y\\ &&&y
			\end{bmatrix}^k 
			\begin{bmatrix}
				K(x_{n+l}, \eta)\\K(y_{n+l}, \eta)\\K(z_{n+l}, \eta)\\K(w_{n+l}, \eta)
			\end{bmatrix}
		\end{equation*}
		Since the matrix
		\begin{equation*}
			M\= \begin{bmatrix}
				x&&&\\ x&y&y&\\ &y&y&y\\ &&&y
			\end{bmatrix}
		\end{equation*}
		has a maximal characteristic value $\lambda=2y$ with characteristic vector $(0,1,1,0)$. Since the characteristic vectors of the other characteristic values are
		\begin{equation*}
			\begin{bmatrix}
				x-2y\\ x-y\\ y\\ 0
			\end{bmatrix},\;
			\begin{bmatrix}
				0\\ -1\\ 0\\ 1
			\end{bmatrix},\text{ and }
			\begin{bmatrix}
				0\\ -1\\ 1\\ 0
			\end{bmatrix},
		\end{equation*}
		which are all not non-negative, there is a constant $C>0$ such that for each $n\in\Z_{>0}$,
		\begin{equation*}
			K(x_n,\eta) = K(w_n,\eta) = 0 \quad\text{ and }\quad K(y_n,\eta) = K(z_n, \eta) = C (2y)^n.
		\end{equation*}
		This determines the whole function $K(\cdot, \eta)$ by 
		\begin{equation*}
			K(u,\eta) = \sum_{v\in\{x_1,y_1, z_1, w_1\}} F(u,v) K(v,\eta)  \quad \text{for each } u\in\Gamma\text{ with }\abs{u}\le k.
		\end{equation*}
		So does $K(\cdot, \zeta)$. Thus, for some $D>0$, $K(\cdot, \eta) = D K(\cdot, \zeta)$. Moreover, $K(o,\eta)=K(o,\zeta)=1$. Hence, $K(\cdot, \eta) = K(\cdot, \zeta)$.

		\smallskip

		\emph{Case 2.} If $\xi \neq 2^{-k} m$ for every $m,k\in\Z_{\ge 0}$, then $\xi$ is not on the boundary of any tile. Hence, for each $n\in \Z_{\ge 0}$, there is a unique tile $y_n\in\Gamma$ with $\abs{y_n}=n$ such that $\xi\in A_{y_n}$. We assume that $y_n=u_{i_n,n}$. Then we denote $x_n\= u_{i_n-1,n}$ and $z_n\= u_{i_n+1,n}$ the two adjacent tiles of the same level as $y_n$.

		Note that for all integers $i\in\Z$ and $n>1$, we have 
		\begin{equation*}
			A_{\mho(u_{i,n})} = \bigl( (i-1)/2^{n+k}, (i+2)/2^{n+k} \bigr).
		\end{equation*}
		If $\xi\in A_{\mho(u_{i,n})}$, then $u_{i,n}\in\{x_n, y_n, z_n\}$. That is, $K(\cdot,\eta)$ and $K(\cdot,\zeta)$ are both supported on $\{x_n, y_n, z_n:n\in\Z_{\ge 0}\}$.

		According to the remainder of $i_n\bmod{4}$, the transition matrix falls into one of the 4 types, which are
		\begin{equation*}
			M_0 \= \begin{bmatrix}
				y&y&\\ y&y&y\\ &&y
			\end{bmatrix}, \;
			M_1 \= \begin{bmatrix}
				y&&\\ y&y&x\\ &y&x
			\end{bmatrix}, \;
			M_2 \= \begin{bmatrix}
				y&x&\\ y&x&y\\ &&y
			\end{bmatrix}, \;
			M_3 \= \begin{bmatrix}
				x&&\\ x&y&y\\ &y&y
			\end{bmatrix}. 
		\end{equation*}
		That is, if $i_n\equiv j_n\pmod{4}$ for some $j_n\in\{0,1,2,3\}$ and $n>2$, then by the construction of the transition probability $p_x$, since $K(\cdot, \eta)$ is harmonic,
		\begin{equation} \label{eq:transition}
			\begin{bmatrix}
				K(x_{n-1}, \eta)\\K(y_{n-1}, \eta)\\K(z_{n-1}, \eta)
			\end{bmatrix}
			=M_{j_n}
			\begin{bmatrix}
				K(x_{n}, \eta)\\K(y_{n}, \eta)\\K(z_{n}, \eta)
			\end{bmatrix}.
		\end{equation}
		By (\ref{eq:transition}) and properties of each $M_j$, we deduce that $K(x_{n-1},\eta)+K(z_{n-1},\eta)=K(y_{n-1}, \eta)$ for each integer $n>2$. We denote, for each integer $n>1$,
		\begin{equation*}
			\Lambda_n(\eta) \= \frac{K(x_{n}, \eta)}{K(y_{n}, \eta)}.
		\end{equation*}
		Then we denote $z\=x/y$. Equation (\ref{eq:transition}) can be written as
		\begin{equation*}
			\Lambda_{n-1}(\eta) = F_{j_n}(\Lambda_{n}(\eta)), 
		\end{equation*}
		where
		\begin{equation*}
			F_j(t) \= \begin{cases}
				\frac{t+1}{2} &\text{if }j=0,\\
				\frac{yt}{yt+(1-t)x+y} = \frac{t}{(1-z)t+z+1} &\text{if }j=1,\\
				\frac{yt+x}{x+y} = \frac{t+z}{1+z} &\text{if }j=2,\\
				\frac{xt}{y+y(1-t)+xt} = \frac{zt}{(z-1)t+2} &\text{if }j=3.
			\end{cases}
		\end{equation*}

		The condition $x\in(0,2/5)$ implies $z\in(0,2)$. Note that for each $t\in[0,1]$, the derivative of $F_j$ satisfies
		\begin{align*}
			F_0'(t) &= 1/2<1,\\
			F_1'(t) &= (z+1)((1-z)t+(1+z))^{-2}\le \max\bigl\{(z+1) / 4, (1+z)^{-1}\bigr\} <1,\\
			F_2'(t) &= 1/ (1+z) <1,\\
			F_3'(t) &=  2z ((z-1)t+2)^{-2} \le \max\bigl\{z /4, 2z (1+z)^{-2}\bigr\} <1.
		\end{align*}
		Hence, there is a number $\lambda\in(0,1)$ such that for each $t\in[0,1]$ each and $j\in\{0,1,2,3\}$, $F_j'(t)\le \lambda$. It follows that for each interval $I\subseteq[0,1]$, $\abs{F_j(I)}\le \lambda\abs{I}$. By iteration, for each integer $n>2$ and each $m \in \Z_{>0}$, the length of the interval
		\begin{equation*}
			\Absbig{F_{j_n}\circ \cdots \circ F_{j_{n+m}} ([0,1])} \le \lambda^m.
		\end{equation*}
		Let $m\to+\infty$. There is a unique point in the decreasing sequence of closed sets
		\begin{equation*}
			\Lambda_{n-1}(\eta) \in F_{j_n}([0,1]) \supseteq F_{j_n}\circ F_{j_{n+1}}([0,1]) \supseteq \cdots.
		\end{equation*}
		The discussion above also holds for $\zeta$ taking the place of $\eta$. Hence, $\Lambda_{n-1}(\eta) = \Lambda_{n-1}(\zeta)$ for each integer $n>2$. By the definition of $\Lambda$, there is a constant $C_n>0$ for each integer $n>2$ such that
		\begin{equation*}
			K(u, \eta) = C_{\abs{u}} K(u, \zeta)  \qquad\text{for each }u\in\Gamma\text{ with }\abs{u}>1.
		\end{equation*}
		By (\ref{eq:transition}), it is straightforward to show that all of the $C_n$'s are identical. Hence, $K(\cdot, \eta)$ is a multiple of $K(\cdot, \zeta)$. Since $K(o,\zeta)=1=K(o,\eta)$, the two functions are identical. That is, $\zeta=\eta$.

		Finally, combining Case~1 and Case~2, we have proved that if $x\in(0,2/5)$, then $\Phi$ is a bijection. Since a continuous bijection between compact Hausdorff spaces is a homeomorphism, we deduce that $\Phi$ is a homeomorphism when $x\in(0,2/5)$.
	\end{proof}
		
	\begin{rem}
		In fact, for $x=2/5$, we can still prove by a similar method that $p_x$ is a homeomorphism. The proof is a little more complicated because, in step~1, the matrix $M$ is not diagonalizable at the characteristic value $x=2y=2/5$, while in step~2, there is not a uniform bound of $F_j'(t)$. In fact, $F'_3(0)=1$. These obstacles can be bypassed by careful discussions. In step~1, the convergence result of the Martin kernel is still true. In step~2, we can still show by the explicit expression of $F_3$ that for $I$ close to $0$, $\abs{F_3'(I)} \le  \abs{I} \big/ \bigl( 1+2^{-1} \abs{I} \bigr)$, and the iterated length of an interval still converges to $0$. 
	\end{rem}
		
	\begin{rem}
		According to this example, in some cases, when the transition probability is not ``balanced'', some points of the phase space split into several points in the Martin boundary of the tile graph. According to the proof, we can see that, the corresponding harmonic functions $K(\cdot,\xi)$ of these points $\xi\in\Mb\Gamma$ may have different growth rates. Moreover, the difference in the growth rate causes the separation of these points. The failure of the Harnack inequality makes the Green kernels $K(\cdot, u)$ and $K(\cdot, v)$ corresponding to two adjacent vertices $u,v\in\Gamma$ different.
	\end{rem}



\section{Fractal dimension of the harmonic measure} \label{sct:harmonicMeasure}

	This section is devoted to establishing Theorem~\ref{thm:main2}. In this section, we assume that the dynamical system $(X,f)$ satisfies the Assumptions in Subsection~\ref{subsect:expandingDyn}. Let $\Gamma$ be the tile graph associated with $f$ and a fixed Markov partition $\alpha$ with each $A\in\alpha$ connected so that Theorem~\ref{thm:hyp} can be applied. The tile graph $\Gamma$ is equipped with a natural shift map $\sigma$ defined in Subsection~\ref{subsct:visualMetric}. We focus on some basic properties of the random walks on the tile graph $\Gamma$ under the Assumptions in Section~\ref{sct:Introduction}. 
	
	By Assumptions~(D) and~(B) in Section~\ref{sct:Introduction}, $\abs{Z_{n+1}}-\abs{Z_n}$ is i.i.d.~with the distribution of $\abs{Z_1}$. By Assumption~(A) in Section~\ref{sct:Introduction}, $\bE(\abs{Z_1})<+\infty$. Hence, by the law of large numbers for i.i.d.~variables, $\abs{Z_n} /n$ has an almost sure limit $l \= \bE(\abs{Z_1})$. We call $l$ the \defn{asymptotic drift} or the \defn{drift} of the random walk $P$.

	Let $\nu\coloneqq \Phi_*\nu^{\Mb\Gamma}$ be the push-forward of the harmonic measure from the Martin boundary to $X$ by the map $\Phi$ provided in Theorem~\ref{thm:surj}. By abuse of terminology, we also call $\nu$ the harmonic measure if there is no confusion on the domain of $\nu$.

	For a sample path $\omega\in\Omega$, we denote by $Z_n=Z_n(\omega)\in\Gamma$ the vertex of the $n$-th step of the path. The shift map $T\:\Omega\to\Omega$ for the Markov process defined by
	\begin{equation}  \label{def:T}
		Z_n(T\omega) = \sigma^{\abs{Z_1(\omega)}} Z_{n+1}(\omega), \qquad \text{for } n \in \Z_{>0} \text{ and } \omega \in \Omega,
	\end{equation}
	induces a dynamical system on the space of sample paths. By Assumption~(D) in Section~\ref{sct:Introduction}, $T$ is $\bP$-measure-preserving. In fact, $T$ is ergodic by Theorem~\ref{thm:mixing}.

	We put, for each $n\in \Z_{>0}$,
	\begin{subequations}\label{def:approxDim}
		\begin{align}
			g_n(\omega) &\= -\log F(o,Z_{n}),\\
			\tg_n(\omega) &\= -\log F(Z_{N_0},Z_{N_0+n}).
		\end{align}
	\end{subequations}
	
	We will show by an ergodic theorem that $g_n/n$ converges to a constant $l_G$, called the \defn{Green drift}, almost surely, and so does the limit supremum of $f_n/n$. The almost sure limit supremum of $f_n/n$ is related to the packing dimension of the harmonic measure $\dims_P \tnu$.
	
	The following lemma justifies the definition of the Green drift $l_G$, which plays an important role in the dimension formula of the harmonic measure.
	
	\begin{lemma} \label{lem:deflG}
		The sequence of measurable functions $\{g_n/n\}_{n\in\Z_{>0}}$ converges almost surely to some constant $l_G\in\R$.
	\end{lemma}

	The Harnack inequality is a key tool in similar investigations. Note that the random walk we consider is one-sided, and the classical Harnack inequality \emph{does not} hold in our context. Our strategy is to formulate and establish a weaker version of the Harnack inequality as follows.

\subsection{Weak Harnack inequality}

\begin{lemma}[Weak Harnack inequality]\label{lem:Harnack}
	There exist constants $C_3>0$ and $N_1 \in\Z_{>0}$ such that, for each pair of $u,v\in \Gamma$, there is a constant $C_2=C_2(u,v)>1$ with the following property: for each $w\in\Gamma$ with $\abs{w}-N_1\ge \max\{\abs{u},\abs{v}\}$ satisfying either
	\begin{enumerate}[\hspace{2em}(1)]
		\smallskip
		\item $\overline{A_{\mho(w)}}\cap A_u\neq\emptyset$ or
		\smallskip
		\item $\overline{A_{N(w)}}\subseteq B\bigl(A_u, C_3e^{-a\abs{u}}\bigr)$,
	\end{enumerate}
	we have $F(v,w) \le C_2 F(u,w)$.
\end{lemma}

	In order to establish Lemma~\ref{lem:Harnack} at the end of this subsection, we first verify the following lemma showing that in general if some tile $w\in\Gamma$, as a subset of $X$, completely lies in a certain neighborhood of another tile $u\in\Gamma$, then $w$ must be in the shadow of $u$.
	\begin{lemma}\label{lem:inShadow}
		There is a number $C_3>0$ such that for all $u, w\in\Gamma$ with $\abs{w}\ge \abs{u}+1$, if $A_w\subseteq B\bigl(A_u, C_3 e^{-a\abs{u}}\bigr)$, then $w\in\mho(u)$.
	\end{lemma}
	\begin{proof}
		For each $u\in\Gamma$, put
		\begin{equation}\label{eq:defC3prime}
			C_3'(u) \= e^{a\abs{u}} \inf \{\rho(A_u,A_v) : v\in\Gamma, \abs{v}=\abs{u}+1, d(A_u,A_v)>0\}.
		\end{equation}
		Let $\xi$ be the constant in Proposition~\ref{prop:visualLocal}. Fix $M>0$ sufficiently large such that for all $u,v\in\Gamma$ with $\abs{u}\ge M$ and $d(u,v)\le 2$, 
		\begin{equation}\label{eq:locality1}
			\diam A_u\cup A_v < \xi.
		\end{equation}
		Consider $u,v,u',v'\in\Gamma$ and $n\in\Z_{>0}$ such that $\abs{u}= M$, $d(u',v')\le 2$, $\sigma^n u'=u$, and $\sigma^n v'=v$. For all $x\in A_{u'}$ and $y\in A_{v'}$, by Proposition~\ref{prop:visualLocal} and (\ref{eq:locality1}), we have $\rho \bigl( f^n x, f^n y \bigr) \asymp e^{-an}f(x,y) \ge e^{-an}C_3'(u)$. Hence, there is a constant $D=D(\asymp)>0$ such that $C_3'(u')\ge D \min_{u''\in\Gamma,\abs{u''}\le M}C_3'(u'')$.

		It follows that $C_3'(u)$ has a positive infimum as $u$ ranges over $\Gamma$. We denote the infimum by $C_3>0$. If $A_w\subseteq B\bigl(A_u, C_3 e^{-a\abs{u}}\bigr)$, then, since $\abs{w} \ge \abs{u}+1$, by (\ref{eq:defC3prime}), $A_w\subseteq A_v$ for some $v\in\Gamma$ with $\abs{v}=\abs{u}+1$ and $A_v\cap A_u\neq\emptyset$. Therefore, by the definition (\ref{def:mho}) of $\mho(\cdot)$ and Assumption~(C) in Section~\ref{sct:Introduction}, $w\in\mho(u)$.
	\end{proof}

	\begin{proof}[Proof of Lemma~\ref{lem:Harnack}]
		Let $C_3>0$ be the constant from Lemma~\ref{lem:inShadow}. 
		Fix $u,v\in\Gamma$. Choose an integer $N_1>R$ such that
		\begin{equation} \label{eq:defN1Harnack}
			4C_6e^{-a(N_1-R)}\le C_3,
		\end{equation}
		where the constants $C_6>0$, $C_3>0$, and $R>0$ are from Corollary~\ref{cor:shadow1}, Lemma~\ref{lem:inShadow}, and the definition (\ref{def:N}) of $N(\cdot)$, respectively.
		
		By Corollary~\ref{cor:shadow1}, for all $w\in\Gamma$ and $x\in N(w)$, since $\mho(x)\cap \mho(w)\neq\emptyset$,
		\begin{equation*}
			\diam A_{w}\cup A_x \le \diam A_{\mho(w)} + \diam A_{\mho(x)} < 2C_6e^{-a(\abs{w}-R)}.
		\end{equation*}
		Hence, 
		\begin{equation}\label{eq:ANwneighbor}
			A_{N(w)}\subseteq B\bigl(A_w, 2C_6 e^{-a(\abs{w}-R)}\bigr).
		\end{equation}
		Similarly, we can show that $\overline{A_{\mho(w)}}\cap A_u\neq\emptyset$ implies
		\begin{equation}\label{eq:Awneighbor}
			A_w\subseteq B\bigl(A_u, C_6 e^{-a\abs{w}}\bigr).
		\end{equation}

		Hence, for each $w\in\Gamma$ with $\abs{w}-N_1\ge \max\{\abs{u},\abs{v}\}$  that satisfies condition~(1), by (\ref{eq:ANwneighbor}), (\ref{eq:Awneighbor}), and (\ref{eq:defN1Harnack}), we always have $A_{N(w)}\subseteq B\bigl(A_u, C_3 e^{-a\abs{u}}\bigr)$. That is, condition~(2) holds for $w$. Hence, for each $w\in\Gamma$  with $\abs{w}-N_1\ge \max\{\abs{u},\abs{v}\}$ satisfying either condition~(1) or~(2) of this lemma, by Lemma~\ref{lem:inShadow}, 
		\begin{equation}\label{neighborInShadow}
			N(w)\subseteq \mho(u).
		\end{equation}

		Denote
		\begin{align}
			S &\= \{y\in\mho(u) : \abs{\abs{y}-\max\{\abs{u},\abs{v}\}-N_1}\le R\}  \qquad\text{and} \label{def:S}\\
			C_2 &\= \#S \cdot \max\{ F(v,y) /F(u,y) : y\in S\}. \label{def:C2}
		\end{align}

		Let $\tau\:\Gamma\to \Gamma$ be from Subsection~\ref{subsct:visualMetric}. Put $w'\=\tau^{\abs{w}-(\max\{\abs{u},\abs{v}\}+N_1)} w$. Then it follows from (\ref{neighborInShadow}) that $N(w')\subseteq S$. By Lemma~\ref{lem:multiplicative}, (\ref{def:C2}), and (\ref{def:S}), we have
		\begin{align*}
			F(v,w) 
			&\le \sum_{x\in N(w')}F(v,x)F(x,w)
			\le \frac{C_2}{\# S} \sum_{x\in N(w')}F(u,x)F(x,w)\\
			&\le \frac{C_2}{\# S}\sum_{x\in N(w')} F(u,w)
			\le  C_2F(u,w).  \qedhere
		\end{align*}
	\end{proof}

\subsection{Green drift \texorpdfstring{$l_G$}{lG}}   \label{subsec:lG}

Before establishing Lemma~\ref{lem:deflG}, we first demonstrate some ergodic property of the shift map $T$ defined by (\ref{def:T}).

We denote the \defn{cylinders} in the space of sample paths $\Omega$ by
\begin{equation}
	[u_0,\dots, u_n] \= \{\omega\in\Omega:Z_0(\omega)=u_0, \dots, Z_n(\omega)=u_n\},
\end{equation}
for $n\in\Z_{>0}$ and $u_0,\dots, u_n\in \Gamma$. Recall that by definition, for each $\omega\in\Omega$, $Z_0(\omega)=o$. Hence, if $u_0=o$, then $[u_0,\dots, u_n]=\emptyset$.

\begin{theorem}\label{thm:mixing}
	Let $(X,f)$ be a dynamical system that satisfies the Assumptions in Subsection~\ref{subsect:expandingDyn} and $\Gamma$ be the tile graph associated with $f$ and a Markov partition $\alpha$ of $(X,f)$ from Theorem~\ref{thm:hyp}. Suppose that the transition probability $P$ on $\Gamma$ satisfies the Assumptions in Section~\ref{sct:Introduction}. Then the shift map $T$ is $\bP$-measure-preserving and mixing. 
\end{theorem}

\begin{proof}
	Let $\Omega$ be the space of all sample paths. By the definition of $T$, for each $m\in\Z_{>0}$ and each cylinder $[u_0,\dots, u_n]\subseteq\Omega$ with $n\in\Z_{>0}$ and $u_0,\dots, u_n\in \Gamma$,
	\begin{equation}
		T^{-m}[u_0,\dots, u_n] = \bigl\{\omega\in\Omega:\sigma^{\abs{Z_m(\omega)}} Z_m(\omega)=u_0, \dots, \sigma^{\abs{Z_m(\omega)}} Z_{n+m}(\omega)=u_n\bigr\}.
	\end{equation}
	Hence, by the Markov property and Assumption~(D) in Section~\ref{sct:Introduction}, for all $k, m\in\Z_{>0}$ and $v_0,\dots, v_m\in \Gamma$,
	\begin{align*}
		& \bP\bigl(T^{-k}[v_0,\dots, v_m]\bigr)\\
		&\qquad= \bP\bigl(\sigma^{\abs{Z_k}}Z_{k}=v_0,\dots, \sigma^{\abs{Z_k}}Z_{k+m}=v_m\bigr)\\
		&\qquad= \sum_{\sigma^{\abs{Z_k}}w_{m-1}=v_{m-1}}\bP\bigl(\sigma^{\abs{Z_k}}Z_{k}=v_0,\dots, \sigma^{\abs{Z_k}}Z_{k+m-2}=v_{m-2}, Z_{k+m-1}=w_{m-1}\bigr) \\ 
		&\qquad\qquad\qquad\qquad\qquad \cdot \sigma_*^{\abs{Z_k}} P(w_{m-1})(\{v_m\})  \\
		&\qquad= \bP\bigl(\sigma^{\abs{Z_k}}Z_{k}=v_0,\dots, \sigma^{\abs{Z_k}}Z_{k+m-1}=v_{m-1}\bigr) P(v_{m-1})(\{v_m\})\\
		&\qquad= \bP\bigl(T^{-k}[v_0,\dots, v_{m-1}]\bigr) \hP(v_{m-1}, v_m).
	\end{align*}
	Applying the equation above recursively, then we get
	\begin{equation*}
		\bP\bigl(T^{-k}[v_0,\dots, v_m]\bigr)
		= \hP(v_0, v_1)\cdots \hP(v_{m-1}, v_m)
		= \bP([v_0,\dots, v_m]).
	\end{equation*}
	That is, $T$ is $\bP$-measure-preserving.
	
	To show that $T$ is mixing, it suffices to show that for all $n,m\in\Z_{>0}$, $u_0,\dots,u_n\in\Gamma$, and $v_0,\dots,v_m\in\Gamma$, 
	\begin{equation}\label{eq:mixing}
		\lim_{k\to+\infty} \bP\bigl([u_0,\dots, u_n]\cap T^{-k}[v_0,\dots, v_m]\bigr) =
		\bP([u_0,\dots, u_n])\bP([v_0,\dots, v_m])
	\end{equation}
	We may assume that $v_0=o$ and $\bP([v_0,\dots, v_m])=\hP(v_0,v_1)\cdots\hP(v_{m-1},v_{m})>0$ because otherwise, both sides of (\ref{eq:mixing}) are zero. Then for each integer $k>n$, by the Markov property and Assumption~(D) in Section~\ref{sct:Introduction},
	\begin{align*}
		&\bP\bigl([u_0,\dots, u_n]\cap T^{-k}[v_0,\dots, v_m]\bigr)\\
		&\qquad= \bP\bigl(Z_0=u_0,\dots, Z_n=u_n, \sigma^{|Z_k|}Z_{k}=v_0,\dots, \sigma^{|Z_k|}Z_{k+m}=v_m\bigr)\\
		&\qquad= \sum_{\sigma^{|Z_k|}w_{m-1}=v_{m-1}} \bP\bigl(Z_0=u_0,\dots, Z_n=u_n, \sigma^{|Z_k|}Z_{k}=v_0,\dots, \\
		&\qquad\qquad\qquad \sigma^{|Z_k|}Z_{k+m-2}=v_{m-2},  Z_{k+m-1}=w_{m-1}\bigr) \sigma_*^{|Z_k|}P(w_{m-1})(\{v_m\})\\
		&\qquad= \bP\bigl(Z_0=u_0,\dots, Z_n=u_n, \sigma^{|Z_k|}Z_{k}=v_0,\dots, \sigma^{|Z_k|}Z_{k+m-1}=v_{m-1}\bigr) \\
		&\qquad\qquad \cdot P(v_{m-1})(\{v_m\})\\
		&\qquad= \bP\bigl([u_0,\dots, u_n]\cap T^{-k}[v_0,\dots, v_{m-1}]\bigr) \hP(v_{m-1}, v_m).
	\end{align*}
	Hence, by applying the equation above recursively, we have
	\begin{align*}
		\bP\bigl([u_0,\dots, u_n]\cap T^{-k}[v_0,\dots, v_m]\bigr)
		&= \bP([u_0,\dots, u_n]) \hP(v_0,v_1)\cdots \hP(v_{m-1}, v_m)\\
		&= \bP([u_0,\dots, u_n]) \bP([v_0,\dots, v_m]).
	\end{align*}
	This proves (\ref{eq:mixing}) and the theorem follows.
\end{proof}

\begin{proof}[Proof of Lemma~\ref{lem:deflG}]
	Let $N_0\in\Z_{>0}$ be the constant in Lemma~\ref{lem:locIsom}. For all $n,m\in\Z_{>0}$ and each sample path $\omega\in\Omega$, by Lemma~\ref{lem:locIsom} and Assumption~(D) in Section~\ref{sct:Introduction},
	\begin{equation*}
		F\bigl(\sigma^{\abs{Z_{n}}}Z_{N_0+n},\sigma^{\abs{Z_{n}}}Z_{N_0+n+m}\bigr) = F(Z_{N_0+n},Z_{N_0+n+m}).
	\end{equation*}
	Hence, by Lemma~\ref{lem:multiplicative} and (\ref{def:approxDim}),
	\begin{align*}
		\tg_n(\omega)+\tg_m(T^n\omega) 
		&= -\log (F(Z_{N_0},Z_{N_0+n})F(Z_{N_0+n},Z_{N_0+n+m}))\\
		&\ge -\log F(Z_{N_0}, Z_{N_0+n+m}) 
		= \tg_{n+m}(\omega).
	\end{align*}
	By Kingman's subadditive ergodic theorem, since $T$ is ergodic by Theorem~\ref{thm:mixing}, $\tg_n/n$ converges almost surely to some constant. Hence, we can define $l_G$ as the almost-sure limit of $\tg_n/n$.

	\smallskip
	\emph{Claim}. For each $N\in\Z_{>0}$, almost surely, there is an integer $n>N$ such that $Z_{\infty}\in A_{Z_{n}}$. 
	\smallskip

	To verify the claim, we denote, for each $u\in\Gamma$, that
	\begin{equation}\label{def:S-1}
		S(u) \= \bigl\{v\in\mho(u): B\bigl(A_{\mho(v)}, e^{-a\abs{v}}\bigr) \subseteq A_{u}\bigr\}.
	\end{equation}
	Then for each $v\in S(u)$, and each $w\in \mho(v)$,
	\begin{equation*}
		B\bigl(A_{\mho(w)}, e^{-a\abs{w}}\bigr) \subseteq B\bigl(A_{\mho(v)}, e^{-a\abs{v}}\bigr) \subseteq A_{u}.
	\end{equation*}
	That is,
	\begin{equation}\label{eq:shadowS}
		S(u) \supseteq \bigcup_{v\in S(u)} \mho(v).
	\end{equation}

	Assume that $u\in\Gamma$. By the properties of Markov partitions, the interior of $A_{u}$ is non-empty. Hence, for each $\xi \in \inter A_{u}$, there is a constant $\delta>0$ such that
	\begin{equation*}
		B(\xi,\delta)\subseteq \inter A_{u}.
	\end{equation*}
	By Lemma~\ref{lem:inShadow}, for each $v\in \Gamma$ with $\xi\in A_v$ and $\abs{v}$ sufficiently large, we have $v\in\mho(u)$. Besides, by Lemma~\ref{cor:shadow1}, if $\abs{v}$ is sufficiently large, then $B\bigl(A_{\mho(v)}, e^{-a\abs{v}}\bigr) \subseteq \inter A_{u}$. It follows that for $v\in \Gamma$ with $\xi\in A_v$ and $\abs{v}$ sufficiently large, $v \in S(u)$. Thus, $S(u)$ is nonempty.

	Let $N_0 \in \Z_{>0}$ be the constant in Lemma~\ref{lem:locIsom}. For each $u\in\Gamma$ with $\abs{u}\le N_0$, choose $v(u)\in S(u)$. By (\ref{eq:shadowS}), for each $v\in S(u)$, $\mho(v)\subseteq S(u)$. By (\ref{def:S-1}), $A_v\subseteq A_u$. Hence, it follows that, after possibly replacing $v(u)$ by vertices in $\mho(v(u))$, we may assume that there is a sufficiently large integer $n_0$ such that $\hP^{(n_0)}(u,v(u))>0$ for each $u\in\Gamma$ with $\abs{u}\le N_0$. Put
	\begin{equation*}
		\epsilon_1 \= \min_{u\in\Gamma, \abs{u}\le N_0}\bigl\{ \hP^{(n_0)}(u,v(u)) \bigr\}.
	\end{equation*}
	It follows from the local similarity of $\Gamma$, i.e., Lemma~\ref{lem:locIsom}, that for each $u\in\Gamma$, there is a vertex 
	\begin{equation}\label{eq:propV}
		v(u) \in S(u)
	\end{equation}
	such that
	\begin{equation}\label{eq:propEps1}
		\hP^{(n_0)}(u,v(u))\ge \epsilon_1.
	\end{equation}
	In fact, we can choose
	\begin{equation*}
		v(u) \= \bigl( \bigl(\sigma^{\abs{u}-N_0} \bigr)\big|_{\mho(u)} \bigr)^{-1} v \bigl( \sigma^{\abs{u}-N_0} u \bigr)
	\end{equation*}
	for all $u\in\Gamma$ with $\abs{u}>N_0$.

	Next, we show that
	\begin{equation}\label{eq:noProbToBound}
		\bP(\exists n,m\in\Z_{\ge 0}, m>n>N, Z_{m} \in S(Z_{n}))=1.
	\end{equation}
	In fact, for each $n\in\Z_{\ge 0}$, by (\ref{eq:propEps1}),
	\begin{equation*}
		\bP(Z_{n+n_0}=v(Z_n)\,|\,\sigma(Z_n))\ge \epsilon_1.
	\end{equation*}
	Here $\sigma(Z_n)$ is the $\sigma$-field generated by $Z_n$. By (\ref{eq:propV}),
	\begin{equation*}
		\bP(Z_{n+n_0}\in S(Z_n)\,|\,\sigma(Z_n))\ge \epsilon_1.
	\end{equation*}
	Hence, by induction, for each $m\in\Z_{>0}$,
	\begin{equation*}
		\bP(\forall k\in\Z_{\ge 0} \text{ with } k<m, Z_{N+(k+1)n_0}\not\in S(Z_{N+kn_0}) ) \le (1- \epsilon_1)^m.
	\end{equation*}
	As $m$ tends to $+\infty$, we have
	\begin{equation*}
		\bP(\forall k\in\Z_{\ge 0}, Z_{N+(k+1)n_0}\not\in S(Z_{N+kn_0}) ) =0.
	\end{equation*}
	Hence, we have proved (\ref{eq:noProbToBound}). Note that by (\ref{def:S-1}), for all integers $m>n>N$, $Z_m(\omega)\in S(Z_n(\omega))$ implies $Z_{\infty}(\omega)\in \overline{A_{\mho(Z_{m}(\omega))}} \subseteq A_{Z_n(\omega)}$. Hence,
	\begin{equation}
		\bP(\exists n\in\Z_{\ge 0}, n>N, Z_{\infty} \in A_{Z_{n}})=1.
	\end{equation}
	The claim follows immediately.

	\smallskip

	We return to the proof of Lemma~\ref{lem:deflG}. Let $N_1$ denote the constant from Lemma~\ref{lem:Harnack}. By the claim, for a.e.~$\omega\in \Omega$, there exists an integer $N>N_0$ such that $Z_{\infty}\in A_{Z_{N}}$. For each integer $n>N+N_1$, since $Z_{\infty}\in\overline{A_{\mho(Z_n)}}\cap A_{Z_{N}}\cap A_{o}$, by applying Lemma~\ref{lem:Harnack} with $(u,v)$ taking the values of $(Z_n,o)$ and $(o,Z_n)$, we have
	\begin{equation*}
		C_2^{-1}F(Z_{N},Z_n) \le F(o,Z_n) \le C_2 F(Z_{N},Z_n).
	\end{equation*}
	It follows that for a.e.~$\omega\in\Omega$,
	\begin{equation*}
		\lim_{n\to+\infty }\frac{\tg_n}{n} =  -\lim_{n\to+\infty }\frac{\log F(Z_{N},Z_n)}{n} = - \lim_{n\to+\infty }\frac{\log F(o,Z_n)}{n} = \lim_{n\to+\infty } \frac{g_n}{n} = l_G .
	\end{equation*}

	The lemma is now established.
\end{proof}

\subsection{Proof of Theorem~\ref{thm:main2}}

	To obtain the dimension of the harmonic measure, we need a lemma about the shadow on the Martin boundary of $(\Gamma,P)$. Recall that the map $\Phi$ defined in Theorem~\ref{thm:surj} pushes forward the harmonic measure on the Martin boundary to the Gromov boundary of $\Gamma$. 

	\begin{lemma}\label{lem:shadowKernel}
		Let $C_3>0$ be the constant in Lemma~\ref{lem:inShadow}. There is a constant $0<C_5<1$ such that for all $u\in\Gamma$ and
		\begin{equation}\label{eq:assumeXiU}
			\xi\in \Phi^{-1} B\bigl(A_u,C_3 e^{-a\abs{u}}/2\bigr),
		\end{equation}
		we have
		\begin{equation*}
			K(u,\xi) \ge C_5\Bigl(\sum_{v\in N(u)} F(o,v)\Bigr)^{-1}.
		\end{equation*}
	\end{lemma}
	\begin{proof}
		Fix $u\in\Gamma$. Let $N_1$ be the constant in Lemma~\ref{lem:Harnack}. Let $\{w_n\}_{n=1}^{+\infty}$ be a sequence in $\Gamma$ that converges to $\xi$ in the Martin boundary. By the definition of $\Phi$, $w_n$ converges to $\Phi(\xi)$ in the Gromov boundary. Hence, there is an integer $N>0$ such that for each integer $n>N$, $\abs{w_n}\ge \abs{u}+2R+N_1$ and
		\begin{equation}
			A_{w_n}\subseteq B\bigl(\Phi(\xi), 2^{-1} C_3e^{-a\abs{u}} -2C_6e^{-a(\abs{w_n}-R)}\bigr),
		\end{equation}
		where $C_6>0$ is the constant from Corollary~\ref{cor:shadow1}.
		Thus, we can apply (\ref{eq:ANwneighbor}) in Lemma~\ref{lem:Harnack} and (\ref{eq:assumeXiU}), and have
		\begin{equation}
			A_{N(w_n)} \subseteq B\bigl(\Phi(\xi), C_3e^{-a\abs{u}}/2\bigr)\subseteq B\bigl(A_u, C_3e^{-a\abs{u}}\bigr).
		\end{equation}
		Hence, $w_n$ satisfies condition~(2) in Lemma~\ref{lem:Harnack}. It follows from (\ref{def:K}), Lemmas~\ref{lem:multiplicative}, and~\ref{lem:Harnack} that
		\begin{align*}
			K(u,w_n) &= F(u,w_n) / F(o,w_n)
			\ge \frac{F(u,w_n)}{\sum_{v\in N(u)}F(o,v)F(v,w_n)}\\
			&\ge \frac{F(u,w_n)}{\sum_{v\in N(u)}F(o,v)C_2(u,v) F(u,w_n)}
			\ge  \Bigl(\min_{v\in N(u)} \{ C_2(u,v) \} \sum_{v\in N(u)} F(o,v)\Bigr)^{-1}.
		\end{align*}

		Recall that by (\ref{def:C2}), it follows from the local similarity, i.e., Lemma~\ref{lem:locIsom} and Assumption~(D) in Section~\ref{sct:Introduction}, that
		\begin{align*}
			C_5 \= \min&\{C_2(u,v):v\in N(u), u\in\Gamma, \abs{u}\le N_0\} \\
			       = \min&\{C_2(u,v):v\in N(u), u\in\Gamma\}. 
		\end{align*}
		Hence, for each $u\in\Gamma$,
		\begin{equation*}
			K(u,w_n) \ge \frac{C_5}{\sum_{v\in N(u)} F(o,v)}.
		\end{equation*}
		As the limit of $K(u,w_n)$, therefore, $K(u,\xi)$ also satisfies this inequality. 
	\end{proof}

	Put
	\begin{equation*}
		\epsilon_0 \= \min_{\hP(u,v)>0, \abs{u}\le N_0}\{\hP(u,v)\}.
	\end{equation*}
	It follows from the local similarity, i.e., Lemma~\ref{lem:locIsom} and Assumption~(D) in Section~\ref{sct:Introduction}, that for each $u,v\in\Gamma$, $\hP^{(k)}(u,v)>0$ implies
	\begin{equation}\label{def:eps0}
		\hP^{(k)}(u,v)\ge \epsilon_0^k.
	\end{equation}
	
	Now we can begin the calculation of the dimension of the harmonic measure.

	\begin{proof}[Proof of Theorem~\ref{thm:main2}]

		Let $C_6>0, C_1>0, C_3>0$ and $C_5>0$ be the constants in Corollary~\ref{cor:shadow1} and Lemmas~\ref{lem:shadow1} and~\ref{lem:shadowKernel}. 

		\smallskip

		We claim that there is a constant $k_1\in\Z_{>0}$ such that for each $u\in \Gamma$, there is a vertex $v_1=v_1(u)\in \Gamma$ with $\hP^{(k_1)}(u,v_1)>0$ such that
		\begin{equation} \label{eq:defv1}
			B\bigl(A_{\mho(v_1(u))}, C_3 e^{-a\abs{u}}/4 \bigr)\subseteq B\bigl(A_u,C_3 e^{-a\abs{u}}/2\bigr).
		\end{equation}

		In fact, choose $k_1\in\Z_{>0}$ such that $C_1e^{-ak_1} \le C_3/4$. For each $u\in\Gamma$, there is a vertex $v_1\in\Gamma$ with $A_{v_1}\subseteq A_u$ and $\abs{v_1}=\abs{u}+k_1$. By Assumption~(C) in Section~\ref{sct:Introduction}, $\hP^{(k_1)}(u,v_1)>0$. By Lemma~\ref{lem:shadow1},
		\begin{equation*}
			A_{\mho(v_1)}\subseteq B\bigl(A_{v_1}, C_1e^{-a\abs{v_1}}\bigr)\subseteq B\bigl(A_u, C_3 e^{-a\abs{u}}/4\bigr).
		\end{equation*}
		This proves the claim.

		\smallskip

		Put
		\begin{equation}
			k_0 \= \max \bigl\{ \min_{w\in \mho(u)\cap \mho(v)}\{\abs{w}-\abs{u}\} : u,v\in\Gamma,\abs{u}\le N_0+R, v\in N(u)\bigr\}.
		\end{equation}
		It follows from the local similarity of $\Gamma$, i.e., Lemma~\ref{lem:locIsom}, that the condition $\abs{u}\le N_0+R$ can be removed from the definition of $k_0$, i.e.,
		\begin{equation}\label{def:k0}
			k_0 = \max \bigl\{ \min_{w\in \mho(u)\cap \mho(v)}\{\abs{w}-\abs{u}\} : u,v\in\Gamma, v\in N(u) \bigr\}.
		\end{equation}
		That is to say, for all $u,v\in\Gamma$ with $v\in N(u)$, there is a vertex $w\in \mho(u)\cap \mho(v)$ such that $\abs{w}\le \abs{u}+k_0$. Moreover, by Assumption~(B) in Section~\ref{sct:Introduction}, there exists an integer $0\le i\le k_0$ such that $\hP^{(i)}(u, w)>0$. Hence, there is also a vertex $w'\in \mho(w)\subseteq\mho(u)\cap \mho(v)$ such that 
		\begin{equation}\label{eq:propk0}
			\hP^{(k_0)}(u,w')>0.
		\end{equation}

		For all $u,v,w\in\Gamma$ such that $v\in N(u)$, $w\in\mho(u)\cap \mho(v)$, and $\hP^{(k_0)}(u, w)>0$, we have $\abs{w} \le \abs{u}+k_0R \le \abs{v}+k_0R+R$. Hence, by Assumption~(B) in Section~\ref{sct:Introduction}, there exists $i\in\Z$ with $0\le i\le k_0R+R$ such that $\hP^{(i)}(v,w)>0$. By the definition of $F$ and (\ref{def:eps0}), for such $u,v,w\in\Gamma$,
		\begin{equation}\label{eq:GreenLowerBound}
			F(v,w)\ge \hP^{(i)}(v,w)\ge \epsilon_0^{i} \ge \epsilon_0^{k_0R+R}.
		\end{equation}

		For each $u\in\Gamma$, choose a vertex $v(u)\in N(u)$ such that
		\begin{equation}\label{eq:defv}
			F(o,v(u))=\sup_{w\in N(u)}F(o,w).
		\end{equation}
		By (\ref{eq:propk0}), there is a vertex $v_0(u)\in \Gamma$ such that $\hP^{(k_0)}(u,v_0(u))>0$ and $v_0(u)\in\mho(v(u))$. Let $v_1(u)$ be chosen as in the claim.

		Consider the events 
		\begin{equation*}
			B_n \= \{ \omega\in\Omega : Z_{n+k_0}=v_0(Z_n)\}, \quad n\ge N_0,
		\end{equation*}
		\begin{equation*}
			B_n^1 \= \{ \omega\in\Omega : Z_{n+k_1} = v_1(Z_n) \}, \quad n\ge N_0.
		\end{equation*}
		Then for each integer $n\ge N_0$, by the definitions of $v_0$ and $v_1$, we have
		\begin{equation*}
			\bP(B_n \,|\,\sigma(Z_n)) >0 \quad \text{ and } \quad \bP \bigl(B^1_n \,\big|\,\sigma(Z_n) \bigr) >0.
		\end{equation*}
		By (\ref{def:eps0}),
		\begin{equation*}
			\bP(B_n \,|\,\sigma(Z_n)) \ge \epsilon_0^{k_0} \quad \text{ and } \quad \bP \bigl(B^1_n \,\big|\,\sigma(Z_n) \bigr) \ge \epsilon_0^{k_1}.
		\end{equation*}
		Hence, by the Markov property,
		\begin{equation}\label{ieq:BcapB1PositiveProb}
			\bP(B_n\cap B^1_{n+k_0}\,|\,\sigma(Z_n)) \ge \epsilon_0^{k_0+k_1}.
		\end{equation}

		By the assumption~(B) in Section~\ref{sct:Introduction}, for each sample path $\{Z_n(\omega)\}_{n\in\Z_{\ge 0}}$ and each pair of vertices $w_1,w_2\in\Gamma$ on the same level, (i.e., $\abs{w_1}=\abs{w_2}$), if $Z_n(\omega)=w_1$ for some $n\in\Z_\ge 0$, then $Z_{n'}(\omega)=w_2$ cannot hold for all $n'\in\Z_\ge 0$. Hence, for each $w,v\in\Gamma$, by the definition (\ref{def:N}) of $N(\cdot)$,
		\begin{equation}\label{eq:kernelUpperBound0}
			\sum_{w'\in N(v)} F(w,w') \le \sum_{k=\abs{v}-R}^{\abs{v}+R} \sum_{v\in N(w'), \abs{w'}=v} F(w,w')\le \sum_{k=\abs{v}-R}^{\abs{v}+R}1 = 2R+1.
		\end{equation}

		Hence, for each integer $n\ge N_0$, each $\omega\in B_{n+k_0}^1$, and each $\xi\in B\bigl(Z_{\infty}, C_3e^{-a\abs{Z_{n+k_0}}}/4\bigr)$, by (\ref{eq:kernelUpperBound0}) and Lemmas~\ref{lem:shadowKernel} and~\ref{lem:multiplicative},
		\begin{align}\label{eq:kernelUpperBound1}
			K(Z_{n+k_0}, \xi)^{-1} &\le C_5^{-1}\sum_{w\in N(Z_{n+k_0})}F(o,w)\notag\\
			&\lesssim \sum_{w\in N(Z_{n})}\sum_{w'\in N(Z_{n+k_0})} F(o,w)F(w,w') 
			\lesssim \sum_{w\in N(Z_{n})} F(o,w).
		\end{align}
		For each integer $n\ge N_0$ and each $\omega\in B_n$, by Lemma~\ref{lem:multiplicative}, (\ref{eq:defv}), and (\ref{eq:GreenLowerBound}),
		\begin{align*}
			\sum_{w\in N(Z_n)} F(o,w) 
			&\le  (\# N(Z_n))F(o,v(Z_n)) \\
			&\le \frac{\# N(Z_n)}{F(v(Z_n),Z_{n+k_0})}F(o,Z_{n+k_0}) 
			\le \frac{\# N(Z_n)}{\epsilon_0^{k_0R+R}}F(o,Z_{n+k_0}).
		\end{align*}
		Combining the inequality above with (\ref{eq:kernelUpperBound1}), we deduce that, for all integer $n\ge N_0$, $\omega\in B_n\cap B^1_{n+k_0}$, and $\xi\in \Phi^{-1} \bigl( B\bigl(Z_{\infty}(\omega), C_3e^{-a\abs{Z_{n+k_0}}}/4\bigr) \bigr)$,
		\begin{equation}\label{eq:kernelUpperBound2}
			K(Z_{n+k_0}, \xi) \gtrsim F(o,Z_{n+k_0})^{-1}.
		\end{equation}

		Recall that $\nu = \Phi_*\nu^{\Mb\Gamma}$ is the push-forward of the harmonic measure from the Martin boundary to $X$ by the map $\Phi$ provided in Theorem~\ref{thm:surj}. 
		Recall the formula (\ref{eq:harmonicBaseChange}) about the change of basepoint of the harmonic measure:
		\begin{equation*}
			K(u,\cdot) = \frac{\mathrm{d}\nu_u^{\Mb\Gamma}}{\mathrm{d}\nu^{\Mb\Gamma}}.
		\end{equation*}
		For each integer $n\ge N_0$ and each sample path $\omega\in B_n\cap B^1_{n+k_0}$, put $u\=Z_n(\omega)$. Then by (\ref{eq:kernelUpperBound2}), with $R_{n,\omega} \= C_3 e^{-a\abs{Z_{n+k_0}}}/4$, we have
		\begin{equation}\label{eq:hMeasureUpperBound}
			1 
			= \int_{\partial_M\Gamma} \!\, \mathrm{d} \nu_u^{\Mb\Gamma}(\xi)  
			\ge \int_{\Phi^{-1}\! B (Z_{\infty},	R_{n,\omega}  )}
				K(u,\xi) \, \mathrm{d} \nu^{\Mb\Gamma}(\xi)   
			\gtrsim \frac{\nu (B (	Z_{\infty}, R_{n,\omega}  ) )	}{F(o,Z_{n+k_0})}.  
		\end{equation}

		By (\ref{ieq:BcapB1PositiveProb}) and inductively by the Markov property, 
		\begin{equation*}
			\bP \bigl(\forall 0\le m\le n, \omega\not\in B_{m(k_0+k_1)}\cap B^1_{m(k_0+k_1)+k_0} \bigr) \le \bigl(1-\epsilon_0^{k_0+k_1} \bigr)^n.
		\end{equation*}
		Letting $n\to+\infty$, we get
		\begin{equation*}
			\bP\bigl(B_m\cap B^1_{m+k_0}~\text{i.o.}\bigr) 
			= \lim_{n\to+\infty} \bP\biggl(\bigcup_{m=n}^{+\infty}B_m\cap B^1_{m+k_0}\biggr) 
			= 1.
		\end{equation*}
		That is, $\omega\in B_n\cap B^1_{n+k_0}$ infinitely often for a.e.~$\omega\in\Omega$.
		
		Hence, for a.e.~$\omega\in\Omega$, there is an infinite sequence $\{n_i\}_{i\in\Z_{>0}}$ such that $\omega\in B_{n_i} \cap B^1_{n_i+k_0}$. Thus, for a.e.~$\omega\in \Omega$, by (\ref{eq:hMeasureUpperBound}) and Lemma~\ref{lem:deflG},
		\begin{equation}\label{eq:hMdim1}
			\lim_{i\to+\infty}\frac{-\log(\nu(B(Z_{\infty}, R_{n_i,\omega} ))) }{n_i}
			\ge \lim_{i\to+\infty}\frac{-\log F(o,Z_{n_i+k_0})}{n_i}
			= l_G.
		\end{equation}

		Conversely, by Corollary~\ref{cor:shadow1}, if $\xi\in \Phi^{-1}\overline{A_{\mho(u)}}$, then $A_{\mho(u)}\subseteq B\bigl(\xi, C_6 e^{-a\abs{u}}\bigr)$. Thus,
		\begin{equation*}
				\nu\bigl(B\bigl(\xi,  C_6 e^{-a\abs{u}}\bigr)\bigr)
				\ge \nu(A_{\mho(u)})
				= \nu^{\Mb\Gamma} \bigl(\Phi^{-1}A_{\mho(u)} \bigr)
				= \int \! K(u,\xi)^{-1} \, \mathrm{d} \nu_u^{\Mb\Gamma}(\xi)
				\ge F(o,u).
		\end{equation*}
		Combine with Lemma~\ref{lem:deflG}, we have, for a.e.~$\omega\in\Omega$,  
		\begin{equation}\label{eq:hMdim2}
			\limsup_{n\to+\infty}\frac{ -\log\bigl(\nu\bigl(B\bigl(Z_{\infty},  C_6 e^{-a\abs{Z_n}}\bigr)\bigr)\bigr)}{n} 
			\le	\lim_{n\to+\infty} \frac{-\log(F(o,Z_n))}{n} 
			= l_G.
		\end{equation}

		Since $\nu$ is the projection of the probability measure $\bP$ by the measurable function $Z_{\infty}$, for $\nu$-a.e.~$\xi\in X$, we can find $\omega\in\Omega$ such that $\omega\in B_n\cap B^1_{n+k_0}$, i.o., and $Z_{\infty}(\omega)=\xi$. Hence, combining (\ref{eq:hMdim2}) with (\ref{eq:hMdim1}), for $\nu$-a.e.~$\xi\in X$, choosing such $\omega$, we have
		\begin{equation*}
			\limsup_{n\to+\infty}\frac{ -\log\bigl(\nu\bigl(B\bigl(Z_{\infty}, e^{-a\abs{Z_n}}\bigr)\bigr)\bigr)}{n} = l_G.
		\end{equation*}
		Recall that $\lim_{n\to+\infty} \frac{\abs{Z_n}}{n}=l$. So we have
		\begin{equation*}
			\limsup_{r\to 0} \frac{\log \nu(B(\xi, r))}{\log r} = \frac{l_G}{al}.
		\end{equation*}
		It follows by the properties of fractal dimensions (see for example, \cite[Theorem~8.6.5 and the last paragraph in Subsection~8.4]{PU10}) that the packing dimension of $\nu$ is $\frac{l_G}{al}$. This completes the proof of Theorem~\ref{thm:main2}.
	\end{proof}

\section{Quasi-invariance of the harmonic measure}  \label{sct:quasi-invariance}

In this section, we demonstrate the quasi-invariance of the harmonic measure in Theorem~\ref{thm:main3}. We assume that the dynamical system $(X,f)$ satisfies the Assumptions in Subsection~\ref{subsect:expandingDyn}. Let $\Gamma$ be the tile graph associated with $f$ and a fixed Markov partition $\alpha$ with each $A\in\alpha$ connected so that Theorem~\ref{thm:hyp} can be applied. The tile graph $\Gamma$ is equipped with a natural shift map $\sigma$ defined in Subsection~\ref{subsct:visualMetric}. We study the the random walks on the tile graph $\Gamma$ under the Assumptions in Section~\ref{sct:Introduction}. 

	For each $u\in\Gamma$, put $\nu_u \= \Phi_*\nu_u^{\Mb\Gamma}$. Then we have the following lemma.
	
	\begin{lemma}\label{lem:harm}
		For each vertex $u\in\Gamma$,
		\begin{equation}  \label{eq:harmInduct}
			\nu_u = \sum_{w\in\Gamma} \hP(u,w) \nu_w,
		\end{equation}
		and if $u\neq o$, we have 
		\begin{equation}\label{eq:harmShiftInv}
			\nu_{\sigma u} = f_*\nu_u.
		\end{equation}
	\end{lemma}

	\begin{proof}
		For each sample path $\omega\in\Omega$, we denote by $Z_{\infty}^{X}(\omega)$ the limit of $\{Z_n(\omega)\}$ in $X$.
		For each $u\in\Gamma$, by the definition (\ref{eq:defHarm}) of $\nu_u^{\Mb\Gamma}$ and the definition of $\Phi$ in Theorem~\ref{thm:main1}, for each Borel subset $A\subseteq X$,
		\begin{equation}  \label{eq:defHarmGromov}
			\nu_u(A) = \bP_u \bigl( Z_{\infty}^{X}\in A \bigr).
		\end{equation}
		Hence, for each $u\in\Gamma$, $\nu_u = \sum_{w\in\Gamma} \hP(u,w) \nu_w$, establishing (\ref{eq:harmInduct}).
	
		Let $N_0\in\Z_{>0}$ be the constant from Lemma~\ref{lem:locIsom}. Fix $u\in\Gamma$ with $\abs{u}\ge N_0$. By Assumption~(D) in Section~\ref{sct:Introduction} and Lemma~\ref{lem:locIsom}, for each cylinder $[u_0,\dots, u_n]\subseteq\Omega$ with $n\in\Z_{>0}$ and $u_0=u,u_1,\dots, u_n\in \Gamma$, we have
		\begin{align*}
			\bP_u([u_0,\dots, u_n]) 
			&= \hP(u_0,u_1)\cdots\hP(u_{n-1},u_n) \\
			&= \hP(\sigma u_0,\sigma u_1)\cdots\hP(\sigma u_{n-1},\sigma u_n) 
			= \bP_{\sigma u}([\sigma u_0,\dots, \sigma u_n]).
		\end{align*}
		For each $\omega\in\Omega$, we define $\sigma_* \omega$ as the sample path satisfying $Z_n(\sigma_*\omega) = \sigma Z_n(\omega)$ for each $n\in\Z_{\ge 0}$. Hence, for each measurable subset $A\subseteq \Omega$,
		\begin{equation*}
			\bP_{\sigma u}( A) = \bP_u(\sigma_*\omega \in A).
		\end{equation*}
		By Corollary~\ref{cor:sigmaExtendF} and the definition of $\sigma_*$, $Z_{\infty}^{X}(\sigma_*\omega) = f \bigl( Z_{\infty}^{X}(\omega) \bigr)$. Hence, for each Borel subset $B\subseteq X$,
		\begin{equation*}
			\bP_{\sigma u} \bigl( Z_{\infty}^{X}\in B \bigr) = \bP_u \bigl( Z_{\infty}^{X}(\sigma_*\omega) \in B \bigr).
		\end{equation*} 
		By (\ref{eq:defHarmGromov}), for each Borel subset $B\subseteq X$,
		\begin{equation*}
			\nu_{\sigma u}(B) 
			= \bP_{\sigma u} \bigl( Z_{\infty}^{X}(\omega)\in B \bigr) 
			= \bP_u \bigl( f Z_{\infty}^{X}(\omega)\in B \bigr) 
			= \nu_u \bigl( f^{-1}B \bigr) 
			= (f_*\nu_u)(B).
		\end{equation*}
		That is, $\nu_{\sigma u} = f_*\nu_u$ for all $u\in\Gamma$ with $\abs{u}\ge N_0$.
		
		For $u\in\Gamma$ with $\abs{u}\le N_0$, we prove (\ref{eq:harmShiftInv}) by induction. Suppose that for some integer $n > 0$, (\ref{eq:harmShiftInv}) holds for each $u\in\Gamma$ with $\abs{u}> n$. Fix $u\in\Gamma$ with $\abs{u}= n$. By Assumption~(B) in Section~\ref{sct:Introduction} and the inductive hypothesis, (\ref{eq:harmShiftInv}) holds for all $w'\in\Gamma$ with $\hP(u,w')>0$. By (\ref{eq:harmInduct}),
		\begin{align*}
			\nu_{\sigma u} &= \sum_{w\in\Gamma} \hP(\sigma u,w) \nu_w
			= \sum_{w,w'\in\Gamma, \sigma w'=w} \hP(u,w') \nu_{w}\\
			&= \sum_{w'\in\Gamma} \hP(u,w') \nu_{\sigma w'}
			= \sum_{w'\in\Gamma} \hP(u,w') f_*\nu_{w'}
			= f_*\nu_u.
		\end{align*}
		Therefore, (\ref{eq:harmShiftInv}) holds for all $u\in\Gamma$ with $u\neq o$.
	\end{proof}

	\begin{proof}[Proof of Theorem~\ref{thm:main3}]
		By Lemma~\ref{lem:harm},
		\begin{equation}\label{eq:decomposeHarm}
			f_*\nu = \sum_{w\in\Gamma, \abs{w}=1} \hP(o,w) \nu + \sum_{w\in\Gamma, \abs{w}>1} \hP(o,w) \nu_{\sigma w}.
		\end{equation}
		By Assumption~(C) in Section~\ref{sct:Introduction}, for each $w\in\Gamma$ with $\abs{w}=1$, $\hP(o,w)>0$. Hence, $\nu$ is absolutely continuous with respect to $f_*\nu$ with Radon--Nikodym derivative
		\begin{equation*}
			\frac{\rd \nu}{\rd (f_*\nu)} \leq \Bigl(\sum_{w\in\Gamma, \abs{w}=1} \hP(o,w)\Bigr)^{-1}.
		\end{equation*}
		Conversely, by induction on (\ref{eq:harmInduct}) Lemma~\ref{lem:harm}, for each $n\in\Z_{>0}$, we have
		\begin{equation*}
			\nu = \sum_{w\in\Gamma} \hP^{(n)}(o,w) \nu_w.
		\end{equation*}
		It follows from Assumption~(C) in Section~\ref{sct:Introduction} that for each $w\in\Gamma$, $\hP^{(\abs{w})}(o,w)>0$. Thus, $\nu_w$ is absolutely continuous to $\nu$ with $\frac{\rd \nu_w}{\rd \nu} \leq \frac{1}{ \hP^{(w)}(o,w)}$. Hence, by (\ref{eq:decomposeHarm}), as the sum of some $\nu_w$'s, $f_*\nu$ is absolutely continuous to $\nu$ with Radon--Nikodym derivative
		\begin{equation*}
			\frac{\rd (f_*\nu)}{\rd \nu} \leq \sum_{w\in\Gamma,\hP(o,w)>0}   \frac{\hP(o,w)}{\hP^{(\abs{\sigma w})}(o,\sigma w)}.
		\end{equation*}
		The theorem follows.
	\end{proof}

	\appendix

	\section{The tile graph and visual metrics}

	The main goal of this appendix is to prove Propositions~\ref{prop:hyp}. The proof mainly follows the ideas in \cite[Chapter~3]{HP09} and \cite[Chapters~8 and 10]{BM17}. 
	
	Fix a dynamical system $(X,f)$ satisfying the Assumptions in Subsection~\ref{subsect:expandingDyn} with a Markov partition $\alpha$ such that $\mesh\alpha<\xi$. Recall that the tile graph $\Gamma$ is defined in Subsection~\ref{subsct:visualMetric}, and each vertex $u\in\Gamma$ associates with a subset $A_u\subseteq X$.

	We start with the given metric $\rho$ on $X$, which may not be a visual metric as recalled in Subsection~\ref{subsct:visualMetric}.

	\begin{lemma}\label{lem:expandIter}
		Let $(X,f)$ be a dynamical system satisfying the Assumptions in Subsection~\ref{subsect:expandingDyn} with a Markov partition $\alpha$ such that $\mesh \alpha<\xi$. Let $\Gamma$ be the tile graph associated with $\alpha$. Then for each $u\in\Gamma$ with $\abs{u}\ge 1$, 
		\begin{equation*}
			\diam A_u<\lambda^{-\abs{u}+1}\xi.
		\end{equation*}
	\end{lemma}
	\begin{proof}
		We prove it by induction on $n=\abs{u}\in\Z_{>0}$. The statement is true for $n=1$ since $\mesh \alpha<\xi$. For $n>1$, we assume that if $\abs{u}=n-1$, then $\diam A_u<\lambda^{-n+2}\xi$. We consider $u=u_1\dots u_n\in\Gamma$ with $\abs{u}=n$. Then $fA_u=A_{\sigma u}$ and $\abs{\sigma u}=n-1$. Hence, by the inductive hypothesis, $\diam fA_u<\lambda^{-n+2}\xi$. Since $A_u\subseteq A_{u_1}$, $\diam A_u<\xi$. By Assumption (iii) in Subsection~\ref{subsect:expandingDyn}, for all $x,y\in A_u$,
		\begin{equation*}
			\rho(x,y)<\lambda^{-1}\rho(fx,fy)<\lambda^{-1}\diam fA<\lambda^{-1}\mesh \alpha_{n-1}<\lambda^{-n+1}\xi.
		\end{equation*}
		Therefore, $\diam A_u<\lambda^{-n+1}\xi$. This finishes the inductive argument.
	\end{proof}

	We use the boundary construction of W.~Floyd \cite{Fl80} to prove Propositions~\ref{prop:hyp} as in \cite[Chapter~3]{HP09}. Fixing some constant $a>0$, we define a metric $\trho\:\Gamma\times \Gamma\to \R_{\ge 0}$ as follows. 

	For each set of integers $J=I\cap \Z$ with $I\subseteq \R$ being an interval, we call $\gamma\:J\to \Gamma$ \defn{a path in $\Gamma$} if for each $i\in J$ such that $i+1\in J$, $\gamma(i)$ and $\gamma(i+1)$ are connected by an edge. If $I$ is a finite interval, then we say that $\gamma$ connects $\gamma(\inf J)$ and $\gamma(\sup J)$.

	Then we denote the length of a path $\gamma\: J\to\Gamma$ by 
	\begin{equation*}
		l(\gamma) \= \sum_{j,j+1\in J} e^{-a\max\{\abs{\gamma(j)}, \abs{\gamma(j+1)}\}}.
	\end{equation*}
	We construct a metric $\trho$ on $\Gamma$ by
	\begin{equation}  \label{def:FloydMetric}
		\trho(x,y) \= \inf\{l(\gamma): \gamma\text{ is a path in $\Gamma$ connecting }x,y\},
	\end{equation}
	for $x,y\in \Gamma$. It is easy to verify that $\trho$ is a metric on $\Gamma$. Let $\Gamma\cup \partial_{\trho}\Gamma$ be the metric completion of $\Gamma$ with respect to $\trho$. We aim to show that, for sufficiently small $a>0$, $\partial_{\trho}\Gamma$ is homeomorphic to $X$.

	\begin{rem}
		Although by our definition, $(\Gamma, \trho)$ is not a geodesic metric space, we can still construct a geodesic metric space rough-isometric to it if we need to. In fact, we can add the edges to the vertex set of $\Gamma$ to make it a 1-complex. Then we set the length of each edge $e=\{u,v\}$ by $l(e)\=e^{-a\max\{\abs{u}, \abs{v}\}}$. Then we can verify that the resulting space is geodesic and $(1, C)$-quasi-isometric to $(\Gamma, \trho)$. This construction is exactly the construction of the Floyd boundary for the function $f(n)=e^{-an}$.
	\end{rem}

	If $I=(s,t)$ is not a finite interval, then we denote $\gamma(-\infty) \= \lim_{n\to-\infty} \gamma(n)$ if $s=-\infty$ and the limit exists. We also denote $\gamma(+\infty) \= \lim_{n\to+\infty} \gamma(n)$ if $t=+\infty$ and the limit exists.
	By this way, we also say that $\gamma$ connects $\gamma(\inf (I\cap\Z))$ and $\gamma(\sup (I\cap\Z))$.

	Here is the statement of the main result of the appendix. As an analog of \cite[Theorem~10.1]{BM17} and \cite[Proposition~3.3.9]{HP09}, it proves the hyperbolicity of the tile graph. It also shows the existence of $a$-visual metrics for all sufficiently small $a>0$.
	\begin{prop}\label{prop:hyp}
		Let $(X,f)$ be a dynamical system satisfying the Assumptions in Subsection~\ref{subsect:expandingDyn} with a Markov partition $\alpha$ such that $\mesh\alpha<\xi$. Then the tile graph $\Gamma$ defined in Subsection~\ref{subsct:visualMetric} is Gromov hyperbolic, and the Gromov boundary of $\Gamma$ is naturally homeomorphic to $X$. Moreover, there is a constant $a_0$ such that for each $a\in(0,a_0)$, $\trho$ is an $a$-visual metric, and, if $\rho$ is an $a$-visual metric, then $\rho\asymp\trho$.
	\end{prop}
	
	To prove Proposition~\ref{prop:hyp} at the end of the appendix, we need more preparations. First, we prove that the $\trho$-boundary coincides with $X$.

	\begin{prop}\label{prop:boundaryCoincide}
		Let $(X,f)$ be a dynamical system with constants $\lambda$ and $\xi$ satisfying the Assumptions in Subsection~\ref{subsect:expandingDyn}. Let $\alpha$ be a Markov partition with $\mesh\alpha<\xi$. Let $\Gamma$ be the tile graph associated with $(X,f)$ and $\alpha$ defined in Subsection~\ref{subsct:visualMetric}. Suppose that $0<a<\log\lambda$ and $\trho$ is a metric on $\Gamma$ defined in (\ref{def:FloydMetric}). Then $\partial_{\trho}\Gamma$ is naturally homeomorphic to $X$ by a map $\Psi$. Moreover, for all $\zeta, \eta\in X$, 
		\begin{equation}\label{eq:visualLessGiven}
			\rho(\Psi\zeta, \Psi\eta) \lesssim \trho(\zeta,\eta).
		\end{equation}
	\end{prop}

	By \emph{naturally homeomorphic} we mean that the homeomorphism $\Psi$ from the $\partial_{\trho}\Gamma$ to $X$ satisfies the following property: for every sequence of vertices $\{u_n\}_{n\in\Z_{>0}}$ in $\Gamma$ converging to $\zeta\in\partial_{\trho}\Gamma$, the corresponding sequence of subsets $\{A_{u_n}\}_{n\in\Z_{>0}}$ converges in the sense of Gromov--Hausdorff convergence to a singleton $\{\Psi(\zeta)\}\subseteq X$.

	\begin{proof}
		First, we define the map $\Psi$ as follows. For each $\zeta\in\partial_{\trho}\Gamma$, choose a sequence of vertices $\{u_n\}_{n\in\Z_{>0}}$ in $\Gamma$ converging to $\zeta$. Then it follows from the definition (\ref{def:FloydMetric}) of $\trho$ that $\lim_{n\to+\infty} \abs{u_n}=+\infty$. Hence, by Lemma~\ref{lem:expandIter}, $\diam A_{u_n}\to 0$. By the compactness of $X$, there is a subsequence of $A_{u_n}$ Gromov--Hausdorff converging to a singleton $\{x\}\subseteq X$. We define $\Psi(\zeta)\=x$ and the well-definedness follows from the following claim with $\zeta=\eta$.

		\smallskip

		\emph{Claim.} Suppose that $\{v_n\}_{n\in\Z_{>0}}$ and $\{w_n\}_{n\in\Z_{>0}}$ are two sequences in $\Gamma$ converging to $\zeta,\eta\in\partial_{\trho}\Gamma$, respectively, such that $\{A_{v_n}\}_{n\in\Z_{>0}}$ Gromov--Hausdorff converges to $\{x\}\subseteq X$ and $\{A_{w_n}\}_{n\in\Z_{>0}}$ Gromov--Hausdorff converges to $\{y\}\subseteq X$, then
		$\rho(x, y) \lesssim \trho(\zeta, \eta)$.
		
		\smallskip

		Indeed, without the loss of generality, we may assume that $x\neq y$. For each $n\in\Z_{>0}$ and each path $\gamma\:\{0,\cdots, k\}\to \Gamma$ connecting $w_n$ and $v_n$, since $a<\log \lambda$, by Lemma~\ref{lem:expandIter},
		\begin{equation}\label{eq:dist1}
			l(\gamma) = \sum_{i\in \{0,\dots, k-1\}} e^{-a\max\{\abs{\gamma(i)}, \abs{\gamma(i+1)}\}}  
			\gtrsim \sum_{i=0}^k e^{-a\abs{\gamma(i)}} 
			 \ge \sum_{i=0}^k \lambda^{-\abs{\gamma(i)}} 
			 \gtrsim \sum_{i=0}^k \diam A_{\gamma(i)}.
		\end{equation}
		Since $\gamma(i)$ and $\gamma(i+1)$ are connected by an edge of $\Gamma$, 
		for each $i\in\{0,\cdots, k-1\}$, $A_{\gamma(i)}\cap A_{\gamma(i+1)}\neq\emptyset$. It follows that
		\begin{equation}\label{eq:dist2}
			\rho(A_{v_n}, A_{w_n}) \le \sum_{i=1}^{k-1} \diam A_{\gamma(i)}.
		\end{equation}
		Recall that $\{A_{v_n}\}$ and $\{A_{w_n}\}$ Gromov--Hausdorff converge to $\{x\}$ and $\{y\}$, respectively. Let $n\to+\infty$. It follows from (\ref{eq:dist1}), (\ref{eq:dist2}), and the definition (\ref{def:FloydMetric}) of $\trho$ that
		$\rho(x, y) \lesssim \trho(\zeta, \eta)$.
		The claim follows.

		\smallskip

		By the claim, the map $\Psi$ we have just defined is continuous, and (\ref{eq:visualLessGiven}) holds. Note that $\Gamma$ is locally compact. Hence, $\partial_{\trho}\Gamma$ is compact. Since a continuous bijection between compact Hausdorff spaces is a homeomorphism, it suffices to prove that, $\Psi$ is a bijection. 
		
		To prove the surjectivity, we consider an arbitrary point $x\in X$. For each $n\in\Z_{>0}$, choose $u_n\in\Gamma$ such that $x\in A_{u_n}$. Then it is easy to verify that $\{u_n\}_{n\in\Z_{>0}}$ is a $\trho$-Cauchy sequence whose limit point maps to $x$ by $\Psi$. Hence, $\Psi$ is surjective. 

		To prove the injectivity, we assume that $\{v_n\}_{n\in\Z_{>0}}$ is a sequence in $\Gamma$ converging to $\eta\in\partial_{\trho}\Gamma$ such that $A_{v_n}$ Gromov--Hausdorff converges to $\{x\}\subseteq X$. Then for each $N\in\Z_{>0}$, there exists an integer $M>0$ such that for each integer $m>M$, 
		\begin{equation}\label{eq:converge}
			A_{v_m} \subseteq \bigcup_{v\in\Gamma, x\in A_v, |v|=N} A_v,
		\end{equation}
		since the interior of the latter set contains $x$.
		For each $n\in\Z_{>0}$, choose $u_n\in\Gamma$ such that $\abs{u_n}=n$ and $x\in A_{u_n}$. It is easy to verify that $\{u_n\}_{n\in\Z_{>0}}$ is a $\trho$-Cauchy sequence whose limit point maps to $x$ by $\Psi$. We denote the limit of $\{u_n\}$ by $\zeta\in\partial_{\trho}\Gamma$.

		It suffices to show that $\zeta=\eta$. For each integer $N>0$, choose an integer $M>0$ such that (\ref{eq:converge}) holds for each integer $m>M$. For each integer $m>M$, we choose a point $z\in A_{v_m}$ and choose for each integer $N\le k\le m$ a vertex $w_k$ such that $\abs{w_k}=k$ and $z\in A_{w_k}$. In particular, by (\ref{eq:converge}), we can assume that $x\in A_{w_N}$ and $w_m=v_m$. Then the sequence $v_m=w_m,w_{m-1},\dots,w_N,u_n,u_{n+1},\dots,$ produces a path in $\Gamma$ connecting $v_m$ and $\zeta$.
		Hence, by the definition (\ref{def:FloydMetric}) of $\trho$,
		\begin{equation*}
			\trho(v_m, \zeta) \le  \sum_{N\le i< m} e^{-a(i+1)} + e^{-aN} + \sum_{i\ge N} e^{-a(i+1)} < \frac{2e^{-aN}}{1-e^{-a}}.
		\end{equation*}
		As $N\to+\infty$, we have $\lim_{m\to+\infty} \trho(v_m, \zeta) = 0$. Hence, as the limit of $\{v_n\}_{n\in\Z_{>0}}$, $\eta=\zeta$. The injectivity follows.
	\end{proof}

	By Proposition~\ref{prop:boundaryCoincide}, we identify $X$ with $\partial_{\trho} \Gamma$ and regard the metric $\trho$ as defined on $X$. By the naturality of the identification, if a sequence $\{u_n\}$ in $\Gamma$ converges to $x\in X$, then $A_{\sigma u_n}=f A_{u_n}$ Gromov--Hausdorff converges to $\{fx\}$. Hence, $\{\sigma u_n\}$ converges to $fx$, which establishes the following corollary:
	\begin{cor} \label{cor:sigmaExtendF}
		Let $(X,f)$, $\alpha$, $\Gamma$, and $a$ be from Proposition~\ref{prop:boundaryCoincide}. Then under the identification in Proposition~\ref{prop:boundaryCoincide}, the map $\sigma\:\Gamma\to\Gamma$ extends to $f$ on the boundary $\partial_{\trho}\Gamma\cong X$.
	\end{cor}

	Roughly speaking, the following proposition says that when $X$ is equipped with the metric $\trho$, $f$ is locally a similarity.

	\begin{prop}\label{prop:visual}
		Let $(X,f)$, $\alpha$, $\Gamma$, $a$, $\xi$, and $\lambda$ satisfy the assumptions in Proposition~\ref{prop:boundaryCoincide}. Let $\trho$ be the metric defined in (\ref{def:FloydMetric}) but regarded as on $X$ under the identification in Proposition~\ref{prop:boundaryCoincide}. Then there is a constant $\xi'>0$ such that for all $x,y\in X$ with 
		$\trho(x,y) < \xi'$, we have $\trho(fx,fy) = e^a\trho(x,y)$.
	\end{prop}
	\begin{proof}
		Let $C=C(\lesssim)$ be the constant in (\ref{eq:visualLessGiven}). Put
		\begin{equation}\label{eq:assumeDistXY}
			\xi'\=\min\bigl\{e^{-2a}, \xi/(2C)\bigr\}.
		\end{equation}
		
		Fix $x,y\in X$ such that $\trho(x,y) < \xi'$. For each $\epsilon>0$, by the definition (\ref{def:FloydMetric}) of $\trho$, we can choose $\gamma\: \Z\to \Gamma$ be a path connecting $x$ and $y$ such that
		\begin{equation} \label{eq:distxy}
			\trho(x,y)+\epsilon > l(\gamma)=\sum_{i\in\Z} e^{-a\max\{\abs{\gamma(i)}, \abs{\gamma(i+1)}\}}.
		\end{equation}
		Consider $\sigma_*\gamma \: i\in\Z\mapsto \sigma\gamma(i)\in\Gamma$. By Corollary~\ref{cor:sigmaExtendF}, it connects $fx$ and $fy$. By the definition (\ref{def:FloydMetric}) of $\trho$,
		\begin{align*}
			\trho(fx, fy)
			\le l(\sigma_*\gamma)
			&= \sum_{i\in\Z} e^{-a\max\{\abs{\sigma_*\gamma(i)}, \abs{\sigma_*\gamma(i+1)}\}}\\
		    &= e^a\sum_{i\in\Z} e^{-a\max\{\abs{\gamma(i)}, \abs{\gamma(i+1)}\}}
			=e^al(\gamma)
			<e^{a} (\trho(x,y)+\epsilon).
		\end{align*}
		Letting $\epsilon\to 0$, we have
		\begin{equation}\label{eq:visualLess}
			\trho(fx, fy)\le e^a\trho(x,y).
		\end{equation}

		For the converse inequality, for each $\epsilon>0$ such that 
		\begin{equation*}
			\epsilon < \xi'-\trho(x,y),
		\end{equation*}
		by (\ref{eq:visualLess}) and (\ref{eq:assumeDistXY}), 
		\begin{equation}\label{eq:epsSmall}
			\epsilon < e^{-2a}-\trho(x,y) \le e^{-a}(e^{-a}-\trho(fx,fy)).
		\end{equation}
		By the definition (\ref{def:FloydMetric}) of $\trho$, we can choose $\gamma\: \Z\to \Gamma$ as a path connecting $fx$ and $fy$ such that
		\begin{equation} \label{eq:distFxFy}
			l(\gamma)<\trho(fx,fy)+\epsilon.
		\end{equation}
		Then by (\ref{eq:epsSmall}), 
		\begin{equation*} 
			l(\gamma)< e^{-a}.
		\end{equation*}
		Hence, it is easy to show by the definition (\ref{def:FloydMetric}) of $\trho$ that the image of $\gamma$ avoids $o$. 
		We can choose a path $\gamma'$ as a lift of $\gamma$ by $\sigma^{-1}$ starting from $x$ which connects $x$ with some point $y'\in X$. Then $fy'=fy$ and $l(\gamma)=e^a l(\gamma')$. By the definition (\ref{def:FloydMetric}) of $\trho$, (\ref{eq:visualLess}), and (\ref{eq:distFxFy}),
		\begin{equation}\label{eq:distXYprime}
			\trho(x,y')\le l(\gamma')=e^{-a}l(\gamma)<e^{-a}(\trho(fx, fy)+\epsilon).
		\end{equation}
		Recall that $\trho(x,y)<\xi'-\epsilon$. By (\ref{eq:visualLess}) and (\ref{eq:distXYprime}),
		\begin{equation*}
			\trho(y,y') \le \trho(x,y)+\trho(x,y') \le 2\trho(x,y) + e^{-a}\epsilon<2\xi'.
		\end{equation*}
		By Proposition~\ref{prop:boundaryCoincide} and (\ref{eq:assumeDistXY}), $\rho(y,y') \le 2C\xi' \le\xi$.
		It follows from Assumption (iii) in Subsection~\ref{subsect:expandingDyn} that $fy'=fy$ implies $y'=y$. By (\ref{eq:distXYprime}),
		\begin{equation*}
			\trho(x,y) = \trho(x,y')
			<e^{-a}(\trho(fx,fy)+\epsilon).
		\end{equation*}
		As $\epsilon\to 0$, we have $\trho(fx, fy)\ge e^a\trho(x,y)$. Combining the inequality above with (\ref{eq:visualLess}), the proposition follows.
	\end{proof}

	As an immediate corollary, the dynamical system $f$ on the metric space $(X,\trho)$ is expanding as well as on the original metric space $(X,\rho)$.
	\begin{cor} \label{cor:visualExpanding}
		Under the notations and the assumptions in Proposition~\ref{prop:visual}. Equipped with the metric $\trho$ instead of the given metric $\rho$, $(X,f)$ satisfies the Assumptions in Subsection~\ref{subsect:expandingDyn} with some constants $\xi$ and $\lambda$.
	\end{cor}
	\begin{proof}
		Assumptions~(i) and (ii) are only related to the topology induced by $\trho$, and hence are both satisfied due to the homeomorphism in Proposition~\ref{prop:boundaryCoincide}. Assumption~(iii) can be easily verified by Proposition~\ref{prop:visual} with $\lambda\=e^{a}$ and $\xi\=\xi'$.
	\end{proof}

	The next proposition follows the ideas of \cite[Proposition~3.3.2]{HP09} and \cite[Lemma~8.11]{BM17}. In general, it shows that equipped with the metric $\rho$, tiles are uniformly ``quasi-round'', that is, every tile contains points that are ``deep inside'' the tile.

	\begin{prop} \label{prop:FloydVisualMetric}
		Under the notations and the assumptions in Proposition~\ref{prop:visual}, there is a constant $C_0 > 1$ such that, for each $u\in\Gamma$, there is a point $x\in A_u$ so that
		\begin{equation*}
			B_{\trho} \bigl(x,C_0^{-1} e^{-a\abs{u}}\bigr) \subseteq A_u \subseteq B_{\trho} \bigl(x,C_0 e^{-a\abs{u}}\bigr).
		\end{equation*}
		
	\end{prop}
	\begin{proof}
		By the definition of Markov partitions, for each $A\in\alpha$, $\inter A\neq\emptyset$. Hence, we can choose $\eta(A)\in \inter A$ and $r(A)\in (0,+\infty)$ for each $A\in\alpha$ such that $r(A)<\xi'$ and 
		\begin{equation*}
			B_{\trho}(\eta(A), r(A))\subseteq \inter A.
		\end{equation*}
		Put $C\=1/\min\{r(A):A\in\alpha\}$. Then for each $u\in\Gamma$, since
		$(f^{\abs{u}-1})|_{A_u}\:A_u\to A_{\sigma^{\abs{u}-1} u}$
		is a homeomorphism and $A \= A_{\sigma^{\abs{u}-1} u} \in \alpha$, we can choose $\eta\=((f^{\abs{u}-1})|_{A_u})^{-1} (\eta(A))$ and $r\= e^{-a(\abs{u}-1)} r(A)$. Applying Proposition~\ref{prop:visual} inductively, it is easy to show that $B_{\trho}(\eta,r) = ((f^{\abs{u}-1})|_{A_u})^{-1} (B_{\trho}(\eta(A),r(A))) \subseteq A_u$.

		The remaining part of the proposition follows immediately from Corollary~\ref{cor:visualExpanding} and Lemma~\ref{lem:expandIter}.
	\end{proof}

	To prove the hyperbolicity of the tile graph $\Gamma$, we introduce a concept called \defn{flowers} from \cite[Section~5.6]{BM17}. Generally speaking, for each pair $u,v\in \Gamma$, we aim to compare the $\trho$-distance between $u$ and $v$ with the Gromov product $\Gprod{u,v}{o}$. We construct a flower $W(w)$ at the level of $\Gprod{u,v}{o}$ and then construct a path by connecting $u,w$ and $w,v$ that reaches the $\trho$-distance between $u$ and $v$.

	\begin{definition}
		For $u\in\Gamma$, the \emph{flower} of $u$ is defined as
		\begin{equation}\label{def:flower}
			W(u) = \{v\in\Gamma: |v|=|u|, A_u\cap A_v\neq\emptyset\}.
		\end{equation}
	\end{definition}
	So, it is the set of all the tiles intersecting with $u$ and at the same level as $u$. It follows that $A_{W(u)}$ contains a neighborhood of $A_u$, i.e.,
	\begin{equation}\label{eq:flowerContainTile}
		\inter A_{W(u)} \supseteq A_u.
	\end{equation}

	\begin{lemma}\label{lem:inFlower}
		Under the notations and the assumptions in Proposition~\ref{prop:visual}, there is a constant $r_0 > 1$ such that for each point $x\in X$, if $u\in\Gamma$ is a tile with $x\in A_u$, then
		\begin{equation*}
			B_{\trho}\bigl(x,r_0e^{-a\abs{u}}\bigr) \subseteq A_{W(u)}.
		\end{equation*}
	\end{lemma}
	\begin{proof}
		By Lemma~\ref{lem:expandIter}, there is an integer $N>0$ such that for each $u\in\Gamma$ with $\abs{u}>N$, we have $3\diam_{\trho} A_u<\xi'$. It follows from (\ref{def:flower}) that for each $u\in\Gamma$ with $\abs{u}>N$, 
		\begin{equation}\label{eq:diamFlower}
			\diam_{\trho} A_{W(u)} <\xi'.
		\end{equation}
		For each $u\in\Gamma$ with $\abs{u}\le N$, by (\ref{eq:flowerContainTile}), we can choose $r(u)>0$ such that $\inter A_{W(u)} \supseteq B_{\trho}(A_u,r(u))$.
		Choose $r_0\= \min_{u\in\Gamma, \abs{u}\le N} r(u)e^{a\abs{u}}>0$. Then for each $u\in\Gamma$ with $\abs{u}\le N$, 
		\begin{equation}\label{eq:flowerContainNbhd}
			\inter A_{W(u)} \supseteq B_{\trho}\bigl(A_u,r_0 e^{-a\abs{u}}\bigr).
		\end{equation}
		
		For a general tile $u\in\Gamma$ and a point $x\in A_u$, choose $n\=\max\{\abs{u}-N, 0\}$. By (\ref{eq:diamFlower}), $f^n|_{A_{W(u)}}$ is a homeomorphism to its image $A_{W(\sigma^n u)}$. By (\ref{eq:diamFlower}) and an inductive argument on Proposition~\ref{prop:visual},
		\begin{equation*}
			f^{n} B_{\trho} \bigl( A_u,r_0e^{-a\abs{u}}\bigr) = B_{\trho}\bigl(A_{\sigma^n u},r_0e^{-a(\abs{u}-n)}\bigr) \subseteq \inter A_{W(\sigma^n u)} = f^{n}A_{W(u)}.
		\end{equation*}
		Therefore, $A_{W(u)} \supseteq B_{\trho} \bigl( x,r_0e^{-a\abs{u}} \bigr)$.
	\end{proof}
	
	Recall that the tile graph is equipped with a combinatorial metric $d$ and the Gromov product with respect to the basepoint $o$ is given in (\ref{def:Gprod}) by
	\begin{equation}\label{eq:oGprod}
		\Gprod{u,v}{o} \=\frac{1}{2}(\abs{u} + \abs{v} - d(u,v)).
	\end{equation}
	The next lemma is the key lemma to prove the hyperbolicity of the tile graph.

	\begin{lemma}\label{lem:diamGprod}
		Under the notations and the assumptions in Proposition~\ref{prop:visual},  for all $u,v\in \Gamma$,
		\begin{equation}\label{eq:diamGprod}
			\diam_{\trho} (A_u\cup A_v) \asymp e^{-a\Gprod{u,v}{o}}.
		\end{equation}
	\end{lemma}
	\begin{proof}
		For one side of the inequality, we fix $u,v\in\Gamma$ and assume that $\gamma\:\{0,\dots, n\}\to \Gamma$ is a $d$-geodesic connecting $u$ and $v$. Then $n=d(u,v)$. Choose $k\=\lfloor(n+\abs{u}-\abs{v})/2\rfloor$. Then by (\ref{eq:oGprod}),
		\begin{equation*}
			e^{-a(\abs{u}-k)} \asymp e^{-a\Gprod{u,v}{o}} \asymp e^{-a(\abs{v}-(n-k))}.
		\end{equation*}
		
		By the definition (\ref{def:FloydMetric}) of $\trho$, since $\abs{\gamma(i)}\ge \max\{\abs{u}-i, \abs{v}-(n-i)\}$,
		\begin{align*}
			\trho(u,v) 
			&\le l(\gamma) 
			= \sum_{i\in \{0,\dots, n-1\}} e^{-a\max\{\abs{\gamma(i), \abs(\gamma(i+1))}\}}\\
			&\le \sum_{i=0}^k e^{-a(\abs{u}-i)} + \sum_{i=0}^{n-k-1} e^{-a(\abs{v}-i)}
			\lesssim e^{-a(\abs{u}-k)}+e^{-a(\abs{v}-(n-k-1))}
			\lesssim e^{-a\Gprod{u,v}{o}}.
		\end{align*}
		It is easy to show that for all $\eta\in A_u$,
		$\trho(u,\eta) \le \sum_{i=\abs{u}}^{+\infty} e^{-ai}\lesssim e^{-a\abs{u}}\le e^{-a\Gprod{u,v}{o}}$. 
		It follows that
		\begin{equation}\label{eq:diamLessProd}
			\diam_{\trho} (A_u\cup A_v) \le \trho(u,v)+2\sup_{\eta\in A_u}\trho(u,\eta)+2\sup_{\eta\in A_v}\trho(v,\eta) \lesssim e^{-a\Gprod{u,v}{o}}.
		\end{equation}

		For the other side of the inequality, we choose the constant $r_0>1$ in Lemma~\ref{lem:inFlower}. Fix a pair of vertices $u,v\in\Gamma$ and arbitrarily choose $x\in A_u$ and $y\in A_v$. We choose
		\begin{equation*}
			n\=\min\bigl\{\lfloor-a^{-1} \log(\trho(x,y)/r_0)\rfloor, \abs{u}, \abs{v}\bigr\}.
		\end{equation*}
		Then we have 
		\begin{equation}\label{eq:defLogDist}
			r_0e^{-an}>\trho(x,y).
		\end{equation}
		
		It follows from Lemma~\ref{lem:inFlower} that for each vertex $w\in\Gamma$ with $\abs{w}=n$ and $x\in A_w$, there is a vertex $w'\in\Gamma$ with $\abs{w'}=n$ and $y\in A_{w'}$ such that $A_w\cap A_{w'}\neq\emptyset$. Recall the definition of the edges of the tile graph $\Gamma$. It follows from $x\in A_w\cap A_u$, $y\in A_{w'}\cap A_v$, and $n\le \min\{\abs{u}, \abs{v}\}$ that $d(u,w)=\abs{u}-n$ and $d(v,w')=\abs{v}-n$. Hence, 
		\begin{equation*}
			\Gprod{u,v}{o} = \frac{1}{2}(\abs{u}+\abs{v}-d(u,v)) = \frac{1}{2}(2n+d(u,w)+d(w',v)-d(u,v))\ge \frac{2n-1}{2}.
		\end{equation*}
		Combining the inequality above with (\ref{eq:defLogDist}), we can prove that $\trho(x,y)\lesssim e^{-a\Gprod{u,v}{o}}$. Therefore, the lemma follows.
	\end{proof}

	Now we can give a proof of Proposition~\ref{prop:hyp}.

	\begin{proof}[Proof of Proposition~\ref{prop:hyp}]
		By Lemma~\ref{lem:diamGprod}, for all $u,v,w\in\Gamma$, we have 
		\begin{align*}
			e^{-a\Gprod{u,v}{o}} 
			\asymp \diam_{\trho} (A_u \cup A_v)
			&\le \diam_{\trho} (A_u \cup A_w) + \diam_{\trho} (A_w \cup A_v)\\
			&\asymp e^{-a\Gprod{u,w}{o}}+ e^{-a\Gprod{w,v}{o}}
			\asymp \max\{e^{-a\Gprod{u,w}{o}}+ e^{-a\Gprod{w,v}{o}}\}.
		\end{align*}
		Hence, there is a constant $\delta>0$ such that 
		\begin{equation*}
			\Gprod{u,v}{o} \ge \min\bigl\{\Gprod{u,w}{o}, \Gprod{w,v}{o}\bigr\} - \delta.
		\end{equation*}
		That is, $\Gamma$ is Gromov hyperbolic.

		Since $X$ is Gromov hyperbolic, then, for each $a>0$ small enough, $\partial_{\trho}\Gamma$ coincides with the Gromov boundary of $\Gamma$ and $\trho|_{\partial_{\trho}\Gamma}$ is a visual metric (see for example, \cite[Chapter~4]{BHK01}). Therefore, by the definition of visual metrics, $\trho|_{\partial_{\trho}\Gamma}\asymp\rho$, and, by Proposition~\ref{prop:boundaryCoincide}, the Gromov boundary of $\Gamma$ is naturally homeomorphic to $X$.		
	\end{proof}

	\begin{cor} \label{cor:visualMetric}
		Let $(X,f)$, $\alpha$, $\Gamma$, and $a_0$ be from Proposition~\ref{prop:hyp}. Let $\rho$ be an $a$-visual metric on $X$ for some $0<a<a_0$. Then there is some constant $C_0 > 1$ such that, for all $u,v\in\Gamma$, there is a point $x\in A_u$ such that
		\begin{align*}
			B_{\rho} \bigl(x,C_0^{-1} e^{-a\abs{u}}\bigr) \subseteq A_u \subseteq B_{\rho} &\bigl(x,C_0 e^{-a\abs{u}}\bigr), \\
			C_0^{-1} e^{-a\Gprod{u,v}{o}} \le \diam_{\rho} (A_u\cup A_v) &\le C_0 e^{-a\Gprod{u,v}{o}}. 
		\end{align*}
	\end{cor}
	\begin{proof}
		By Proposition~\ref{prop:hyp}, $\partial_{\trho}\Gamma$ coincides with the Gromov boundary of $\Gamma$ and $\trho|_{\partial_{\trho}\Gamma} \asymp \rho$. Hence, Proposition~\ref{prop:FloydVisualMetric} and Lemma~\ref{lem:diamGprod} can be applied to $\rho$, and the corollary follows immediately.
	\end{proof}

	The following result for visual metrics is a stronger version of Assumption~(iii) in Subsection~\ref{subsect:expandingDyn}. Since we do not care about the exact value of $\xi>0$, we use the same notation $\xi$ as in Subsection~\ref{subsect:expandingDyn}.
	\begin{cor} \label{cor:visual}
		Let $(X,f)$, $\alpha$, $\Gamma$, and $a_0$ be from Proposition~\ref{prop:hyp}. Let $\rho$ be an $a$-visual metric on $X$ for some $0<a<a_0$. Then there is a constant $\xi>0$ such that for all $x,y\in X$ and $n\in\Z_{>0}$ satisfying that $\rho\bigl(f^k x,f^k y \bigr)<\xi$ for each integer $0\le k<n$, we have
		\begin{equation}\label{eq:uniformSimilar}
			\rho(f^n x, f^n y) \asymp e^{an}\rho(x,y).
		\end{equation}
		Moreover, for each $x\in X$, $f|_{B(x,\xi)}$ is a homeomorphism to its image.
	\end{cor}
	\begin{proof}
		By Proposition~\ref{prop:hyp}, $\rho\asymp\trho$. Choose the constant $C=C(\asymp)$ and put $\xi\=\xi'/C$, where $\xi'$ is from Proposition~\ref{prop:visual}. Then (\ref{eq:uniformSimilar}) follows from Proposition~\ref{prop:visual} by an inductive argument. Moreover, (\ref{eq:uniformSimilar}) implies that for all $x,y\in X$ with $\rho(x,y)<\xi$, 
		\begin{equation*}
			\rho(f x, f y) \asymp \rho(x,y).
		\end{equation*}
		Therefore, $f|_{B(x,\xi)}$ is a Lipschitz homeomorphism to its image.
	\end{proof}

	\printbibliography

@book{PU10,
  place={Cambridge},
  series={London Mathematical Society Lecture Note Series}, 
  title={Conformal Fractals: Ergodic Theory Methods},
  publisher={Cambridge University Press},
  author={Przytycki, Feliks and Urbański, Mariusz},
  year={2010},
  collection={London Mathematical Society Lecture Note Series},
  volume={371},
}

@article{Anc87,
  title={Negatively curved manifolds, elliptic operators, and the Martin boundary},
  author={Ancona, Alano},
  journal={Ann.\ of Math.\ (2)},
  year={1987},
  volume={125},
  pages={495-536}
}

@article{Br65,
  title={Invariant sets under iteration of rational functions},
  author={Brolin, Hans},
  journal={Ark.\ Mat.},
  year={1965},
  volume={6},
  pages={103-144},
}

@book{Woe00,
	place={Cambridge},
	series={Cambridge Tracts in Mathematics},
  title={Random Walks on Infinite Graphs and Groups},
	publisher={Cambridge University Press},
	author={Woess, Wolfgang},
	year={2000},
	collection={Cambridge Tracts in Mathematics}
}

@article{Ma84,
  title={The dimension of the maximal measure for a polynomial map},
  author={Anthony Manning},
  journal={Ann.\ of Math.\ (2)},
  year={1984},
  volume={119},
  pages={425}
}

@article{Ly83,
  title={Entropy properties of rational endomorphisms of the Riemann sphere},
  author={Ljubich, M. {Ju.}},
  journal={Ergodic Theory Dynam. Systems},
  year={1983},
  volume={3},
  pages={351-385},
}

@article{Kai97,
author = {Kaimanovich, Vadim A.},
year = {1997},
pages = {57-104},
title = {Ergodicity of harmonic invariant measures for the geodesic flow on hyperbolic spaces},
volume = {1994},
journal = {J.\ Reine Angew.\ Math.},
}

@inbook{Ka96,
  place={Cambridge},
  series={London Mathematical Society Lecture Note Series},
  title={Boundaries of invariant Markov operators: the identification problem},
  booktitle={Ergodic Theory and $\Z^d$ Actions},
  publisher={Cambridge University Press},
  author={Kaimanovich, Vadim A.},
  year={1996},
  pages={127-176},
  collection={London Mathematical Society Lecture Note Series}
}

@article{BHK01,
  title={Uniformizing Gromov hyperbolic spaces},
  author={Mario Bonk and Juha M. Heinonen and Pekka Koskela},
  journal={Ast{\'e}risque},
  year={2001},
  volume = {270},
  pages = {viii+99}
}

@article{Gu80,
 Author = {Guivarc'h, Yves},
 Title = {Sur la loi des grands nombres et le rayon spectral d'une marche al{\'e}atoire},
 Year = {1980},
 pages = {47-98},
 volume = {74},
 Journal = {Ast{\'e}risque},
}

@book{BM17,
  title={Expanding Thurston Maps},
  author={Bonk, Mario and Meyer, Daniel},
	publisher={Amer.\ Math.\ Soc.},
  location={Providence, RI},
  year={2017}
}

@misc{Em18,
      title={Brownian motion, random walks on trees, and harmonic measure on polynomial Julia sets}, 
      author={Nathaniel D. Emerson},
      year={2006},
      eprint={math/0609044},
      archivePrefix={arXiv},
      primaryClass={math.DS}
}

@article{Fl80,
author = {Floyd, William J.},
journal = {Invent. Math.},
pages = {205-218},
title = {Group completions and limit sets of Kleinian groups.},
volume = {57},
year = {1980},
}

@book{BH99,
  title={Metric Spaces of Non-Positive Curvature},
  author={Martin R. Bridson and Andr'e Haefliger},
  publisher={Springer Berlin, Heidelberg},
  year={1999},
}

@article{HP09,
    AUTHOR = {Ha\"{\i}ssinsky, Peter and Pilgrim, Kevin M.},
     TITLE = {Coarse expanding conformal dynamics},
   JOURNAL = {Ast\'{e}risque},
  FJOURNAL = {Ast\'{e}risque},
    volume = {325},
      YEAR = {2009},
     PAGES = {viii+139},
      ISSN = {0303-1179,2492-5926},
      ISBN = {978-2-85629-266-2},
   MRCLASS = {37F10 (30D05 37F30)},
  MRNUMBER = {2662902},
MRREVIEWER = {Volker\ Mayer},
}

@article{BHM11,
  author = {Blachère, Sébastien and Haïssinsky, Peter and Mathieu, Pierre},
  journal = {Ann.\ Sci.\ \'{E}c.\ Norm.\ Sup\'{e}r.},
  number = {4},
  pages = {683-721},
  publisher = {Société mathématique de France},
  title = {Harmonic measures versus quasiconformal measures for hyperbolic groups},
  volume = {44},
  year = {2011}
}

@article{BHM08,
 author = { Blachère, Sébastien and Haïssinsky, Peter and Mathieu, Pierre},
 journal = {Ann. Probab.},
 number = {3},
 pages = {1134--1152},
 publisher = {Institute of Mathematical Statistics},
 title = {Asymptotic Entropy and Green Speed for Random Walks on Countable Groups},
 volume = {36},
 year = {2008}
}

@article{BMS03,
	author = {Binder, I. and Makarov, N. and Smirnov, S.},
	journal = {Duke Math. J.},
	number = {2},
	pages = {343--365},
	title = {Harmonic measure and polynomial Julia sets},
	volume = {117},
	year = {2003}
}

@article{Dy69,
year = {1969},
author = {Dynkin, E.~B.},
title = {Boundary theory of Markov processes (The Discrete Case)},
journal = {Russian Math.\ Surveys},
volume = {24},
}

\end{document}